\documentclass[11pt]{article}
\usepackage{color}

\usepackage{amssymb}   
\usepackage{amsthm}    
\usepackage{amsmath}   
\usepackage{stmaryrd}  
\usepackage{titletoc}  
\usepackage{mathrsfs}  
\usepackage{graphicx}
\usepackage{enumitem}
\newlength{\defbaselineskip}
\newcommand{\setlinespacing}[1]%
           {\setlength{\baselineskip}{#1 \defbaselineskip}}

\theoremstyle{plain}
\newtheorem{thm}{Theorem}[section]

\newtheorem{lem}[thm]{Lemma}
\newtheorem{prop}[thm]{Proposition}
\newtheorem{exam}[thm]{Example}

\theoremstyle{definition}
\newtheorem{defn}{Definition}[section]

\newtheorem{rmk}{Remark}[section]

\newcommand{\eps}{\varepsilon}

\DeclareMathOperator*{\esssup}{esssup}
\DeclareMathOperator*{\essinf}{essinf}

\newcommand{\cO}{\mathcal{O}}
\newcommand{\cH}{\mathcal{H}}

\newcommand{\cL}{\mathcal{L}}
\newcommand{\cT}{\mathcal{T}}

\newcommand{\cA}{\mathcal{A}}
\newcommand{\cS}{\mathcal{S}}

\newcommand{\cG}{\mathcal{G}}

\newcommand{\cU}{\mathcal{U}}

\newcommand{\bH}{\mathbb{H}}
\newcommand{\bV}{\mathbb{V}}

\newcommand{\bP}{\mathbb{P}}
\newcommand{\bB}{\mathbb{B}}
\newcommand{\bR}{\mathbb{R}}
\newcommand{\bN}{\mathbb{N}}

\newcommand{\sF}{\mathscr{F}}

\makeatletter\@addtoreset{equation}{section} \makeatother
 \allowdisplaybreaks
\begin{document}

\title{Optimal control of infinite-dimensional differential systems with randomness and path-dependence and stochastic path-dependent Hamilton-Jacobi equations\footnotemark[1] 
}

\author{Jinniao Qiu\footnotemark[2]  \and Yang Yang\footnotemark[2]
}
\footnotetext[1]{This work was partially supported by the National Science and Engineering Research Council of Canada (NSERC). Yang was partially supported by a graduate scholarship through the NSERC-CREATE Program on Machine Learning in quantitative Finance and Business Analytics (Fin-ML CREATE). The authors also acknowledge the support of the Banff International Research Station (BIRS) for the Focused Research Group [22frg198] ``Novel perspectives in kinetic equations for emerging phenomena", July 17-24, 2022, where part of this work was done. In addition, part of the work was revised during Qiu's visit to Universit\'e Paris Dauphine partially supported by the 2022 PIMS-Europe Fellowship and both the hospitality and supports are gratefully acknowledged.}
\footnotetext[2]{Department of Mathematics \& Statistics, University of Calgary, 2500 University Drive NW, Calgary, AB T2N 1N4, Canada. \textit{E-mail}: \texttt{jinniao.qiu@ucalgary.ca} (J. Qiu), \texttt{yang.yang1@ucalgary.ca} (Y. Yang)
.}

\maketitle

\begin{abstract}
This paper is devoted to the stochastic optimal control problem of infinite-dimensional differential systems allowing for both path-dependence and measurable randomness. As opposed to the deterministic path-dependent cases studied by Bayraktar and Keller [J. Funct. Anal. 275 (2018), 2096--2161], the value function turns out to be a random field on the path space and it is characterized by a stochastic path-dependent Hamilton-Jacobi (SPHJ) equation. A notion of viscosity solution is proposed and the value function is proved to be the unique viscosity solution to the associated SPHJ equation.
\end{abstract}



\section{Introduction}
The purpose of this paper is to characterize the value function of a certain class of path-dependent stochastic optimal control problems in an infinite dimensional setting as the unique solution to the corresponding stochastic path-dependent Hamilton-Jacobi (SPHJ) equation. These equations in finite dimensional set up have been studied by Qiu (\cite{qiu2022controlled}). However, our extension to the infinite dimensional spaces is nontrivial.

Let $(\Omega,\sF,\{\sF_t\}_{t\geq0},\bP)$ be a complete filtered probability space. The filtration $\{\sF_t\}_{t\geq0}$ satisfies the usual conditions and is generated by an $m$-dimensional Wiener process $W=\{W(t):t\in[0,\infty)\}$ together with all the $\bP$-null sets in $\sF$. Let $\bV\subseteq \bH\subseteq \bV^*$ be a Gelfand triple, where $\bV$ is a separable reflexive Banach space with a continuous, dense, and compact embedding into a separable Hilbert space $\bH$.

Throughout this work,  the number $T\in (0,\infty)$ denotes a fixed deterministic terminal time and the set $C([0,T];\bH)$ represents the space of $\bH$-valued continuous functions on $[0,T]$.  For each $x\in C([0,T];\bH)$, denote by $x_t$ its restriction to time interval $[0,t]$ for each $t\in [0,T]$ and by $x(t)$ its value at time $t\in[0,T]$. 

Consider the following stochastic optimal control problem
\begin{align}
\min_{\theta\in\cU}E\left[\int_0^T\!\! f(s,X_s,\theta(s))\,ds +G(X_T) \right], \label{Control-probm}
\end{align}
subject to
\begin{equation}\label{state-proces-contrl}
\left\{
\begin{split}
&\frac{dX(t)}{dt}=AX(t)+\beta(t,X_t,\theta(t)),\,\,
\,t\geq 0; \\
& X_0=x_0\in \bH.
\end{split}
\right.
\end{equation}
Here, we denote by $\cU$ the set of all the $U$-valued and $\{\sF_t\}_{t\in [0,T]}$-adapted processes with $U$ being a nonempty compact set. The process $(X(t))_{t\in[0,T]}$ is the {\sl state process}, governed by the {\sl control} $\theta\in\cU$. The notation $X^{r,x_r;\theta}(t)$ for $0\leq r\leq t\leq T$ may be used to indicate the dependence of the state process on the control $\theta$, the initial time $r$, and initial path $x_r$. Here $A: \bV\to \bV^*$ is a linear time-constant operator\footnote{Under certain conditions, our results may be extended to operators $A$ that are time and path-dependent, and even nonlinear. To avoid cumbersome arguments, we  consider a time-constant linear operator $A$ herein.} and $\beta(t,X_t,\theta(t))$ takes values in $\bH$. For the well-posedness of the state system, we apply the following assumptions on operator $A$.
\begin{enumerate}[label=(\roman*)]
\item (Coercivity) There exists $c_1\in\bR$, $c_2\in\bR^+$ such that for all $v\in \bV$,
\begin{equation*}
2\cdot {_{\bV^*}\langle} Av,v \rangle_{\bV} \le c_1 \| v \|_{\bH}^2 - c_2 \| v \|_{\bV}^2.
\end{equation*}
\item (Boundedness) There exists $c_3 \ge 0$ such that for all $v\in \bV$,
\begin{equation*}
\|Av\|_{\bV^*} \le c_3 \|v\|_{\bV}.
\end{equation*}
\end{enumerate}


In this paper, we consider the non-Markovian cases where the coefficients $\beta,f$ may depend not only on time and control but also \textit{explicitly} on $\omega \in\Omega$ and  the path/history of the state process. The function $G$ is random and path-dependent as well. Such problems arise naturally from controlled partial differential equations allowing for path-dependent and random coefficients. An example is sketched as follows.

\begin{exam}\label{example}
Let $\mathcal{O}\subseteq \bR^d$ be a bounded domain with smooth boundary $\partial\mathcal{O}$. Denote by $W_0^{k,p}(\mathcal{O})$ the $k$-th order Sobolev spaces on $\mathcal{O}$ with elements vanishing at $\partial\mathcal{O}$, for $k\in\mathbb{Z}$, $p\in (1,\infty)$. We consider the following control problem
\begin{align*}
\min_{\theta\in\cU}E\left[\int_0^T\!\!  \tilde{f}(s, \tilde{X}_s,\theta(s))\,ds + \tilde{G}(\tilde{X}_T) \right]\,, \label{Control-probm-exam2}
\end{align*}
subject to
\begin{equation*}\label{state-proces-contrl-exam2}
 { \tilde{X}(t)}=   x_0 + \int_0^t (\Delta \tilde{X}(s)+ \tilde{\beta}(s, \tilde{X}_s,\theta(s)))\,ds + \tilde{\eta}(t)\,\,
\,t\geq 0,  x_0\in \bH,
\end{equation*}
where $\bV = W_0^{1,2}(\mathcal{O})$ is dense and compactly embedded into $\bH = L_0^2(\mathcal{O})$, and $\bV^*=W_0^{-1,2}(\mathcal{O})$. $\tilde{f}$ and $\tilde{G}$ are functions taking values in $\mathbb{R}$. $\tilde{\beta}$ is taking values in $\bH$, while $\tilde{\eta}$ may be any $(\sF_t)_{t\ge 0}$-adapted $W_0^{2,2}(\mathcal{O})$-valued stochastic process with or without rough paths, including integrals with respect to Wiener processes, fractional Brownian motions, and general semimartingales, and so on. And the Laplace operator $\Delta$ is linearly mapping $W_0^{1,2}(\mathcal{O})$ onto $W_0^{-1,2}(\mathcal{O})$. Set $X(t) = \tilde{X}(t) - \tilde{\eta}(t)$, for any $t\in[0,T]$. The control problem above could be written equivalently as \eqref{Control-probm}-\eqref{state-proces-contrl}, while the associated coefficients $(f,\beta)(s, X_s, \theta(s)) = (\tilde{f}, \tilde{\beta})(s, (X+\tilde{\eta})_s, \theta(s))$ and $G(X_T) = \tilde{G}((X+\tilde{\eta})_T)$ are obviously random.

\end{exam}

\text{ }

Back to the control problem (1.1) - (1.2), we define the dynamic cost functional:
\begin{equation}\label{cost}
J(t, x_t; \theta) = E_{\sF_t} \left[  \int_t^T \!f(s, X_s^{t, x_t; \theta}, \theta(s))ds + G(X_T^{t, x_t; \theta})  \right], \text{ } t\in[0, T],
\end{equation}
 and the value function $V$ is given by
 \begin{equation}\label{value}
 V(t,x_t) = \essinf_{\theta\in \mathcal{U}} J(t, x_t; \theta), \text{ } t\in[0, T].
 \end{equation}
 
 \text{ }
 
Due to the randomness and path-dependence of the coefficients, the value function $V(t, x_t)$ is a function of time $t$, path $x_t$, and $\omega\in\Omega$. In this work, it is in fact proven to be the unique viscosity solution to the following stochastic path-dependent Hamilton-Jacobi (SPHJ) equation:
\begin{equation}\label{SPHJ}
\left\{
\begin{split}
- \mathfrak{d}_t u(t, x_t) - \mathcal{H}(t, x_t, \nabla u(t, x_t)) &= 0, \text{ }\text{ } (t, x_t)\in [0, T)\times C([0, T]; \bH), 
\\
u(T, x_T) &= G(x_T), \text{ }\text{ } x_T\in C([0, T]; \bH), 
\end{split}
\right.
\end{equation}
with
\begin{equation}\label{Hamilton}
\mathcal{H}(t, x_t, p) = \essinf_{v\in U} \{  {_{\bV^*}\langle}Ax(t),p\rangle_{\bV} + {_{\bV^*}\langle}\beta(t, x_t, v),p\rangle_{\bV} + f(t, x_t, v)  \}, \text{ for } p\in \bV,
\end{equation}
where $\nabla u(t, x_t)$ denotes the vertical derivative of $u(t, x_t)$ at the path $x_t$ (see Definition \ref{def_deri}) and the unknown adapted random field $u$ is confined to the following form:
\begin{equation}\label{1.12}
u(t, x_t) = u(T, x_{t, T-t}) - \int_t^T \! \mathfrak{d}_s u(s, x_{t, s-t})ds -\int_t^T \! \mathfrak{d}_{\omega} u(s, x_{t, s-t})dW(s),
\end{equation}
where $x_{t, r-t}(s) = x_t(s)1_{[0, t)}(s) + x_t(t)1_{[t, r]}(s)$ for $0\le t\le r\le T$, $0\le s\le r$. The semimartingale decomposition theorem indicates the uniqueness so that the pair $(\mathfrak{d}_t, \mathfrak{d}_{\omega})$ is well-defined as two linear operators; in finite dimensional cases, the operators are consistent with those differential operators defined in \cite[Section 5.2]{cont2013functional} and \cite[Theorem 4.3]{leao2018weak} for instance. Comparing \eqref{SPHJ} - \eqref{1.12}, we can rewrite the SPHJ equation formally as a path-dependent backward stochastic partial differential equation (BSPDE):
\begin{equation}\label{BSPDE}
\left\{
\begin{split}
- du(t, x_t) &= \mathcal{H}(t,x _t, \nabla u(t, x_t))dt - \psi(t, x_t)dW(t), \text{ }\text{ } (t, x_t)\in [0, T)\times C([0, T]; \bH)
\\
u(T, x_T) &= G(x_T), \text{ }\text{ } x_T\in C([0, T]; \bH)
\end{split}
\right.
\end{equation}
 where the pair $(u, \psi) = (u, \mathfrak{d}_{\omega} u)$ is unknown.
 \\
\\
The notion of viscosity solutions for deterministic partial differential equations   can be traced back to early works (see \cite{crandall1984some,crandall1983viscosity,ishii1990viscosity,jensen1988uniqueness,lions1983optimal} for the finite dimensional cases, and \cite{crandall1985hamilton,lions1983optimal,lions1988viscosity} for the infinite dimensional cases, to name a few). The optimal control problem with deterministic path-dependent coefficients is first studied in \cite{lukoyanov2007viscosity} with a viscosity solution approach under the finite dimensional setup. The Hamiltonian is non-anticipatory and the existence and uniqueness theorems are proved. Other cases with similar setups can be found in \cite{ekren2016viscosity1,ekren2016pseudo,ren2017comparison}; for theories of general deterministic path-dependent PDEs, please refer to \cite{cont2013functional,cosso2018path,peng2016bsde,ren2017comparison} to mention but a few. When it comes to the extension of optimal control problems under an infinite dimensional framework, a wellposedness result is provided in \cite{mete1988hamilton} for the viscosity solution to the HJB equations in Banach spaces with state-dependent, deterministic coefficients and controls; another study with path-dependent setup is provided by \cite{bayraktar2018path} for a class of fully nonlinear path-dependent PDEs with nonlinear, monotone and coercive operators in Hilbert spaces. The work is focused on the wellposedness and stability of minimax solutions while some discussion of the viscosity solution approach is included as well. Under the finite dimensional framework, when the coefficients are state-dependent and possibly random (see \cite{bender2016first,qiu2018viscosity,qiu2019uniqueness}), the value function is proved to be the solution to a backward stochastic partial differential equation (BSPDE); for more research work in general BSPDEs, please see \cite{Bayraktar-Qiu_2017,cardaliaguet2019master,hu2002semi,peng1992stochastic} among many others. When the coefficients are both possibly random and path-dependent, a class of optimal control problems have been studied in \cite{qiu2022controlled} under the finite dimensional framework. 
Our work is a nontrivial extension of \cite{qiu2022controlled} in the infinite dimensional setup. Inspired by early works of viscosity solution in \cite{bayraktar2018path,fabbri2017stochastic,prevot2007concise}, we use a Gelfand triple $\bV\subset \bH \subset \bV^*$, where the continuous and compact embedding is set as usual and helps to deal with the lack of local compactness issue in the path space caused by both the path-dependence and the infinite dimensional spaces the state process takes values in. Meanwhile the dense and compact embedding argument, along with the bounded and coercive assumptions of operator $A$ helps us handle the possible unboundedness in $A:\bV\to\bV$ or $\bV\to\bH$ as well as the local compactness issue. It is worthwhile to point out that in \cite{fabbri2017stochastic}, the viscosity solution approach is discussed under one or more continuously and densely embedded Gelfand triples while both $\bV$ and $\bH$ are Hilbert spaces. In contrast to our setup, the state dynamics is deterministic and state-dependent, with an extra diffusion term. And the linear operator $A$ is possibly depending on both time and the admissible control. In general, our theory makes it possible to employ the dynamic programming and viscosity solution approach to study optimal control problems for a fairly general class of stochastic path-dependent Hamilton-Jacobi equations in infinite dimensional spaces. In this setting, the control problems \eqref{Control-probm}, \eqref{state-proces-contrl} and the associated path-dependent stochastic Hamilton-Jacobi equation \eqref{SPHJ}, to the best of our knowledge, have never been studied in the literature. 
 
The main obstacles in our paper are three-folded. First the operator $A:\bV\to \bV^*$ and the path-dependence make it hard to directly obtain an appropriate \emph{uniform-in-time}  stability estimation of the state processes in space $\bV$ or $\bH$, because for a.e. $t\in [0,T]$, $X(t)\in \bV$, $AX(t)$ takes value in $\bV^*$. It is thus difficult to achieve uniform convergence of its own finite dimensional projections. Although eventually we manage to resolve such issues in the uniqueness discussion of the viscosity solution, the unboundedness of $A: \bV\to\bV$ or $\bV\to\bH$, which leads to weaker estimations of path distance within the Gelfand triple as shown in Lemma \ref{lemma3.1}, further causes trouble in ensuring uniform convergence of the finite-dimensional approximations of other coefficients $f$, $\beta$, and $G$ (see \eqref{f_uniform}). For the same reason, we need to introduce further assumptions to ensure the existence of viscosity subsolution. This is different from the result under a finite dimensional framework. Second, the coefficients and thus the solution $u$ are random and path-dependent. The lack of local compactness in the path space, which is caused not just by the infinite dimensional setup for the state process but also by the path-dependence of the coefficients, drives us to define the random test functional spaces as sequences of compact subspaces using compact embedding and stopping time theories. A similar technique is adopted as in \cite{ekren2016viscosity1,ekren2016pseudo} to replace the pointwise extremality in the standard definition of viscosity solutions by the corresponding extremality in the context of optimal stopping problems. However, in contrast to the finite dimensional case \cite{qiu2022controlled}, not only the definition of viscosity solutions needs to be changed accordingly, but we can only obtain an even weaker version of comparison principle that is associated with infinite sequences of integers $\{k_n\}_{n\in\bN^+}$ and has an explicit initial state dependence. This is different from the deterministic nonlinear path-dependent cases (for instance \cite{ekren2016viscosity1,lukoyanov2007viscosity}) and the stochastic state-dependent cases (for instance \cite{qiu2018viscosity,qiu2019uniqueness}). 
Third, we did not define any topology in the measurable space $(\Omega,\sF)$. As a result, it is inappropriate to express our coefficients, test functionals as well as viscosity solutions pointwisely w.r.t. $\omega\in\Omega$, although they may be explicitly dependent on it. Instead, both the test functions and viscosity solutions are parametrized by $\Omega_{\tau}\in\sF_{\tau}$ for each point $(\tau,\xi)$, where $\tau$ may be a stopping time and $\xi$ valued in $ C([0,\tau];\bV)$. The combination of path-dependence and measurable randomness prompts us to take a different approach via dense and compact embedding arguments, as well as a weak version of comparison principle, other than the conventional variable-doubling techniques, in solving our stochastic path-dependent Hamilton-Jacobian equations.  

Our paper is organized as follows. In Section 2, we show some preliminary notations and main assumptions. Then we bring in certain regularity conditions to our test functional spaces, under which we are able to introduce the definition of the viscosity solutions. In Section 3, we first introduce the well-posedness results on the state process and some of its important estimates. Then we move on to the regularity analysis of the cost functional and the value function. Applying stopping time techniques, a generalized dynamic programming principle is proved,  and a generalized It\^o-Kunita formula is given under the path-dependent setup. In the end of this section, we prove that the value function is a viscosity solution to the associated SPHJ equation. In Section 4, a weak version of comparison principle is proved followed by some stronger assumptions and a finite dimensional approximation lemma. At last, we are able to prove the uniqueness of the viscosity solution to the SPHJ equation using finite dimensional approximations. In the appendix, one may find the proofs of Lemma 3.2 and Proposition 3.3.


\section{Preliminaries and definition of viscosity solution}
\subsection{Preliminaries}
Let $\bB$ be a Banach space equipped with norm $\|\cdot\|_{\bB}$. For each $r\in[0,T]$, let space $\Lambda_r^0(\bB):=C([0,r];\bB)$ be the set of all $\bB$-valued continuous functions on $[0,r]$ and $\Lambda_r(\bB):=D([0,r];\bB)$ the space of $\bB$-valued c\`adl\`ag (right continuous with left limits) functions on $[0,r]$. Define
\begin{align*}
\Lambda^0(\bB)=\cup_{r\in[0,T]}\Lambda_r^0(\bB), \Lambda(\bB)=\cup_{r\in[0,T]}\Lambda_r(\bB).
\end{align*}
Throughout this paper, for each path $X\in\Lambda_T(\bB)$ and $t\in[0,T]$, denote by $X_t=(X(s))\textsubscript{$0\leq s\leq t$}$ its restriction to time interval $[0,t]$, while using $X(t)$ to represent its value at time $t$. For each $(x_r,\overline{\rm x}_t)\in\Lambda_r(\bB)\times\Lambda_t(\bB)$ with $0\leq r\leq t\leq T$,
\begin{equation*}
\begin{split}
\|x_r\|_{0,\bB}=\sup_{s\in[0,r]}\|x_r(s)\|_{\mathbb{B}};\\
d_{0,\bB}(x_r,\overline{\rm x}_t)=\sqrt{|t-r|}+\sup_{s\in[0,t]}\{\|x_r(s)-\overline{\rm x}_t(s)\|_{\mathbb{B}}1_{[0,r)}(s)+\|x_r(r)-\overline{\rm x}_t(s)\|_{\mathbb{B}}1_{[r,t]}(s)\}.
\end{split}
\end{equation*}
Here for each $t\in[0,T]$, $(\Lambda_t^0(\bB),\|\cdot\|_{0,\mathbb{B}})$ and $(\Lambda_t(\bB),\|\cdot\|_{0,\mathbb{B}})$ are Banach spaces, while $(\Lambda_t^0(\bB),d_{0,\bB})$ and $(\Lambda_t(\bB),d_{0,\bB})$ are complete metric spaces. For each $x_t\in\Lambda_t(\bB)$ (respectively for $x_t\in\Lambda_t^0(\bB)$), we may define, correspondingly, $\overline{\rm x}\in\Lambda_T(\bB)$ (respectively for $\overline{\rm x}\in\Lambda_T^0(\bB)$) with $\overline{\rm x}(s)=x_t(t\land s)$ for $s\in[0,T]$. In addition, we shall use $\mathcal{B}(\Lambda^0(\bB))$, $\mathcal{B}(\Lambda(\bB))$, $\mathcal{B}(\Lambda_t^0(\bB))$, $\mathcal{B}(\Lambda_t(\bB))$ to denote the associated Borel $\sigma$-algebras.

For each $x_t\in\Lambda_t(\bV^*)$ and any $h\in \bV^*$, we denote its vertical perturbation $x_t^h\in\Lambda_t(\bV^*)$ with $x_t^h(s)=x_t(s)1_{[0,t)}(s)+(x_t(s)+h)1_{t}(s)$ for $s\in[0,t]$. Note that this vertical perturbation may not be time continuous at the end point.

\begin{defn}\label{def_deri}
Given a functional $\phi :\Lambda(\bV^*)\to\mathbb{R}$ and $x_t\in\Lambda_t(\bV^*)$, $\phi$ is said to be vertically differentiable at $x_t$ if the function
\begin{align*}
\phi(x_t^{\cdot}): \bV^*&\to\mathbb{R}
\\
h&\to\phi(x_t^h)
\end{align*}
is differentiable at 0 in a Gateaux derivative sense. The gradient $\nabla\phi(x_t)$ is $\bV$-valued and defined by
\begin{equation*}
{_{\bV}\langle}  \nabla\phi(x_t),h  \rangle_{\bV^*} = \lim_{\lambda\to 0}\frac{\phi(x_t^{\lambda h})-\phi(x_t)}{\lambda}
\end{equation*}
for all $h\in \bV^*$.
\end{defn}

Let $\bB'$ be another Banach space equipped with norm $\|\cdot\|_{\bB'}$. The continuity of functionals on metric spaces $\Lambda^0(\bB)$ and $\Lambda(\bB)$ can be defined in a standard way. Given $x_t\in\Lambda(\bB)$, we say a mapping $\phi: \Lambda(\bB)\to\mathbb{B}'$ is continuous at $x_t$ if for any $\epsilon > 0$ there exists some $\delta >0$ such that for any $x_r\in\Lambda(\bB)$ satisfying $d_{0,\bB}(x_r, x_t)<\delta$, it holds that $\|\phi(x_r)-\phi(x_t)\|_{\mathbb{B}'}<\epsilon$. If the $\mathbb{B}'$-valued functional $\phi$ is continuous and bounded at all $x_t\in\Lambda(\bB)$, $\phi$ is said to be continuous on $\Lambda(\bB)$ and denoted by $\phi\in C(\Lambda(\bB) ; \mathbb{B}')$. Similarly we define $C(\Lambda^0(\bB) ;\mathbb{B}')$, $C([0,T]\times\Lambda(\bB) ; \mathbb{B}')$, and $C([0,T]\times\Lambda^0(\bB) ; \mathbb{B}')$.

Throughout this paper, as usual the measurability of non-separable space-valued random functions is in a strong sense, i.e., such measurable functions may be approximated point-wisely (a.e. if a measure is given) by simple functions.  For each $t\in [0,T]$, denote by $L^0(\Omega\times\Lambda_t(\bB), \sF_t\otimes\mathcal{B}(\Lambda_t(\bB)); \mathbb{B}')$,  the space of $\mathbb{B}'$-valued  $\sF_t\otimes\mathcal{B}(\Lambda_t(\bB))$-measurable random variables. For each
measurable function
\begin{equation*}
u: (\Omega\times[0,T]\times\Lambda(\bB), \sF\otimes\mathcal{B}([0,T])\otimes\mathcal{B}(\Lambda(\bB)))\to (\mathbb{B}',\mathcal{B}(\mathbb{B}')),
\end{equation*}
we say $u$ is $adapted$ if for any time $t\in [0,T]$, $u$ is $\sF_t\otimes\mathcal{B}(\Lambda_t(\bB))$-measurable. For $p\in [1,\infty]$, denote by $\mathcal{S}^p(\Lambda(\bB); \mathbb{B}')$ the set of the adapted functions $u: \Omega\times [0,T]\times\Lambda(\bB)\to \mathbb{B}'$ such that for almost all $\omega\in\Omega$, $u$ is valued in $C([0,T]\times\Lambda(\bB); \mathbb{B}')$ and
\begin{equation*}
\| u \|_{\mathcal{S}^p(\Lambda(\bB); \mathbb{B}')}:=\left\| \sup_{(t,x_t)\in [0,T]\times\Lambda_t(\bB)} \| u(t,x_t) \|_{\mathbb{B}'} \right\|_{L^p(\Omega, \sF, \mathbb{P})}<\infty.
\end{equation*}
For $p\in [1,\infty)$, denote by $\mathcal{L}^p(\Lambda(\bB); \mathbb{B}')$ the totality of all the adapted functions $\mathcal{X}: \Omega\times [0,T]\times\Lambda(\bB)\to\mathbb{B}'$ such that for almost all $(\omega,t)\in\Omega\times [0,T]$, $\mathcal{X}$ is valued in $C(\Lambda(\bB); \mathbb{B}')$, and
\begin{equation*}
\|\mathcal{X}\|_{\mathcal{L}^p(\Lambda(\bB); \mathbb{B}')}:=\left\| \left(\int_0^T \! \sup_{x_t\in\Lambda_t(\bB)} \| \mathcal{X}(t,x_t) \|^p_{\mathbb{B}'}dt\right)^{1/p} \right\|_{L^p(\Omega, \sF, \mathbb{P})}<\infty.
\end{equation*}
Obviously $(\mathcal{S}^p(\Lambda(\bB); \mathbb{B}'), \| \cdot \|_{\mathcal{S}^p(\Lambda(\bB); \mathbb{B}')})$ and $(\mathcal{L}^p(\Lambda(\bB); \mathbb{B}'), \| \cdot \|_{\mathcal{L}^p(\Lambda(\bB); \mathbb{B}')})$ are Banach spaces. In a similar way, we define spaces $L^0(\Omega\times\Lambda_t^0(\bB), \sF_t\otimes\mathcal{B}(\Lambda_t^0(\bB)); \mathbb{B}')$, $(\mathcal{S}^p(\Lambda^0(\bB); \mathbb{B}'), \| \cdot \|_{\mathcal{S}^p(\Lambda^0(\bB); \mathbb{B}')})$, and $(\mathcal{L}^p(\Lambda^0(\bB); \mathbb{B}'), \| \cdot \|_{\mathcal{L}^p(\Lambda^0(\bB); \mathbb{B}')})$.

For each $\delta >0$, $0\le\tau\le t\le T$, Banach space $\bB\subset\bB'$, and $\xi\in\Lambda_{\tau}^0(\bB)$, define the neighbourhood of $\xi$ in the space $\Lambda_t^0(\bB')$ as $B^{\bB'}_{\delta}(\xi)$ being the family of $x\in \Lambda^0_{t}(\bB')$ satisfying
\begin{equation*}
\sup_{r\in[0,t]}\|x(r)-\xi(r\land \tau)\|_{\bB'} < \delta.
\end{equation*}

Within the Gelfand triple, there exists $c>0$ such that
\begin{equation*}
\|\cdot\|_{\bV^*} \le c\|\cdot\|_{\bH} \le c^2\|\cdot\|_{\bV},
\end{equation*}
i.e., for all $h\in\bV$,
\begin{equation*}
\|h\|_{\bV^*} \le c\|h\|_{\bH} \le c^2\|h\|_{\bV}.
\end{equation*}
W.l.o.g, we assume $c=1$. Following is the assumption we use throughout this paper.

($\mathcal{A}1$) $G\in L^{\infty}(\Omega; \sF_T; C(\Lambda_T(\bV^*); \mathbb{R}))$. 
\begin{enumerate}[label=(\roman*)]
\item for each $v\in \cU$, $f(\cdot, \cdot, v)$, $\beta(\cdot, \cdot, v)$ is adapted;
\item for almost all $(\omega, t)\in\Omega\times [0,T]$, $f(t, \cdot, \cdot)$, $\beta(t, \cdot, \cdot)$ is continuous on $\Lambda_t(\bV^*)\times U$;
\item there exists $L>0$ such that for all $x_T,\overline{x}_T\in \Lambda_T(\bH)$, and $t\in [0,T]$, there hold
\begin{equation*}
\esssup_{\omega\in\Omega} | G(x_T) |+\esssup_{\omega\in\Omega} \sup_{v\in U} | f(t, x_t, v) |+\esssup_{\omega\in\Omega} \sup_{v\in U} \| \beta(t, x_t, v)  \|_{\bH} \le L,
\end{equation*}
\begin{align*}
&\esssup_{\omega\in\Omega} | G(x_T)-G(\overline{x}_T) | + \esssup_{\omega\in\Omega} \sup_{v\in U} | f(t, x_t, v)-f(t, \overline{x}_t, v) | 
\\
&+\esssup_{\omega\in\Omega} \sup_{v\in U} \| \beta(t, x_t, v)-\beta(t, \overline{x}_t, v) \|_{\bV^*} \le L(\| x_T-\overline{x}_T \|_{0,\bH} + \|x_t - \overline{x}_t\|_{0,\bH}).
\end{align*}
\end{enumerate}



\subsection{Definition of viscosity solutions}
For $\delta\ge0$, $x_t\in\Lambda_t(\bV^*)$, we define the horizontal extension $x_{t,\delta}\in\Lambda_{t+\delta}(\bV^*)$ by setting $x_{t,\delta}(s)=x_t(s\land t)$ for all $s\in [0,t+\delta]$.

\begin{defn}\label{cf1}
For $u\in\mathcal{S}^2(\Lambda(\bV^*); \mathbb{R})$ with $\nabla u\in\mathcal{L}^2(\Lambda(\bV^*); \bV)$, we say $u\in\mathcal{C}^1_{\sF}$ if
\begin{enumerate}[label=(\roman*)]
\item there exists $(\mathfrak{d}_t u, \mathfrak{d}_{\omega} u)\in\mathcal{L}^2(\Lambda(\bV^*); \mathbb{R})\times\mathcal{L}^2(\Lambda(\bV^*); \mathbb{R}^m)$ such that for all $0\le r\le \tau\le T$, $x_r\in\Lambda_r(\bV^*)$
\begin{equation*}
u(\tau,x_{r,\tau -r})=u(r,x_r)+\int_r^{\tau} \! \mathfrak{d}_s u(s, x_{r,s-r})ds +\int_r^{\tau} \! \mathfrak{d}_{\omega} u(s, x_{r,s-r})dW(s), a.s.;
\end{equation*}
\item there exists a constant $\rho\in (0,\infty)$ such that for almost all $(\omega,t)\in\Omega\times [0,T]$ and all $x_t\in\Lambda_t^0(\bV^*)$, there holds $\| \nabla u(t,x_t) \|_{\bV} \le\rho$;
\item there exists a constant $\alpha\in (0,1)$ and a finite partition $0=\underline{t}_0<\underline{t}_1<...<\underline{t}_n=T$, for integer $n\ge 1$, such that $\nabla u$ is a.s. valued in $C((\underline{t}_j, \underline{t}_{j+1})\times\Lambda(\bV^*); \bV)$ for $j=0, ..., n-1$, and for any $0<\delta<\min_{0\le j\le {n-1}} | \underline{t}_{j+1}-\underline{t}_j |$, there exists $L_{\alpha}^{\delta}\in (0,\infty)$ satisfying  a.s for almost all $t\in\cup_{0\le j\le n-1} (\underline{t}_j, \underline{t}_{j+1}-\delta]$ and all $x_t, y_t\in\Lambda_t^0(\bV^*)$,
\begin{align*}
|  u(t, x_t) - u(t, y_t) | + \| \nabla u(t, x_t)-\nabla u(t, y_t) \|_{\bV}&\le L_{\alpha}^{\delta}\| x_t-y_t \|_{0,\bV^*}^{\alpha},
\\
| \mathfrak{d}_t u(t, x_t)-\mathfrak{d}_t u(t, y_t) | + \| \mathfrak{d}_{\omega} u(t, x_t)-\mathfrak{d}_{\omega} u(t, y_t) \|_{\bR^m}  &\le L_{\alpha}^{\delta}\| x_t-y_t \|_{0,\bV^*}^{\alpha}.
\end{align*}
\end{enumerate}
We call the constant $\alpha$ is the exponent associated to $u\in\mathcal{C}_{\sF}^1$ and $0=\underline{t}_0<\underline{t}_1<...<\underline{t}_n=T$ the associated partition.
\end{defn}

Each $u\in\mathcal{C}_{\sF}^1$ can be thought of as an It\^{o} process and thus a semi-martingale parameterized by $x\in\Lambda(\bV^*)$. Doob-Meyer  decomposition theorem ensures the uniqueness of the integrable pair $(\mathfrak{d}_t u, \mathfrak{d}_{\omega} u)$ at points $(\omega, t, x_{s,t-s})$ for $0\le s<t\le T$, and a standard denseness argument may yield the uniqueness of the pair $(\mathfrak{d}_t u, \mathfrak{d}_{\omega} u)$ in $\mathcal{L}^2(\Lambda(\bV^*); \mathbb{R})\times\mathcal{L}^2(\Lambda(\bV^*); \mathbb{R}^m)$. In finite dimensional cases, this makes sense of the two linear operators $\mathfrak{d}_t$ and $\mathfrak{d}_{\omega}$ which are consistent with the differential operators in \cite[Section 5.2]{cont2013functional} and \cite[Theorem 4.3]{leao2018weak}. Particularly, if $u(t,x)$ is a deterministic function on the time-state space $[0,T]\times \bV^*$, we may have $\mathfrak{d}_{\omega} u \equiv 0$ and $\mathfrak{d}_t u$ coincides with the classical partial derivative in time; if the random function $u$ on $\Omega\times [0,T]\times \bV^*$ is regular enough (w.r.t. $\omega$) in the sense of Malliavin's calculus, the term $\mathfrak{d}_{\omega} u$ is just the Malliavin's derivative. In addition, the operators $\mathfrak{d}_t$ and $\mathfrak{d}_{\omega}$ are different from the path derivatives $(\partial_t, \partial_{\omega})$ via the functional It\^{o} formulas (see \cite{buckdahn2015pathwise} and \cite[Section 2.3]{ekren2016viscosity}); if $u(\omega, t, x)$ is smooth enough w.r.t. $(\omega, t)$ in the path space, for each $x$, we have the relation
\begin{equation*}
\mathfrak{d}_t u(\omega, t, x)=(\partial_t+\frac{1}{2}\partial_{\omega\omega}^2)u(\omega, t, x),  \text{ } \mathfrak{d}_{\omega} u(\omega, t, x)=\partial_{\omega} u(\omega, t, x),
\end{equation*}
which may be seen from \cite[Section 6]{ekren2016viscosity} and \cite{buckdahn2015pathwise}.

Let $\mathcal{T}^t$ be the set of stopping times $\tau$ valued in $[t, T]$ and $\mathcal{T}^t_{+}$ the subset of $\mathcal{T}^t$ such that for each $\tau\in\mathcal{T}_{+}^t$, $\tau>t$. Then for each $\tau\in\mathcal{T}^0$ and $\Omega_{\tau}\in\sF_{\tau}$, we denote by $L^0(\Omega_{\tau}, \sF_{\tau}; \Lambda_{\tau}^0(\bB))$\footnote{Here we actually mean $L^0(\Omega_{\tau}, \sF_{\tau}\cap\Omega_{\tau}; \Lambda_{\tau}^0(\bB))$, where $\sF_{\tau}\cap\Omega_{\tau}:=\{ B\cap\Omega_{\tau}: B\in\sF_{\tau} \}$. For simplicity, we just denote $L^0(\Omega_{\tau}, \sF_{\tau}; \Lambda_{\tau}^0(\bB)) = L^0(\Omega_{\tau}, \sF_{\tau}\cap\Omega_{\tau}; \Lambda_{\tau}^0(\bB))$.} the set of $\Lambda_{\tau}^0(\bB)$-valued $\sF_{\tau}$-measurable functions defined on $\Omega_{\tau}$. 

Now, for each $k\in\mathbb{N}^{+}$, $0\le t\le s\le T$, $\xi\in\Lambda_t^0(\bV)$, we define
\begin{align*}
\Lambda_{t, s}^{0,k;\xi}(\bV) = \Bigg\{ &x\in\Lambda^0_s(\bH)\cap L^2(0,s;\bV) : x(\tau) = \xi(t\land \tau)+\int_{t\land\tau}^{\tau} \! Ax(r) + g(r)dr, \tau\in [0,s], \text{ for some }
\\
&g\in L^{\infty}(0,T; \bH)\text{ with } \| g \|_{L^{\infty}(0,T; \bH)}\le k \Bigg\},
\end{align*}
and in particular , we set $\Lambda_{0,t}^{0,k}(\bV) = \cup_{\xi\in \bV} \Lambda_{0,t}^{0,k;\xi}(\bV)$ for each $t\in [0,T]$. Then obviously $\tilde{\Lambda}^0_{0,t}(\bV) := \cup_{k\in\mathbb{N}^+} \cup_{\xi\in \bV} \Lambda_{0,t}^{0,k;\xi}(\bV)$ is dense in $\Lambda_t^0(\bH)$. We also note that by \cite[Remark 2.3]{bayraktar2018path}, the path space $\Lambda_{0, t}^{0,k;\xi}(\bV)$ is compactly embedded into $\Lambda_t^0(\bH)$. 



We now introduce the notion of viscosity solutions. For each $(u, \tau)\in\mathcal{S}^2(\Lambda(\bV^*); \mathbb{R})\times\mathcal{T}^0$, $\Omega_{\tau}\in\sF_{\tau}$ with $\mathbb{P}(\Omega_{\tau})>0$ and $\xi\in L^0(\Omega_{\tau}, \sF_{\tau}; \Lambda_{\tau}^0(\bV))$, we define for each $k\in\mathbb{N}^{+}$,
\begin{align*}
\underline{\mathcal{G}}u(\tau, \xi; \Omega_{\tau}, k):=&\Bigg\{ \phi\in\mathcal{C}_{\sF}^1: \text{ there exists } \hat{\tau}_k\in \mathcal{T}^{\tau}_{+} \text{ such that } 
\\
&(\phi-u)(\tau, \xi)1_{\Omega_{\tau}}=0=\essinf_{\overline{\tau}\in \mathcal{T}^{\tau}} E_{\sF_{\tau}} \left[ \inf_{y\in\Lambda_{\tau,\overline{\tau}\land \hat{\tau}_k}^{0,k;\xi}(\bV)}(\phi-u)(\overline{\tau}\land \hat{\tau}_k,y) \right]1_{\Omega_{\tau}} \text{ }a.s. \Bigg\},
\end{align*}

\begin{align*}
\overline{\mathcal{G}}u(\tau, \xi; \Omega_{\tau}, k):=&\Bigg\{ \phi\in\mathcal{C}_{\sF}^1: \text{ there exists } \hat{\tau}_k\in \mathcal{T}^{\tau}_{+} \text{ such that } 
\\
&(\phi-u)(\tau, \xi)1_{\Omega_{\tau}}=0=\esssup_{\overline{\tau}\in \mathcal{T}^{\tau}} E_{\sF_{\tau}} \left[ \sup_{y\in\Lambda_{\tau,\overline{\tau}\land \hat{\tau}_k}^{0,k;\xi}(\bV)}(\phi-u)(\overline{\tau}\land \hat{\tau}_k,y) \right]1_{\Omega_{\tau}} \text{ }a.s. \Bigg\},
\end{align*}

Throughout this work, by saying $(s, x)\to (t^{+},\xi)$ for some $(t,\xi)\in [0,T)\times\Lambda_t^0(\bV)$, we mean that there exists $k\in \bN^+$  such that $(s,x)\to(t^{+},\xi)$ with $s\in (t, T]$, $x\in \Lambda_{t,s}^{0,k;\xi}(\bV)$, and $\sup_{r\in[t,s]}\|x(r)-\xi(t)\|_{\bV^*}\to 0$. The definition of viscosity solutions then comes as follows.
\begin{defn}\label{def_vis_sol}
We say $u\in\mathcal{S}^2(\Lambda^0(\bH); \mathbb{R})$ is a viscosity subsolution (resp. supersolution) of SPHJ equation \eqref{SPHJ}, if $u(T, y)\le (\text{resp.} \ge)G(y)$ for all $y\in\Lambda_T^0(\bH)$ a.s., and for any $K_0\in\mathbb{N}^{+}$, there exists $k\ge K_0$ such that for any $\tau\in\mathcal{T}^0$, $\Omega_{\tau}\in\sF_{\tau}$ with $\mathbb{P}(\Omega_{\tau})>0$ and $\xi\in L^0(\Omega_{\tau}, \sF_{\tau}; \Lambda_{\tau}^0(\bV))$ and any $\phi\in\underline{\mathcal{G}}u(\tau, \xi; \Omega_{\tau}, k)$(resp. $\phi\in\overline{\mathcal{G}}u(\tau, \xi; \Omega_{\tau}, k)$), there holds,
\begin{equation}\label{subsol}
\mbox{ess}	\liminf_{(s,x)\to({\tau}^{+}, \xi)}   \left\{ -\mathfrak{d}_s \phi(s,x)-\mathcal{H}(s, x, \nabla \phi(s,x)) \right\}\le 0, \text{ for } \omega\in\Omega_{\tau} \text{ a.s. }
\end{equation}

\begin{equation}\label{supsol}
(\text{ resp. } \mbox{ess}	\limsup_{(s,x)\to({\tau}^{+}, \xi)}   \left\{ -\mathfrak{d}_s \phi(s,x)-\mathcal{H}(s, x, \nabla \phi(s,x)) \right\}\ge 0, \text{ for } \omega\in\Omega_{\tau} \text{ a.s. )}
\end{equation}
The function $u$ is a viscosity solution of SPHJ equation \eqref{SPHJ} if it is both a viscosity subsolution and a viscosity supersolution of \eqref{SPHJ}.
\end{defn}


  The test function space $ \mathscr C_{\sF}^1$ is expected to include the $L^2$-solutions of \textit{ordinary} backward stochastic differential equations which are space-invariant and have $\mathfrak{d}_{t}u$ allow for time-discontinuity and just measurability in time $t$.   On the other hand, as $(s,x)\to({\tau}^{+}, \xi)$, we have $x\in\Lambda_{\tau, s}^{0,k;\xi}(\bV)$ for some $k\in\bN^+$ and thus  $x\in L^2(0,s;\bV)$ which is well defined a.e. but not pointwisely on $[0,s]$ as $\bV$-valued functions; consequently, $Ax(r)$ and $\mathcal{H}(r, x, \nabla \phi(r,x))$ for $\tau \leq r \leq s$ are not pointwisely but a.e. well defined. Also, the involved functions and terms are just measurable w.r.t. $\omega\in\Omega$ and all such measurability features  motivate us to use essential limits in \eqref{subsol} and \eqref{supsol}.

\begin{rmk}\label{r2.1}

In the above definition, we can see that each viscosity subsolution (resp. supersolution) of SPHJ equation \eqref{SPHJ} is associated to an infinite sequence of integers $1\le \underline{k}_1 \le \cdots \le \underline{k}_n \le \cdots$ (resp. $1\le \overline{k}_1 \le \cdots \le \overline{k}_n \le \cdots$) such that the required properties in Definition \ref{def_vis_sol} are holding for all the test functions in $\underline{\mathcal{G}}u(\tau, \xi; \Omega_{\tau}, \underline{k}_i)$ (resp. $\overline{\mathcal{G}}u(\tau, \xi; \Omega_{\tau}, \overline{k}_i)$) for $\forall \text{ }i\in\mathbb{N}^+$.

Throughout this paper we define for $\forall \phi\in \mathcal{C}_{\sF}^1$, $v\in U$, $t\in[0, T]$, and $x_t \in \tilde{\Lambda}^0_{0,t}(\bV)$,
\begin{equation*}
\mathcal{L}^v \phi(t, x_t) = \mathfrak{d}_t \phi(t, x_t) + {_{\bV^*}\langle} Ax(t),\nabla \phi(t, x_t)\rangle_{\bV} + {_{\bV^*}\langle}\beta(t, x_t, v),\nabla \phi(t, x_t)\rangle_{\bV}.
\end{equation*}

\end{rmk}

\begin{rmk}\label{r2.2}
In view of $(\mathcal{A}1)$, for each $\phi\in\mathcal{C}_{\sF}^1$, there exists an $(\sF_t)_{t\ge 0}$-adapted process $\zeta^{\phi}\in \cL^2(\Lambda(\bV^*);\bR)$ such that for a.e. $(\omega, t)\in\Omega\times[0, T]$, and all $x_t \in \tilde{\Lambda}^0_{0,t}(\bV)$, we have
\begin{equation*}
|  -\mathfrak{d}_t \phi(t, x_t) - \mathcal{H}(t, x_t, \nabla\phi(t, x_t))  | \le \sup_{  v\in U  } |  \mathcal{L}^v \phi(t, x_t) + f(t, x_t, v)  | \le \zeta_t^{\phi} + c_3\rho \|x(t)\|_{\bV}.
\end{equation*}
So the essential limits in the above definition is well defined. Meanwhile, there exists a finite partition $0 = \underline{t}_0 <  \underline{t}_1 < \cdots < \underline{t}_n = T$, such that for any $\min_{0\le j\le n-1} |  \underline{t}_{j+1} - \underline{t}_{j}  | > \delta > 0$, and for almost all $t\in \cup_{0\le j\le n-1} (  \underline{t}_j, \underline{t}_{j+1} - \delta  ]$, $k\in\bN^+$, $\xi\in \bV$ and all $x_t, \overline{x}_t\in\Lambda_{0,t}^{0,k;\xi}(\bV)$, it holds that
\begin{align*}
&\text{ }\text{ }\text{ }\text{ } \left|  \{  -\mathfrak{d}_t \phi(t, x_t) - \mathcal{H}(t, x_t, \nabla\phi(t, x_t))  \} - \{  -\mathfrak{d}_t \phi(t, \overline{x}_t) - \mathcal{H}(t, \overline{x}_t, \nabla \phi(t, \overline{x}_t))  \}  \right|
\\
&\le \sup_{v\in U} \left|  (  \mathcal{L}^v \phi(t, x_t) + f(t, x_t, v)  ) - (  \mathcal{L}^v \phi(t, \overline{x}_t) + f(t, \overline{x}_t, v)  )  \right|
\\
&\le (L^{\delta}_{\alpha} + L) \left[  (L+1)\|  x_t - \overline{x}_t  \|_{0,\bV^*}^{\alpha} +  (1 + \rho)\|  x_t - \overline{x}_t  \|_{0,\bH}  \right] 
\\
&\quad
	+ \sup_{v\in U}\left| {_{\bV^*}\langle}Ax(t), \nabla \phi(t, x_t)\rangle_{\bV} - {_{\bV^*}\langle} A\overline{x}(t),\nabla \phi(t, \overline{x}_t)\rangle_{\bV} \right|
\end{align*}
where $\alpha$ is the exponent associated to $\phi\in\mathcal{C}_{\sF}^1$. 

Notice that here on contrary to the finite dimensional case, we have extra $A$ operator terms in the functions $\mathcal{L}^v \phi(t, x_t)$ and $\mathcal{L}^v \phi(t, \overline{x}_t)$. They may be unbounded  according to our set up, thus we need to consider them separately. Applying triangular inequality and the coercivity of $A$ operator, we have the following result,
\begin{align*}
&\text{ }\text{ }\text{ }\text{ }\sup_{v\in U}\left| {_{\bV^*}\langle}Ax(t),\nabla \phi(t, x_t)\rangle_{\bV} - {_{\bV^*}\langle}A\overline{x}(t), \nabla \phi(t, \overline{x}_t)\rangle_{\bV} \right|
\\
& = \sup_{v\in U}\left | {_{\bV^*}\langle}A(x(t) - \overline{x}(t)), \nabla \phi(t, x_t)\rangle_{\bV} + {_{\bV^*}\langle }A\overline{x}(t), \nabla\phi(t,x_t) - \nabla\phi(t,\overline{x}_t)\rangle_{\bV} \right |
\\
& \le  \left\| A(x(t) - \overline{x}(t))  \right\|_{\bV^*} \left\| \nabla \phi(t, x_t) \right\|_{\bV} +  \left\| A\overline{x}(t) \right\|_{\bV^*}  \left\| \nabla\phi(t,x_t) - \nabla\phi(t,\overline{x}_t) \right\|_{\bV}
\\
&\le  c_3 \rho \left\| x(t) - \overline{x}(t) \right\|_{\bV} +  c_3 L_{\alpha}^{\delta} \left\| \overline{x}(t) \right\|_{\bV} \cdot \left\| x_t - \overline{x}_t \right\|_{0,\bV^*}^{\alpha}.
\end{align*}




\end{rmk}


\section{Existence of the viscosity solution}
\subsection{Some auxiliary results}
For any $T>0$, denote by  $\overline{\cS}^2([0,T];\bH)$ the space of $(\sF_t)_{t\ge 0}$-adapted $\bH$-valued time continuous stochastic processes $u: \Omega\times [0,T]\to \bH$, such that for any $u\in\overline{\cS}^2([0,T];\bH)$, we have
\begin{equation*}
\|u\|_{\overline{\cS}^2([0,T];\bH)} := \left(E \left[ \max_{t\in [0,T]} \left\| u(t) \right\|_{\bH}^2 \right]\right)^{\frac{1}{2}} < \infty.
\end{equation*}
Similarly, denote by $\overline{L}^2(0,T;\bV)$ the space of $(\sF_t)_{t\ge 0}$-adapted $\bV$-valued stochastic processes $v: \Omega\times [0,T]\to \bV$, such that for any $v\in\overline{L}^2(0,T;\bV)$, we have
\begin{equation*}
\|v\|_{\overline{L}^2(0,T;\bV)} := \left(E \int_0^T \! \left\| v(t) \right\|_{\bV}^2 dt \right)^{\frac{1}{2}} < \infty.
\end{equation*}

\begin{defn}\label{def_state_sol}
An $(\sF_t)_{t\ge 0}$-adapted process $X\in \overline{\cS}^2([0,T];\bH) \cap \overline{L}^2(0,T;\bV)$ is called a solution of (\ref{state-proces-contrl}) if we have $\bP$-a.s.
\begin{equation}
X(t) = X_0 + \int_0^t \! AX(s) + \beta(s,X_s,\theta(s)) ds, \text{ }\text{ }\forall t\in [0,T],
\end{equation}
where both sides are thought of as $\bV^*$-valued $(\sF_t)_{t\ge 0}$-adapted processes.

\end{defn}

\begin{thm}\label{state_sol_wellposedness}
Given $X_0\in \bH$, $\theta\in\cU$, under Assumption $(\cA 1)$, the stochastic differential equation \eqref{state-proces-contrl} admits a unique solution $X$ in the sense of Definition \ref{def_state_sol}.
\end{thm}

\begin{proof}
For each $X_0\in \bH$, $\theta\in\cU$, $T_0\in [0,T]$, and any given $\tilde{X} \in \overline{\cS}^2([0,T_0];\bH) \cap \overline{L}^2(0,T_0;\bV)$, the following stochastic differential equation
\begin{equation}\label{SDE}
\left\{
\begin{split}
\frac{dX(t)}{dt} &= AX(t) + \beta(t, \tilde{X}_t,\theta(t)),
\\
X(0) &= X_0,
\end{split}
\right.
\end{equation}
admits a unique solution (\cite[Theorem 4.2.4]{prevot2007concise})
\begin{equation*}
X\in \overline{L}^2(0,T_0;\bV)\cap\overline{\cS}^2([0,T_0];\bH).
\end{equation*}
Therefore, we may define the solution map
\begin{align*}
M: \overline{L}^2(0,T_0;\bV)\cap\overline{\cS}^2([0,T_0];\bH) &\to \overline{L}^2(0,T_0;\bV)\cap\overline{\cS}^2([0,T_0];\bH)
\\
\tilde{X} &\to M_{\tilde{X}},
\end{align*}
where $X$ is the unique solution of (\ref{SDE}) associated to $\tilde{X}$.

For each $\tilde{X},\tilde{Y}\in \overline{L}^2(0,T_0;\bV)\cap\overline{\cS}^2([0,T_0];\bH)$, we set
\begin{equation*}
X = M_{\tilde{X}}, Y = M_{\tilde{Y}}.
\end{equation*}

Applying the It\^o's formula introduced by \cite[Theorem 4.2.5]{prevot2007concise}, we have
\begin{align*}
&\left\|Y(t) - X(t)\right\|_{\bH}^2
\\
= &\left\| M_{\tilde{Y}}(t) - M_{\tilde{X}}(t) \right\|_{\bH}^2
\\
= &\int_0^t \! 2 \langle A(Y(s) - X(s)), Y(s) - X(s) \rangle + 2 \langle \beta(s,\tilde{Y}_s,\theta(s)) - \beta(s,\tilde{X}_s,\theta(s)), Y(s) - X(s) \rangle ds
\\
\le &\int_0^t \! c_1 \left\| Y(s) - X(s) \right\|_{\bH}^2 - c_2 \left\| Y(s) - X(s) \right\|_{\bV}^2 ds
\\
&+2\int_0^t \! \left\| \beta(s,\tilde{Y}_s,\theta(s)) - \beta(s,\tilde{X}_s,\theta(s)) \right\|_{\bV^*} \left\| Y(s) - X(s) \right\|_{\bV} ds
\\
\le &\int_0^t \! c_1 \left\| Y(s) - X(s) \right\|_{\bH}^2 - c_2 \left\| Y(s) - X(s) \right\|_{\bV}^2 + 2L\left\| \tilde{Y}_s - \tilde{X}_s \right\|_{0,\bH}  \left\| Y(s) - X(s) \right\|_{\bV} ds
\\
\le &\int_0^t \! c_1 \left\| Y(s) - X(s) \right\|_{\bH}^2 - \frac{c_2}{2} \left\| Y(s) - X(s) \right\|_{\bV}^2 ds + \frac{2L^2}{c_2} \int_0^t \!\left\| \tilde{Y}_s - \tilde{X}_s \right\|_{0,\bH}^2 ds, 
\end{align*}
where $c_1\in\bR$ and $c_2>0$ by the coercivity assumption of operator $A$.

Put $c_1^+ = max\{0, c_1\}$, and $\tilde{c}_2 := \frac{c_2}{2} >0$. We have
\begin{align*}
&\max_{s\in [0,t]}\left\|Y(s) - X(s)\right\|_{\bH}^2 + \tilde{c}_2\int_0^t \! \|Y(s) - X(s)\|_{\bV}^2 ds
\\
\le &\int_0^t \! c_1^+ \left\| Y(s) - X(s) \right\|_{\bH}^2 ds + \frac{2L^2t}{c_2} \max_{s\in [0,t]}\left\| \tilde{Y}(s) - \tilde{X}(s) \right\|_{\bH}^2,  
\end{align*}
which, by Gr\"onwall's inequality, gives
\begin{align*}
&\max_{s\in [0,T_0]}\left\|Y(s) - X(s)\right\|_{\bH}^2 + \tilde{c}_2\int_0^{T_0} \! \|Y(s) - X(s)\|_{\bV}^2 ds
\\
\le &\text{ }\frac{2L^2T_0}{c_2} \max_{s\in[0,T_0]}\left\| \tilde{Y}(s) - \tilde{X}(s) \right\|_{\bH}^2 \cdot e^{\int_0^{T_0} \!  c_1^+  ds }
\\
\le &\text{ }\frac{L^2T_0}{\tilde{c}_2} \cdot e^{T_0 c_1^+} \max_{s\in [0,T_0]}\left\| \tilde{Y}(s) - \tilde{X}(s) \right\|_{\bH}^2.
\end{align*}
If we choose proper $T_0 >0$ such that $\frac{L^2T_0}{\tilde{c}_2} \cdot e^{T_0 c_1^+ } <1$, we can prove that the mapping $M$ is a contraction. By Banach fixed point theorem, it admits a unique fixed point. Such fixed point is the unique solution of \eqref{SDE} over the time interval $[0,T_0]$. Note that $T_0$ depends only on $c_1^+$, $\tilde{c}_2$, and $L$. Similarly, we may obtain the unique solution over time intervals $[T_0, 2T_0]$ and recursively after finite steps, we may obtain the unique solution over the whole time interval $[0,T]$. The proof is complete.

\end{proof}

Under assumption $(\cA1)$, the following assertions hold.
\begin{lem}\label{lemma3.1}
Let $(\mathcal{A}1)$ hold. Given $\theta\in\mathcal{U}$, the stochastic ordinary differential equation $(1.2)$ admits a unique solution, and there exists a constant $K>0$ such that, for any $0\le r\le t\le s\le T$, and $\xi\in L^0(\Omega, \mathcal{F}_r; \Lambda_r(\bH))$,

(i) the two processes $(X_s^{r,\xi;\theta})_{t\le s\le T}$ and $(X_s^{t,X_t^{r,\xi;\theta};\theta})_{t\le s\le T}$ are indistinguishable;

(ii) $max_{r\le l\le T} \| X^{r,\xi;\theta}(l) \|_{\bH}^2 + \int_r^T \! c_2 \cdot \|X^{r,\xi;\theta}(l)\|_{\bV}^2 dl \le K^2(1 + \| \xi \|_{0,\bH}^2)$ a.s.;

(iii) $d_{0,\bV^*}(X_s^{r,\xi;\theta}, X_t^{r,\xi;\theta}) \le K (1 + \|\xi\|_{0,\bH}) |s-t|^{1/2}$ a.s.;

(iv) given another $\hat{\xi}\in L^0(\Omega, \sF_r; \Lambda_r(\bH))$, 
\begin{align*}
&\max_{r\le l\le T} \| X^{r, \xi; \theta}(l) - X^{r, \hat{\xi}; \theta}(l) \|_{\bH}^2 + c_2 \int_r^T \! \left\|X^{r, \xi; \theta}(l) - X^{r, \hat{\xi}; \theta}(l)\right\|_{\bV}^2 dl  \le K^2 \| \xi - \hat{\xi} \|_{0,\bH}^2 \text{  a.s.;}
\end{align*}

(v) if we further assume $\xi\in L^0(\Omega,\sF_r;\Lambda_r(\bV))$ with $A\xi\in L^0(\Omega,\sF_r;\Lambda_r(\bH))$, then we have
\begin{equation*}
\max_{r\le l\le s} \left\| X^{r,\xi_r;\theta}(l) - \xi(r) \right\|_{\bH}^2 + c_2 \int_r^s \! \left\| X^{r,\xi_r;\theta}(l) - \xi(r) \right\|_{\bV}^2 dl \le K(1+\|A\xi(r)\|_{\bH}^2) \cdot |s-r|^2 \text{ a.s.;}
\end{equation*}

(vi) the constant $K$ only depends on the choice of $c$, $c_1^+$, $c_2$, $c_3$, $L$, and $T$, and is independent of the control process.
\end{lem}

We postpone the cumbersome calculations for the proof to the appendix.


\begin{prop}\label{prop3.2}
Let $(\mathcal{A}1)$ hold.

(i) For each $t\in [0,T]$, $\epsilon\in (0, \infty)$, and $\xi \in L^0(\Omega,\sF_t; \Lambda_t(\bH))$, there exists $\overline{\theta}\in \mathcal{U}$ such that
\begin{equation*}
E\left[ J(t,\xi;\overline{\theta}) - V(t,\xi) \right] < \epsilon.
\end{equation*}

(ii) For each $(\theta, x_0)\in \mathcal{U}\times \bH$, $\left\{ J(t,X_t^{0,x_0;\theta}; \theta) - V(t, X_t^{0,x_0;\theta}) \right\}_{t\in [0,T]}$ is a supermartingale, i.e., for any $0\le t\le \tilde{t}\le T$,
\begin{equation}
V(t, X_t^{0,x_0;\theta}) \le E_{\sF_t} V(\tilde{t}, X_{\tilde{t}}^{0,x_0;\theta}) + E_{\sF_t} \int_t^{\tilde{t}} \! f(s, X_s^{0, x_0;\theta}, \theta_{s})ds, \text{a.s.. } \text{ } 
\end{equation}

(iii) For each $(\theta, x_0)\in \mathcal{U} \times \bH$, $\left\{ V(s, X_s^{0, x_0; \theta}) \right\}_{s\in [0,T]}$ is a continuous process.

(iv) With probability 1, $V(t,x)$ and $J(t,x;\theta)$ for each $\theta\in\mathcal{U}$ are continuous on $[0,T]\times \Lambda(\bH)$ and
\begin{equation*}
\sup_{(t,x)\in [0,T] \times \Lambda(\bH)} \max \left\{ | V(t,x_t) |, | J(t, x_t;\theta) | \right\} \le L(T+1) \text{ }\text{ }\text{ } a.s..
\end{equation*}

(v) There exists $L_V > 0$ such that for each $(\theta, t)\in \mathcal{U} \times [0,T]$,
\begin{equation*}
| V(t,x_t) - V(t,y_t) | + | J(t, x_t; \theta) - J(t, y_t; \theta) | \le L_V\| x_t - y_t \|_{0,\bH}, \text{ }\text{ }\text{ } a.s., \text{ }\text{ }\text{ } \forall x_t,y_t \in \Lambda_t(\bH),
\end{equation*}
with $L_V$ depending only on $T$ and $L$
\end{prop}

Again we postpone the proof to the Appendix. Then we prove the following dynamical programming principle.


 \begin{thm}\label{dynamic}
 Let $(\mathcal{A}1)$ hold. For any stopping times $\tau, \hat{\tau}$ with $\tau \le \hat{\tau} \le T$, and any $\xi \in L^0(\Omega, \sF_{\tau}; \Lambda^0_{\tau}(\bV))$, we have
 \begin{equation*}
 V(\tau, \xi) = \essinf_{\theta \in \mathcal{U}} E_{\sF_{\tau}} \left[ \int_{\tau}^{\hat{\tau}} \! f( s, X_s^{\tau, \xi; \theta}; \theta(s) )ds + V(\hat{\tau}, X_{\hat{\tau}}^{\tau, \xi; \theta}) \right] \text{ } \text{ } \text{ } a.s.
 \end{equation*}
 
\begin{proof}
This proof is similar to \cite[Theorem 3.4]{qiu2018viscosity}, but with some delicate compactness argument about path subspaces and the infinite dimensional expansion.

Denote the right hand side by $\overline{V}(\tau, \xi)$. By Proposition \ref{prop3.2} $(iv)$, $(v)$, we can see that both $V(\tau,\xi)$ and $\overline{V}(\tau,\xi)$ are lying in the space $\mathcal{S}^{\infty}(\Lambda(\bH); \mathbb{R})$ and thus the continuity indicates that it is sufficient to prove Theorem \ref{dynamic} when $\tau$, $\hat{\tau}$, and $\xi$ are deterministic.

For $\forall \text{ } \epsilon > 0$, by Proposition \ref{prop3.2} $(v)$, there exists $\delta = \epsilon/L_V >0$ s.t. whenever $|| x_{\hat{\tau}} -y_{\hat{\tau}} ||_{0,\bH} < \delta$ for $x_{\hat{\tau}}, y_{\hat{\tau}} \in \Lambda_{\hat{\tau}}(\bH)$, it holds that
\begin{equation*}
| J(\hat{\tau}, x_{\hat{\tau}}; \theta) - J(\hat{\tau}, y_{\hat{\tau}}; \theta) | + | V(\hat{\tau}, x_{\hat{\tau}}) - V(\hat{\tau}, y_{\hat{\tau}}) | \le \epsilon \text{ }\text{ }\text{ } a.s., \forall \text{ } \theta \in \mathcal{U}.
\end{equation*}

Note that $\Lambda_{\tau, \hat{\tau}}^{0, L; \xi}(\bV)$ is compactly embedded into $\Lambda_{\hat{\tau}}^{0}(\bH)$. There exists a sequence $\{ x^j \}_{j\in\mathbb{N}^+} \in \Lambda_{\tau, \hat{\tau}}^{0, L; \xi}(\bV)$ s.t. $\cup_{j\in\mathbb{N}^+}( \Lambda_{\tau, \hat{\tau}}^{0, L; \xi}(\bV) \cap B^{\bH}_{\delta/3}(x^j) ) = \Lambda_{\tau, \hat{\tau}}^{0, L; \xi}(\bV)$. Denote $D_1 = \Lambda_{\tau, \hat{\tau}}^{0, L; \xi}(\bV) \cap B^{\bH}_{\delta/3}(x^1)$, and
\begin{equation*}
D_j = (  B^{\bH}_{\delta/3}(x^j) - (  \cup_{i=1}^{j-1} B^{\bH}_{\delta/3}(x^i)  )  ) \cap \Lambda_{\tau, \hat{\tau}}^{0, L; \xi}(\bV), \text{ }\text{ }\text{ } j>1.
\end{equation*}
Then $\{D_j\}_{j\in\mathbb{N}^+}$ is a partition of $\Lambda_{\tau, \hat{\tau}}^{0, L; \xi}(\bV)$ with diameter $diam(D^j)<\delta$, i.e., $D^j \subset \Lambda_{\tau, \hat{\tau}}^{0, L; \xi}(\bV)$, $\cup_{j\in\mathbb{N}^+} D^j = \Lambda_{\tau, \hat{\tau}}^{0, L; \xi}(\bV)$, $D^i \cap D^j = \emptyset$ if $i \neq j$, and for any $x, y \in D^j$, $\| x - y \|_{0,\bH}<\delta$.

Then the rest of the proof is similar to  that of \cite[Theorem 3.4]{qiu2018viscosity}. For each $j\in\mathbb{N}^+$, take $\overline{x}^j \in D^j$, and a straight forward application of Proposition \ref{prop3.2} $(i)$ leads to some $\theta ^j \in \mathcal{U}$ satisfying
\begin{equation*}
0 \le J(\hat{\tau}, \overline{x}^j; \theta ^j) - V(\hat{\tau}, \overline{x}^j) := \alpha^j \text{ }\text{ }\text{ } a.s., with \text{ } E\left[ \alpha^j \right] < \frac{\epsilon}{2^j}.
\end{equation*}
Thus for each $x\in D^j$, with triangular inequality and results above, it holds that
\begin{align*}
&\text{ }\text{ }\text{ }\text{ } J(\hat{\tau}, x; \theta^j) - V(\hat{\tau}, x)
\\
&\le | J(\hat{\tau}, x; \theta^j) - J(\hat{\tau}, \overline{x}^j, \theta^j) | + | J(\hat{\tau}, \overline{x}^j; \theta^j) - V(\hat{\tau}, \overline{x}^j) | + | V(\hat{\tau}, \overline{x}^j) - V(\hat{\tau}, x) |
\\
&\le \epsilon + \alpha^j, \text{ }\text{ }\text{ } a.s..
\end{align*}

Further by the uniform boundedness of coefficient $\beta$ introduced by $(\mathcal{A}1)$ $(iii)$, for $\forall \text{ } \theta \in \mathcal{U}$, $X_{\hat{\tau}}^{\tau, \xi; \theta} \in\Lambda_{\tau, \hat{\tau}}^{0, L; \xi}(\bV)$ a.s. We then introduce the following control

\[
\tilde{\theta}(s) = 
\begin{cases}
      \theta(s), &\text{if} \text{ } s\in[0, \hat{\tau});\\
      \sum_{j\in\mathbb{N}^+} \theta^j(s)1_{D^j}(X_{\hat{\tau}}^{\tau, \xi; \theta}), &\text{if} \text{ } s\in [\hat{\tau}, T].
\end{cases}
\]
Then if follows that
\begin{align*}
V(\tau, \xi) &\le J(\tau, \xi; \tilde{\theta})
\\
&= E_{\sF_{\tau}} \left[  \int_{\tau}^{\hat{\tau}} \! f(  s, X_s^{\tau, \xi; \theta}, \theta(s)  )ds + J(\hat{\tau}, X_{\hat{\tau}}^{\tau, \xi; \theta}; \tilde{\theta})  \right]
\\
&\le E_{\sF_{\tau}} \left[  \int_{\tau}^{\hat{\tau}} \! f(  s, X_s^{\tau, \xi; \theta}, \theta(s)  )ds + V(\hat{\tau}, X_{\hat{\tau}}^{\tau, \xi; \theta}) + \sum_{j\in\mathbb{N}^+} \alpha^j  \right] + \epsilon,
\end{align*}
where ${\alpha^j}$ is independent of the choices of the control process $\theta$. Then take expectation on both hand side, we are going to obtain
\begin{align*}
EV(\tau, \xi) \le E\overline{V}(\tau, \xi) + 2\epsilon,
\end{align*}
which by the arbitrariness of $\epsilon$ and together with the obvious relation
\begin{equation*}
V(\tau, \xi) \ge \overline{V}(\tau, \xi),
\end{equation*}
yields the equality
\begin{equation*}
V(\tau, \xi) = \overline{V}(\tau, \xi) \text{ }\text{ }\text{ } a.s.
\end{equation*}

\end{proof}
\end{thm}

\subsection{Existence of the viscosity solution}
Due to our path-dependence setting, we shall compose the random fields and stochastic differential equations by generalizing an It\^{o}-Kunita formula \cite[Pages 118-119]{kunita1981some}. Recall that for each $\phi\in\mathcal{C}_{\sF}^1$, $v\in U$, $t\in [0, T]$, $x_t\in\tilde\Lambda^0_{0,t}(\bV)$, we have
\begin{equation*}
\mathcal{L}^v\phi(t,x_t)=\mathfrak{d}_t\phi(t,x_t)+{_{\bV^*}\langle}Ax(t), \nabla\phi(t, x_t)\rangle_{\bV} +{_{\bV^*}\langle}\beta(t,x_t,v),\nabla\phi(t,x_t)\rangle_{\bV}.
\end{equation*}

\begin{lem}\label{lemma3.4}
Let Assumption ($\mathcal{A}1$) hold. Suppose $u\in\mathcal{C}^1_{\sF}$ with the associated partition $0=\underline{t}_0<\underline{t}_1<...<\underline{t}_n=T$, then for each $\theta\in\mathcal{U}$, it holds a.s. that for each $\underline{t}_j\le\rho\le\tau<\underline{t}_{j+1}$, $j=0, ..., n-1$ and $x_{\rho}\in \Lambda^0_{\rho}(\bV)$,
\begin{equation}\label{lemma3.4eq}
u(\tau, X_{\tau}^{\rho, x_{\rho}; \theta})-u(\rho, x_{\rho})=\int_{\rho}^{\tau} \! \mathcal{L}^{\theta(s)}u(s, X_s^{\rho, x_{\rho}; \theta})ds+\int_{\rho}^{\tau} \! \mathfrak{d}_{\omega} u(s, X_s^{\rho, x_{\rho}; \theta})dW(s) \text{ a.s.  }
\end{equation}

\end{lem}

We will again postpone the proof to the Appendix.




\begin{thm}\label{exist}
Let $(\mathcal{A}1)$ hold and $\bV^A:= \{\eta\in \bV: A\eta\in \bH\}$ be dense in $\bV$. The value function V defined by \eqref{value} is a viscosity solution of the SPHJ equation \eqref{SPHJ}.
\begin{proof}
We prove the value function $V$ is both a viscosity supersolution and a viscosity subsolution.
Obviously by Proposition \ref{prop3.2} ($iv$), the value function $V\in\mathcal{S}^{\infty}(\Lambda(\bH); \mathbb{R})$.
 
\textbf{Step 1.} To prove $V$ is a viscosity subsolution, we only need to prove \eqref{subsol} holds. 

Suppose it does not hold, i.e. suppose for any $k\in\mathbb{N}^{+}$, $k\ge K_0$ for some $K_0\in\mathbb{N}^{+}$, there exists $\phi\in\underline{\mathcal{G}}V(\tau, \xi_{\tau}; \Omega_{\tau}, k)$ with $\tau\in\mathcal{T}^0$, $\Omega_{\tau}\in\sF_{\tau}$, $\mathbb{P}(\Omega_{\tau})>0$, and $\xi_{\tau}\in L^0(\Omega_{\tau}, \sF_{\tau}; \Lambda_{\tau}^0(\bV))$ such that there exists $\epsilon,\tilde \delta>0$, $\tau'\in\mathcal T^{\tau}_+$ and $\Omega'\subset\Omega_{\tau}$ satisfying $\mathbb{P}(\Omega')>0$, $\Omega'\subset \{\tau<\tau'\}$, and a.e. on $\Omega'$,
\begin{equation}
\essinf_{s\in [\tau, (\tau+\tilde{\delta}^2)\land \tau'], x\in B^{\bV^*}_{\tilde{\delta}} (\xi_{\tau})\cap\Lambda_{{\tau, s\land T}}^{0, k; \xi_{\tau}}(\bV)}   \left\{ -\mathfrak{d}_s \phi(s, x) - \mathcal{H}(s, x, \nabla\phi(s, x)) \right\} \ge 3\epsilon,
\end{equation}
Choose $k\in\mathbb{Z}^{+}$ s.t. $k>L$ and let $\hat{\tau}_k$ be the stopping time corresponding to the test function $\phi\in\underline{\mathcal{G}}V(\tau, \xi_{\tau}; \Omega_{\tau}, k)$.
Here for any $t\in[0,T]$, set $\xi(t) = \xi(\tau\land t)$ and we have $\xi\in\Lambda^0_{T}(\bV)$. 
 
 Recall that, for $\phi\in\mathcal C_{\sF}^1$, there is a partition $0=\underline t_0<\underline t_1<\ldots<\underline t_n=T$. W.l.o.g., we assume that there exists $\tilde{\delta}\in (0,1)$ with $2 \tilde \delta^2 <\min_{0\leq j\leq n-1} (\underline t_{j+1}-\underline t_j)$ such that $\Omega'=\{[\tau,\tau+2\tilde{\delta}^2] \subset [\underline t_j,\underline t_{j+1})\}$ for some $j\in\{0,\ldots, n-1\}$. W.l.o.g., we take $\hat\tau_k=\tau'$.
  
By the continuity assumption $(\cA 1)$ (ii) and the measurable selection theorem, there exists $\overline{\theta}\in\cU$ such that for almost all $\omega\in\Omega'$, it holds that
\begin{equation}\label{pick_theta}
- \cL^{\overline{\theta}(s)} \phi(s, \xi_{s}) - f(s, \xi_{s},\overline{\theta}(s)) \ge - \mathfrak{d}_s \phi(s, \xi_{s}) - \cH(s, \xi_{s}, \nabla \phi(s, \xi_{s})) - \frac{\epsilon}{2},
 \end{equation}
 for almost all $s$ satisfying $\tau\le s < (\tau+\tilde{\delta}^2)\land \hat{\tau}_k$.
 
By the continuity argument regarding the path dependence in Proposition \ref{prop3.2} (iv) and Definition \ref{cf1}, for each $h\in (0,\tilde{\delta}^2)$, $t\in [\tau,(\tau+h)\land T]$, there exist $\delta_1>0$ such that for any $\tilde{\xi}_{t}\in \Lambda_{t}^{0}(\bH)$ satisfying
\begin{equation*}
\left\| \xi_t - \tilde{\xi_t} \right\|_{0,\bH} \le \delta_1,
\end{equation*}
it holds that
\begin{equation*}
\left| \left( \phi - V \right)(t,\xi_t) - \left( \phi - V \right)(t,\tilde{\xi}_t) \right| \le \frac{\epsilon}{4}h.
\end{equation*}

Similarly by Lemma \ref{lemma3.1} (iv), for each $h\in (0,\tilde{\delta}^2)$, $t\in [\tau,(\tau+h)\land T]$, there exist $\delta_2>0$ such that for any $\overline{\xi}_{\tau}\in \Lambda_{\tau}^{0}(\bH)$ satisfying
\begin{equation*}
\left\| \xi_{\tau} - \overline{\xi}_{\tau} \right\|_{0,\bH} \le \delta_2,
\end{equation*}
it holds a.s. that
\begin{equation*}
\left| \left( \phi - V \right)(t,X_t^{\tau,\xi_{\tau};\overline{\theta}}) - \left( \phi - V \right)(t,X_t^{\tau,\overline \xi_{\tau};\overline{\theta}}) \right| \le \frac{\epsilon}{4}h.
\end{equation*}
  
By the denseness of $\bV^A$ in $\bV$ and the H\"older continuity argument in Remark \ref{r2.2}, for each $\epsilon >0$, there exists $\hat{\xi}$ with $\hat{\xi}(s) := \xi(s)1_{[0,\tau)}(s) + \hat{\xi}(\tau)1_{[\tau,T]}(s)$ such that,
\begin{enumerate}
\item [(1)] $\hat{\xi}(\tau)$ takes value in $\bV^A$,
\item [(2)] for each $\delta_1,\delta_2>0$ mentioned above, 
\begin{align}
\left\| \xi(\tau) - \hat{\xi}(\tau) \right\|_{\bV} \le &\left[\frac{\epsilon}{6(1+c_3)(1+\rho) (1+L+L_{\alpha}^{\delta}) (1+(L+1)(L+L_{\alpha}^{\delta}) + c_3L_{\alpha}^{\delta}\|\xi(\tau)\|_{\bV})} \right]^{\frac{1}{\alpha}} \nonumber
\\
&\land \delta_1\land \frac{\delta_2}{K+1}, \label{ini_path_conti}
\end{align}
where the constant $K>0$ is introduced in Lemma \ref{lemma3.1} (iv). Thus by Remark \ref{r2.2}, for almost all $\omega\in\Omega$ and all $\tau\le s< (\tau +h)\land\hat{\tau}_k$,
\begin{equation*}
\left| -\cL^{\overline{\theta}(s)}\phi(s,\xi_s) - f(s,\xi_s,\overline{\theta}(s)) + \cL^{\overline{\theta}(s)}\phi(s,\hat{\xi}_s) + f(s,\hat{\xi}_s,\overline{\theta}(s)) \right| \le \frac{\epsilon}{2}.
\end{equation*}
\end{enumerate}
Note that by Lemma \ref{lemma3.1} (iv), we have
 \begin{equation*}
 \max_{s\in [\tau,(\tau+h)\land \hat{\tau}_k]} \left\| X^{\tau,\xi_{\tau}\overline{\theta}}(s) - X^{\tau,\hat{\xi}_{\tau};\overline{\theta}}(s) \right\|_{\bH} \le K\|\xi - \hat{\xi}\|_{0,\bH}.
 \end{equation*}


By the dynamic programming principle, \eqref{pick_theta}, Lemma \ref{lemma3.4}, and the time continuous property of state process under $(\cA1)$ we have
\begin{align*}
0 \ge &\frac{1}{h}E_{\sF_{\tau}} \left[ (\phi-V)(\tau, \xi_{\tau}) - (\phi - V)( (\tau+h)\land\hat{\tau}_k, X^{\tau, \xi_{\tau}; \overline{\theta}}_{(\tau+h)\land\hat{\tau}_k} ) \right]
\\
= &\frac{1}{h}E_{\sF_{\tau}} \left[ (\phi-V)(\tau, \xi_{\tau}) - (\phi-V)(\tau, \hat{\xi}_{\tau}) + (\phi-V)(\tau, \hat{\xi}_{\tau}) - (\phi - V)( (\tau+h)\land\hat{\tau}_k, X^{\tau, \hat{\xi}_{\tau}; \overline{\theta}}_{(\tau+h)\land\hat{\tau}_k} ) \right]
\\
& -\frac{1}{h}E_{\sF_{\tau}} \left[  (\phi - V)( (\tau+h)\land\hat{\tau}_k, X^{\tau, \xi_{\tau}; \overline{\theta}}_{(\tau+h)\land\hat{\tau}_k} ) - (\phi - V)( (\tau+h)\land\hat{\tau}_k, X^{\tau, \hat{\xi}_{\tau}; \overline{\theta}}_{(\tau+h)\land\hat{\tau}_k} )  \right]
\\
\ge &-\frac{\epsilon}{2} + \frac{1}{h}E_{\sF_{\tau}} \left[ (\phi-V)(\tau, \hat{\xi}_{\tau}) - (\phi - V)( (\tau+h)\land\hat{\tau}_k, X^{\tau, \hat{\xi}_{\tau}; \overline{\theta}}_{(\tau+h)\land\hat{\tau}_k} ) \right]
\\
\ge &-\frac{\epsilon}{2} + \frac{1}{h}E_{\sF_{\tau}} \left[ \phi(\tau, \hat{\xi}_{\tau}) - \phi( (\tau+h)\land\hat{\tau}_k, X^{\tau, \hat{\xi}_{\tau}; \overline{\theta}}_{(\tau+h)\land\hat{\tau}_k} ) - \int_{\tau}^{(\tau+h)\land\hat{\tau}_k} \! f(s, X_s^{\tau, \hat{\xi}_{\tau}; \overline{\theta}}, \overline{\theta}(s))ds \right]
\\
= &-\frac{\epsilon}{2} + \frac{1}{h} E_{\sF_{\tau}} \int_{\tau}^{(\tau+h)\land{\hat{\tau}_k}} \! \left[ -\mathcal{L}^{\overline{\theta}(s)}\phi(s, X_s^{\tau, \hat{\xi}_{\tau}; \overline{\theta}}) - f( s, X_s^{\tau, \hat{\xi}_{\tau}; \overline{\theta}}, \overline{\theta}(s) ) \right]ds
\\
\ge &-\frac{\epsilon}{2} + \frac{1}{h} E_{\sF_{\tau}} \int_{\tau}^{(\tau+h)\land{\hat{\tau}_k}} \! \left[ -\mathcal{L}^{\overline{\theta}(s)}\phi(s, \xi_{s}) - f( s, \xi_{s}, \overline{\theta}(s) ) \right]ds
\\
&-\frac{1}{h} E_{\sF_{\tau}} \int_{\tau}^{(\tau+h)\land{\hat{\tau}_k}} \! \left| -\mathcal{L}^{\overline{\theta}(s)}\phi(s, \hat{\xi}_{s}) - f( s, \hat{\xi}_{s}, \overline{\theta}(s) ) + \mathcal{L}^{\overline{\theta}(s)}\phi(s, \xi_{s}) + f( s, \xi_{s}, \overline{\theta}(s) ) \right|ds
\\
& - \frac{1}{h} E_{\sF_{\tau}} \int_{\tau}^{(\tau+h)\land{\hat{\tau}_k}} \! \left| -\mathcal{L}^{\overline{\theta}(s)}\phi(s, \hat{\xi}_{s}) - f( s, \hat{\xi}_{s}, \overline{\theta}(s) ) + \mathcal{L}^{\overline{\theta}(s)}\phi(s, X_s^{\tau, \hat{\xi}_{\tau}; \overline{\theta}}) + f( s, X_s^{\tau, \hat{\xi}_{\tau}; \overline{\theta}}, \overline{\theta}(s) ) \right| ds
\\
\ge &-\frac{\epsilon}{2}+ 2\epsilon \cdot E_{\sF_{\tau}}\left[ \frac{(\tau+h)\land \hat\tau_k-\tau}{h} \right]\\
& -\frac{1}{h} E_{\sF_{\tau}} \int_{\tau}^{(\tau+h)\land{\hat{\tau}_k}} \!   \|X_s^{\tau,\hat{\xi}_{\tau};\overline{\theta}} - \hat{\xi}_s\|^{\alpha}_{0,\bV^*} ((L+1)(L+L_{\alpha}^{\delta})+c_3L_{\alpha}^{\delta}\|\hat{\xi}(s)\|_{\bV})  ds
\\
&-\frac{1}{h} E_{\sF_{\tau}} \int_{\tau}^{(\tau +h)\land\hat{\tau}_k} \! (L+L_{\alpha}^{\delta})(1+\rho)\|X_s^{\tau,\hat{\xi}_{\tau};\overline{\theta}} - \hat{\xi}_s\|_{0,\bH} + c_3\rho \|X^{\tau,\hat{\xi}_{\tau};\overline{\theta}}(s) - \hat{\xi}(s)\|_{\bV} ds
\\
\ge &-\frac{\epsilon}{2}+ 2\epsilon \cdot E_{\sF_{\tau}}\left[ \frac{(\tau+h)\land \hat\tau_k-\tau}{h} \right]\\
& -\frac{1}{h} E_{\sF_{\tau}} \int_{\tau}^{(\tau+h)\land{\hat{\tau}_k}} \!   K^{\alpha} (1+ \|\hat{\xi}\|_{0,\bH})^{\alpha} h^{\frac{\alpha}{2}} ((L+1)(L+L_{\alpha}^{\delta})+c_3L_{\alpha}^{\delta}\|\hat{\xi}(\tau)\|_{\bV})  ds
\\
&-\frac{1}{h} E_{\sF_{\tau}} \int_{\tau}^{(\tau +h)\land\hat{\tau}_k} \! (L+L_{\alpha}^{\delta})(1+\rho)\|X_s^{\tau,\hat{\xi}_{\tau};\overline{\theta}} - \hat{\xi}_s\|_{0,\bH} + c_3\rho \|X^{\tau,\hat{\xi}_{\tau};\overline{\theta}}(s) - \hat{\xi}(s)\|_{\bV} ds
\\
\ge &-\frac{\epsilon}{2}+ 2\epsilon \cdot E_{\sF_{\tau}}\left[ \frac{(\tau+h)\land \hat\tau_k-\tau}{h} \right]\\
&
-  E_{\sF_{\tau}} \Bigg[ K^{\alpha} (1+ \|\hat{\xi}\|_{0,\bH})^{\alpha} h^{\frac{\alpha}{2}} ((L+1)(L+L_{\alpha}^{\delta})+c_3L_{\alpha}^{\delta}\|\hat{\xi}(\tau)\|_{\bV})
\\
&\text{ }\text{ }\text{ }\text{ }\text{ }\text{ }\text{ }\text{ }\text{ }\text{ }\text{ }\text{ } - (L+L^{\delta}_{\alpha})(1+\rho)  K^{\frac{1}{2}}\cdot(1+\|A\hat{\xi}(\tau)\|_{\bH}) h
\\
&\text{ }\text{ }\text{ }\text{ }\text{ }\text{ }\text{ }\text{ }\text{ }\text{ }\text{ }\text{ } - \frac{c_3\rho}{\sqrt{c_2}h} \left(  E_{\sF_{\tau}} \left[ \int_{\tau}^{(\tau +h)\land\hat{\tau}_k} \! c_2\|X^{\tau,\hat{\xi}_{\tau};\overline{\theta}}(s) - \hat{\xi}(\tau)\|_{\bV}^2 ds  \right] \right)^{\frac{1}{2}}\cdot h^{\frac{1}{2}} \Bigg]
\\
\ge &-\frac{\epsilon}{2}+ 2\epsilon \cdot E_{\sF_{\tau}}\left[ \frac{(\tau+h)\land \hat\tau_k-\tau}{h} \right]\\
& - E_{\sF_{\tau}}\Bigg[ K^{\alpha} (1+ \|\hat{\xi}\|_{0,\bH})^{\alpha} h^{\frac{\alpha}{2}} ((L+1)(L+L_{\alpha}^{\delta})+c_3L_{\alpha}^{\delta}\|\hat{\xi}(\tau)\|_{\bV})
\\
&\text{ }\text{ }\text{ }\text{ }\text{ }\text{ }\text{ }\text{ }\text{ }\text{ }\text{ }\text{ } - (L+L^{\delta}_{\alpha})(1+\rho)  K^{\frac{1}{2}}(1+\|A\hat{\xi}(\tau)\|_{\bH}) h - \frac{c_3\rho}{\sqrt{c_2}} K^{\frac{1}{2}} \left( 1 + \| A\hat{\xi}(\tau) \|_{\bH}^2 \right)^{\frac{1}{2}} \cdot h^{\frac{1}{2}} \Bigg]
\\
\to &\frac{3\eps}{2} > 0 \text{, as } h\to 0^+. 
\end{align*}
This gives us a contradiction. Thus the value function $V$ is indeed a viscosity subsolution.

\textbf{Step 2.} To prove the value function $V$ is a viscosity supersolution, similarly we assume the opposite and argue with contradiction like what we did for proving viscosity subsolution.  Assume that for any $k\in\mathbb{N}^+$ with $k\ge K_0$ for some $K_0\in\mathbb{N}^+$, there exists $\phi\in\overline{\mathcal{G}}V(\tau, \xi_{\tau}; \Omega_{\tau}, k)$ with $\tau\in\mathcal{T}^0$, $\Omega_{\tau}\in\sF_{\tau}$, $P(\Omega_{\tau})>0$, and $\xi_{\tau}\in L^0(\Omega_{\tau}, \sF_{\tau}; \Lambda^0_{\tau}(\bV))$ such that there exist $\epsilon, \tilde{\delta}>0$, $\tau'\in \mathcal T^{\tau}_+$ and $\Omega'\in\sF_{\tau}$ with $\Omega'\subset \Omega_{\tau}$, $\Omega'\subset \{\tau<\tau'\}$, $P(\Omega')>0$, satisfying a.e. on $\Omega'$,
\begin{equation*}
\esssup_{  s\in(\tau, (\tau+\tilde{\delta}^2)\land \tau'], x\in B^{\bV^*}_{\tilde{\delta}}(\xi_{\tau})\cap\Lambda_{\tau, s\land T}^{0, k; \xi_{\tau}}(\bV)  }  \left\{  -\mathfrak{d}_s \phi(s,x) - \mathcal{H}(s, x, \nabla\phi(s,x))  \right\} \le - \epsilon.
\end{equation*}  

By Definition \ref{cf1}, for $\phi\in\mathcal C_{\sF}^1$, there is a partition $0=\underline t_0<\underline t_1<\ldots<\underline t_n=T$. W.l.o.g., we assume that there exists $\tilde{\delta}\in (0,1)$ such that $2 \tilde \delta ^2 <\min_{0\leq j\leq n-1} (\underline t_{j+1}-\underline t_j)$ and $\Omega'=\{[\tau,\tau+2\tilde{\delta}^2] \subset [\underline t_j,\underline t_{j+1})\}$ for some $j\in\{0,\ldots, n-1\}$. Also, w.l.o.g., we assume $\hat\tau_k=\tau'$ where $\hat{\tau}_k$ is the stopping time associated to the test function $\phi\in\overline{\mathcal{G}}V(\tau, \xi_{\tau}; \Omega_{\tau}, k)$.

For each $\theta\in\mathcal{U}$, define $\tau^{\theta} = \inf\{  s>\tau: X_s^{\tau, \xi_{\tau}; \theta}\notin B^{\bV^*}_{\tilde{\delta}}(\xi_{\tau})  \}$. Then for any $h\in(0, \frac{\tilde{\delta}^2}{4})$, by Chebyshev's inequality, we have
\begin{align*}
E_{\sF_{\tau}} [  1_{\tau^{\theta}<\tau+h}  ]
&=
      E_{\sF_{\tau}}\left[  1_{\max_{\tau\le s\le \tau+h}\|  X^{\tau, \xi_{\tau}; \theta}(s) - \xi(\tau) \|_{\bV^*} + \sqrt{h} > \tilde{\delta}  }  \right]
\\
&\le
      \frac{1}{(  \tilde{\delta} - \sqrt{h}  )^8} E_{\sF_{\tau}} \left[  \max_{\tau\le s\le \tau+h} \|  X^{\tau, \xi_{\tau}; \theta}(s) - \xi(\tau)  \|_{\bV^*}^8    \right]
\\
&\le
      \frac{K^8(1+\|\xi\|_{0,\bH})^8}{(  \tilde{\delta} - \sqrt{h}  )^8} \left( \sqrt{h}  \right)^8
\\
&\le
      \frac{256 K^8(1+\|\xi\|_{0,\bH})^8}{  \tilde{\delta}^8  } h^4
\\
&\le
      \tilde{C}h^4 \text{  a.s.}
\end{align*}
Here $K$ is the constant in Lemma \ref{lemma3.1}, being independent of the choice of control process $\theta$. Hence $\tilde{C}$ does not depend on the control process as well.

By the definition of $\overline{\mathcal{G}}V(\tau, \xi_{\tau}; \Omega_{\tau}, k)$, Lemma \ref{lemma3.4}, Theorem \ref{dynamic}, H\"older's inequality, Remark \ref{r2.2}, and the approximation above, we have for any $h\in(0, \frac{\tilde{\delta}^2}{4})$ and almost all $\omega\in\Omega_{\tau}$,
\begin{align*}
0
&=
      \frac{V(\tau, \xi_{\tau}) - \phi(\tau, \xi_{\tau})}{h}
\\
&=
      \frac{1}{h} \essinf_{\theta\in\mathcal{U}} E_{\sF_{\tau}} \left[  \int_{\tau}^{\hat{\tau}_k \land(\tau+h)} \! f(s,X_s^{\tau, \xi_{\tau}; \theta}, \theta(s))ds + V(  \hat{\tau}_k \land(\tau+h) ,X_{\hat{\tau}_k \land(\tau+h)}^{\tau, \xi_{\tau}; \theta}  ) - \phi(\tau, \xi_{\tau})  \right]
\\
&\ge
      \frac{1}{h} \essinf_{\theta\in\mathcal{U}} E_{\sF_{\tau}} \left[  \int_{\tau}^{\hat{\tau}_k \land(\tau+h)} \! f(s,X_s^{\tau, \xi_{\tau}; \theta}, \theta(s))ds + \phi(  \hat{\tau}_k \land(\tau+h) ,X_{\hat{\tau}_k \land(\tau+h)}^{\tau, \xi_{\tau}; \theta}  ) - \phi(\tau, \xi_{\tau})  \right]
\\
&=
      \frac{1}{h} \essinf_{\theta\in\mathcal{U}} E_{\sF_{\tau}}\left[  \int_{\tau}^{\hat{\tau}_k \land(\tau+h)}  \! (  \mathcal{L}^{\theta(s)} \phi(s, X_s^{\tau, \xi_{\tau}; \theta}) + f(s,X_s^{\tau, \xi_{\tau}; \theta}, \theta(s)) )ds  \right]
\\
&\ge
      \frac{1}{h} \essinf_{\theta\in\mathcal{U}} E_{\sF_{\tau}}\Bigg[  \int_{\tau}^{\tau^{\theta} \land(\tau+h)\land \hat{\tau}_k}  \! \left(  \mathcal{L}^{\theta(s)} \phi(s, X_s^{\tau, \xi_{\tau}; \theta}) + f(s,X_s^{\tau, \xi_{\tau}; \theta}, \theta(s)) \right)ds 
      \\
      &\text{ }\text{ }\text{ }-1_{  \{ \hat{\tau}_k>\tau^{\theta} \} \cap \{ \tau+h>\tau^{\theta} \}} \int_{\tau}^{(\tau+h)\land \hat{\tau}_k} \left|  \mathcal{L}^{\theta(s)} \phi(s, X_{s}^{\tau, \xi_{\tau}; \theta}) +  f(s,X_s^{\tau, \xi_{\tau}; \theta}, \theta(s)) \right| ds \Bigg]
\\
&\ge
      \frac{1}{h} \essinf_{\theta\in\mathcal{U}} E_{\sF_{\tau}}\Bigg[  \int_{\tau}^{(\tau+h)\land \hat{\tau}_k}  \! \left(  \mathcal{L}^{\theta(s)} \phi(s, X_{s\land\tau^{\theta}}^{\tau, \xi_{\tau}; \theta}) + f(s,X_{s\land\tau^{\theta}}^{\tau, \xi_{\tau}; \theta}, \theta(s)) \right)ds 
      \\
      &\text{ }\text{ }\text{ }-1_{  \{ \hat{\tau}_k>\tau^{\theta} \} \cap \{ \tau+h>\tau^{\theta} \}} \int_{\tau}^{(\tau+h)\land \hat{\tau}_k} \left|  \mathcal{L}^{\theta(s)} \phi(s, X_{s\land\tau^{\theta}}^{\tau, \xi_{\tau}; \theta}) +  f(s,X_{s\land\tau^{\theta}}^{\tau, \xi_{\tau}; \theta}, \theta(s)) \right|ds
             \\ 
             &\text{ }\text{ }\text{ }-1_{  \{ \hat{\tau}_k>\tau^{\theta} \} \cap \{ \tau+h>\tau^{\theta} \}} \int_{\tau}^{(\tau+h)\land \hat{\tau}_k} \left|  \mathcal{L}^{\theta(s)} \phi(s, X_{s}^{\tau, \xi_{\tau}; \theta}) +  f(s,X_{s}^{\tau, \xi_{\tau}; \theta}, \theta(s)) \right| ds \Bigg]
\\
&\ge
       E_{\sF_{\tau}}\left[  \frac{ ( \hat{\tau}^k \land (\tau+h) ) - \tau }{h} \right]\cdot \epsilon
       - \frac{1}{h} \esssup_{\theta\in\mathcal{U}} \left(  E_{\sF_{\tau}} \left[  1_{\{\tau+h>\tau^{\theta}\}}  \right] \right)^{\frac{1}{2}} 
      \\
      &\text{ }\text{ }\text{ }\cdot 
      \left(  E_{\sF_{\tau}} \left|  \int_{\tau}^{(\tau+h)\land \hat{\tau}_k} \left|  \mathcal{L}^{\theta(s)} \phi(s, X_{s\land\tau^{\theta}}^{\tau, \xi_{\tau}; \theta}) +  f(s,X_{s\land\tau^{\theta}}^{\tau, \xi_{\tau}; \theta}, \theta(s)) \right|ds  \right|^2  \right)^{1/2}
      \\
      &\text{ }\text{ }\text{ } -\frac{1}{h}\esssup_{\theta\in\mathcal{U}}\left(  E_{\sF_{\tau}} \left[  1_{\{  \tau+h>\tau^{\theta}  \}}  \right]  \right)^{1/2} 
      \\
      &\text{ }\text{ }\text{ }\cdot \left(  E_{\sF_{\tau}} \left|  \int_{\tau}^{(\tau+h)\land \hat{\tau}_k} \left|  \mathcal{L}^{\theta(s)} \phi(s, X_{s}^{\tau, \xi_{\tau}; \theta}) +  f(s,X_{s}^{\tau, \xi_{\tau}; \theta}, \theta(s)) \right| ds  \right|^2  \right)^{1/2}
\\
&\ge
       E_{\sF_{\tau}}\left[  \frac{ ( \hat{\tau}^k \land (\tau+h) ) - \tau }{h} \right]\cdot \epsilon 
      - 2h \tilde{C}^{1/2}\left(  E_{\sF_{\tau}} \left[  \int_{\tau}^{(\tau+h)\land\hat{\tau}_k} \! \left|  \zeta_s^{\phi} + c_3\rho \|X^{\tau,\xi_{\tau};\theta}(s)\|_{\bV} \right|^2 ds  \right]  \right)^{1/2}
\\
&\ge
           E_{\sF_{\tau}}\left[  \frac{ ( \hat{\tau}^k \land (\tau+h) ) - \tau }{h} \right]\cdot \epsilon
      - 2\sqrt{2}h \tilde{C}^{1/2} 
      \\
      &\text{ }\text{ }\text{ }\cdot\left(  E_{\sF_{\tau}} \left[  \int_{\tau}^{(\tau+h)\land\hat{\tau}_k} \! \left|  \zeta_s^{\phi} \right|^2 ds + \frac{c_3^2\rho^2}{c_2} \int_{\tau}^{(\tau+h)\land\hat{\tau}_k} \! c_2\|X^{\tau,\xi_{\tau};\theta}(s)\|_{\bV}^2 ds  \right]  \right)^{1/2}
\\
&\ge
     E_{\sF_{\tau}}\left[  \frac{ ( \hat{\tau}^k \land (\tau+h) ) - \tau }{h} \right]\cdot \epsilon
     \\
     &\quad
      - 2\sqrt{2}h \tilde{C}^{1/2}\left(  E_{\sF_{\tau}} \left[  \int_{\tau}^{(\tau+h)\land\hat{\tau}_k} \! \left|  \zeta_s^{\phi} \right|^2 ds + \frac{c_3^2\rho^2 K^2}{c_2} \left( 1 + \|\xi\|_{0,\bH}^2 \right)  \right]  \right)^{1/2}
\\
&\to
      \epsilon, \text{ as } h\to 0^+,
\end{align*}
Hence a contradiction occurs. Then the value function $V$ is proved to be a viscosity supersolution.


\end{proof}

\begin{rmk}
By Theorem \ref{state_sol_wellposedness}, the state process $X$ is continuous in $\bH$ but only $L^2$-integrable in $\bV$. Then the lack of uniform boundedness and continuity of $X$ in space $\bV$ make the H\"older continuity discussed in Remark \ref{r2.2} inapplicable in the proof in Step 1. To overcome it, we use the following assumption for the existence of viscosity subsolution
\begin{equation*}
\bV^A := \left\{ \eta\in\bV: A\eta\in\bH \right\} \text{ is dense in } \bV.
\end{equation*}
In fact, most of the differential operators we have encountered in partial differential equation theories, in particular the Laplacian, satisfy the above assumption (See Example \ref{example}, where $\bV^A = W_{0}^{2,2}(\mathcal{O})$.).

This is different from the finite dimensional case studied in \cite{qiu2022controlled}. However, this assumption is not necessary in proving the existence of viscosity supersolution. This is because the definition of the Hamilton function \eqref{Hamilton} imposes a one-sided preference of the inequalities, making the proof of the existence of viscosity subsolution harder than that of the viscosity supersolution.

\end{rmk}

\end{thm}

\section{Uniqueness}

\subsection{A weak comparison principle}
\begin{prop}\label{prop-cor-cmp}
Assume $(\cA1)$ holds and $u$ is a viscosity subsolution (resp. supersolution)  of SPHJ equation \eqref{SPHJ}. Then there is an infinite sequence of integers $1\leq \underline k_1<\underline k_2<\cdots<\underline k_n<\cdots$ (resp., $1\leq \overline k_1<\overline k_2<\cdots<\overline k_n<\cdots$), such that for each $i\in\bN^+$, $x(0)\in \bV$, $\phi_i\in\mathcal C^1_{\sF}$ satisfying 
$\phi_i(T,x)\geq (\text{resp. }\leq) G(x)$ for all $x\in\Lambda_T^0(\bH)$
\text{ }a.s. and  
\begin{align*}
\text{ess}\liminf_{(s,x)\rightarrow (t^+,y)}
	  \left\{  -\mathfrak{d}_{s}\phi_i(s,x)-\cH(s,x,\nabla \phi_i(s,x) ) \right\} \geq 0, \text{ a.s.,}
\\
\text{(resp. }
\text{ess}\limsup_{(s,x)\rightarrow (t^+,y)}
	 \left\{  -\mathfrak{d}_{s}\phi_i(s,x)-\cH(s,x,\nabla \phi_i(s,x) ) \right\} \leq 0\text{, a.s.)}
\end{align*}
for  each $t\in[0,T)$ with $y\in\Lambda_{0,t}^{0,\underline k_i;x(0)}(\bV)$ (resp., $y\in\Lambda_{0,t}^{0,\overline k_i;x(0)}(\bV)$), it holds a.s. that $u(t,x)\leq$ (resp., $\geq$) $\phi_i(t,x)$,  for each $t\in[0,T]$ with $x\in\Lambda_{0,t}^{0,\underline k_i;x(0)}(\bV)$ (resp. $x\in\Lambda_{0,t}^{0,\overline k_i;x(0)}(\bV)$).
\end{prop}
\begin{proof}
We only need to prove the case when $u$ is a viscosity supersolution. For the viscosity subsolution the proof will be similar. 

Given $i\in\bN^+$, we assume that $u(t,\bar x_t)<\phi_i(t,\bar x_t)$ with a positive probability  at some point $(t,\bar x_t) $ with $t\in[0,T)$ and $\bar x_t\in \Lambda^{0,\overline k_i;x(0)}_{0,t}(\bV)$ for some $x(0) \in \bV$. 

Note that we do not need to include $t=T$ since, by Definition \ref{def_vis_sol}, the terminal condition clearly disagrees with our assumption here. Without any loss of generality, we assume $x(0)=0$. Then, there exists $\delta>0$ such that $\bP( \overline\Omega_t)>0$ with $\overline\Omega_t:=\{\phi_i(t,\bar x_t)-u(t,\bar x_t) >\delta  \}$.  




We know $\Lambda_{t,s}^{0,k;\xi}(\bV)$ is compactly embedded in the space $\Lambda_s^0(\bH)$. Then applying the measurable selection theorem, there exists $\xi^{\overline k_i}_t\in L^0\left(\overline\Omega_t,\sF_t;\Lambda^{0,{\overline k_i};0}_{0,t}(\bV) \right)$ such that
$$
\phi_i(t,\xi^{\overline k_i}_t)- u(t,\xi^{\overline k_i}_t) =\max_{x_t\in   \Lambda^{0,{\overline k_i};0}_{0,t}(\bV) } \left\{  \phi_i(t,x_t)-u(t,x_t) \right\}\geq \delta\text{ for almost all }\omega\in\overline\Omega_t.
$$ 
   W.l.o.g., we take $\overline\Omega_t=\Omega$ in what follows. For each $s\in(t,T]$, choose an $\sF_s$-measurable and $ \Lambda^{0,{\overline k_i};\xi^{\overline k_i}_t}_{t,s}(\bV)$-valued variable $\xi^{\overline k_i}_s$ such that  
\begin{align}
\left(\phi_i(s,\xi^{\overline k_i}_s)- u(s,\xi^{\overline k_i}_s)  \right)^+=\max_{x_s\in \Lambda^{0,{\overline k_i};\xi^{\overline k_i}_t}_{t,s}(\bV)   }   \left( \phi_i(s,x_s)-u(s,x_s) \right)^+ , \quad \text{a.s.,}\label{eq-maxima}
\end{align}  
and set
$Y^{\overline k_i}(s)
=
	(\phi_i(s,\xi^{\overline k_i}_s)-u(s,\xi^{\overline k_i}_s) )^+ +\frac{\delta (s-t)}{3(T-t)}
	$, and 
$Z^{\overline k_i}(s)
= 
	\esssup_{\tau\in\cT^s} E_{\sF_s} [Y^{\overline k_i}({\tau})]$.
Here, recall that  $\mathcal{T}^s$ denotes the set of stopping times valued in $[s,T]$. As $(\phi_i-u)^+\in \cS^2(\Lambda^0_{T}(\bH);\bR)$, there follows obviously the time-continuity of  
$$
	\max_{x_s\in \Lambda^{0,{\overline k_i};\xi^{\overline k_i}_t}_{t,s}(\bV)  }   \left(\phi_i(s,x_s)-u(s,x_s) \right)^+,
$$ 
and thus that of
$\left( \phi_i(s,\xi^{\overline k_i}_s)-u(s,\xi^{\overline k_i}_s) \right)^+$, although the continuity of process $(\xi^{\overline k_i}_s)_{s\in [t,T]}$ (as path space-valued process) in space $\bV$ might not be ensured. Therefore, the process $(Y^{\overline k_i}(s))_{t\leq s \leq T}$ has continuous trajectories. 
Define $\tau^{\overline k_i}=\inf\{s\geq t:\, Y^{\overline k_i}(s)=Z^{\overline k_i}(s)\}$ as a stopping time. In view of the terminal condition, it is obvious that 
\begin{equation*}
E_{\sF_t} \left[ Y^{\overline k_i}(T) \right] = \frac{\delta}{3} < \delta \le Y^{\overline{k}_i}(t).
\end{equation*}

Then by definition of $Z^{\overline{k}_i}(t)$ and the fact that $Y^{\overline{k}_i}(t)$ is $\sF_t$-measurable, we have
\begin{equation*}
Z^{\overline{k}_i}(t) \ge E_{\sF_t} \left[  Y^{\overline{k}_i}(t)  \right] = Y^{\overline{k}_i}(t).
\end{equation*}
Further with the optimal stopping theory, we have
$$
E_{\sF_t}Y^{\overline k_i}(T)
=\frac{\delta}{3}
	<\delta \leq 
	Y^{\overline k_i}(t) \leq Z^{\overline k_i}(t)=E_{\sF_t} \left[ Y^{\overline k_i}({\tau^{\overline k_i}}) \right] =E_{\sF_t}\left[Z^{\overline k_i}({\tau^{\overline k_i}})\right],
$$
which gives that $\bP(\tau^{\overline k_i}<T)>0$. As 
\begin{equation}
	\left(\phi_i(\tau^{\overline k_i},\xi^{\overline k_i}_{\tau^{\overline k_i}}) - u(\tau^{\overline k_i},\xi^{\overline k_i}_{\tau^{\overline k_i}}) \right)^+ 
	+\frac{\delta (\tau^{\overline k_i}-t)}{3(T-t)}
=Z^{\overline k_i}({\tau^{\overline k_i}}) \geq E_{\sF_{\tau^{\overline k_i}}}\left[Y^{\overline k_i}(T)\right]= \frac{\delta}{3},
\label{relation-tau}
\end{equation}
we have 
\begin{align*}
\bP\left(( \phi_i(\tau^{\overline k_i},\xi^{\overline k_i}_{\tau^{\overline k_i}}) - u(\tau^{\overline k_i},\xi^{\overline k_i}_{\tau^{\overline k_i}}))^+>0\right)>0. 
\end{align*}
 Define 
$$\hat\tau^{\overline k_i}=\inf\left\{s\geq\tau^{\overline k_i}:\, (\phi_i(s,\xi^{\overline k_i}_s)-u(s,\xi^{\overline k_i}_{s})  )^+\leq 0\right\}.$$ Obviously, $\tau^{\overline k_i}\leq\hat\tau^{\overline k_i}\leq T$.
 Put $\Omega_{\tau^{\overline k_i}}=\left\{\tau^{\overline k_i}<\hat\tau^{\overline k_i}\right\}$. Then $\Omega_{\tau^{\overline k_i}}\in \sF_{\tau^{\overline k_i}}$ and in view of relation \eqref{relation-tau}, and the definition of $\hat\tau^{\overline k_i}$, we have $\Omega_{\tau^{\overline k_i}} =\left\{\tau^{\overline k_i}<T\right\}$ and  $\bP(\Omega_{\tau^{\overline k_i}})>0$.

Set 
$$\Phi_i(s,x_s)=\phi_i(s,x_s)  +\frac{\delta (s-t)}{3(T-t)}-E_{\sF_s}Y^{\overline k_i}({\tau^{\overline k_i}})
.$$
 Then $\Phi_i\in\mathcal{C}^1_{\sF}$ since $\phi_i \in \mathcal{C}^1_{\sF}$. For each $\bar\tau\in\cT^{\tau^{\overline k_i}}$, 
 we have for almost all $\omega\in\Omega_{\tau^{\overline k_i}}$, 
\begin{align*}
\left(\Phi_i-u\right)(\tau^{\overline k_i},\xi^{\overline k_i}_{\tau^{\overline k_i}})
=0=Z^{\overline k_i}({\tau^{\overline k_i}})-Y^{\overline k_i}({\tau^{\overline k_i}})
&\geq 
E_{\sF_{\tau^{\overline k_i}}}\left[Y^{\overline k_i}({\bar\tau\wedge \hat{\tau}^{\overline k_i}}) \right] - Y^{\overline k_i}({\tau^{\overline k_i}})
\\
&\geq 
 E_{\sF_{\tau^{\overline k_i}}} \left[ \max_{y\in  \Lambda^{0,{\overline k_i};  \xi^{\overline k_i}_{\tau^{\overline k_i}}}_{\tau^{\overline k_i}, \bar\tau\wedge\hat\tau^{\overline k_i}}(\bV)} (\Phi_i-u)(\bar\tau\wedge\hat\tau^{\overline k_i},y)   \right],
\end{align*}
where we have used the obvious relation $\Lambda^{0,{\overline k_i};  \xi^{\overline k_i}_{\tau^{\overline k_i}}}_{\tau^{\overline k_i}, \bar\tau\wedge\hat\tau^{\overline k_i}}(\bV) \subset \Lambda^{0,{\overline k_i};\xi^{\overline k_i}_t}_{t, \bar\tau\wedge\hat\tau^{\overline k_i} }(\bV)$.
This together with the arbitrariness of $\bar\tau$ implies that $\Phi_i\in \overline{\cG} u\left(\tau^{\overline k_i},\xi^{\overline k_i}_{\tau_{\overline k_i}};\Omega_{\tau^{\overline k_i}},{\overline k_i}\right)$. In view of the correspondence between the viscosity supersolution $u$ and the infinite sequence $\left\{{\overline k_1}, \overline k_2,\cdots,\overline k_n,\cdots\right\}$ and Definition \ref{def_vis_sol},   we have for almost all $\omega\in\Omega_{\tau^{{\overline k_i}}}$, 
\begin{align*}
0
&\leq
	 \text{ess}\limsup_{(s,x)\rightarrow ((\tau^{\overline k_i})^+,\,\, \xi^{{\overline k_i}}_{\tau^{{\overline k_i}}})}
	  \left\{ -\mathfrak{d}_{s}\Phi_i(s,x)-\cH(s,x,\nabla \Phi_i(s,x) ) \right\}
\\
&=
	-\frac{\delta}{3(T-t)}
	+ \text{ess}\limsup_{(s,x)\rightarrow ((\tau^{{\overline k_i}})^+,\,\, \xi^{{\overline k_i}}_{\tau^{{\overline k_i}}})}
	  \left\{  
	-\mathfrak{d}_{s}\phi_i(s,x)-\cH(s,x,\nabla \phi_i(s,x)  ) \right\}
\\
&\leq -\frac{\delta}{3(T-t)} ,
\end{align*}
which is a contradiction.
\end{proof}

\begin{rmk}\label{rmk4.1}
This is an even weaker version of the weak comparison principle than that introduced in the finite dimensional case (see \cite[Proposition 4.1]{qiu2022controlled}). Under the infinite dimensional framework, we need to fix the initial state $x(0)$ in order to deal with the lack of compactness of the path space when taking the essential supremum in Lemma \ref{lem-approx} below. The proof of uniqueness of the viscosity solution will need to be adjusted accordingly. 
\end{rmk}

\subsection{Uniqueness}

Since $\bH$ is an infinite dimensional separable Hilbert space, it has a standard orthogonal basis $\{e_i\}_{i\in \bN^+}$ such that $\bH = \text{span}\{ e_1, e_2, e_3, \cdots \}$. Thus for any $h\in \bH$, its $d$-dimensional projection is defined as
\begin{equation*}
(^d{P})h = \sum_{i=1}^d  \langle e_i,  h \rangle \cdot e_i, \text{ }\text{ } d\in\bN^+.
\end{equation*}
Of course, such definition can be extended to space $\bV^*$ by defining
\begin{equation*}
{_{\bV^*}\langle} (^d{P})h^*, h \rangle_{\bV} = {_{\bV^*}\langle} h^*, (^d{P})h \rangle_{\bV},
\end{equation*}
for any $h\in \bV$, $h^*\in \bV^*$.

It is critical to point out that although by denseness, for any $i\in\bN^+$, $e_i\in \bV$ holds, $\{e_i\}_{i\in\bN^+}$ might not be the standard orthogonal basis for space $\bV$ because $\bV$ is only a Banach space, and for any $i\in\bN^+$, $\|e_i\|_{\bV} = 1$ does not necessarily hold. Let $  (^d{P})Ax(t) $, $(^d{P}) \beta(t, x_t, v)$ be such projection of $ Ax(t)$, $\beta(t, x_t, v)$ for any $v\in U$, $t\in [0,T]$, $x(t)\in \bV$, and $x_t\in\Lambda_t(\bV)$. Notice here we need to consider $Ax(\cdot)$ as a whole when we are computing the finite dimensional projection. Likewise, for the path dependence of $f(t,\cdot,v), \beta(t,\cdot,v), G(\cdot)$, we choose not to do projection of the path inside these functions to avoid changes in these functions every time the projection dimension $d$ is changed. Later we will see that the projection of the path process itself is not necessary.



Indeed one can argue that for any $t\in [0,T]$, $x(t)\in \bV$, the finite dimensional projection of the coefficient $Ax(t)$, which originally appears in the infinite dimensional cases, will then be combined with the projection of $\beta$. It does not hurt if we keep the projection of this term from the beginning in $\hat{\bH}^d$, where $\hat{\bH}^d := \text{span}\{e_1, e_2, \cdots, e_d\}$. Then the finite dimensional approximation method would ensure us a similar result in the infinite dimensional case.


We need to further introduce the following assumptions.

$(\cA 2)$ The linear operator $A$ satisfies
\begin{enumerate}
\item [(i)] $(^d{P})A = A (^d{P})$, for any $d\in\bN^+$;
\item [(ii)] there exists $L>0$ such that for all $x_T,\overline{x}_T\in \Lambda_T(\bV^*)$, and $t\in [0,T]$, there holds
\begin{align*}
&\esssup_{\omega\in\Omega} | G(x_T)-G(\overline{x}_T) | + \esssup_{\omega\in\Omega} \max_{v\in  U} | f(t, x_t, v)-f(t, \overline{x}_t, v) | 
\\
&+\esssup_{\omega\in\Omega} \max_{v\in  U} \| \beta(t, x_t, v)-\beta(t, \overline{x}_t, v) \|_{\bV^*} \le L(\| x_T-\overline{x}_T \|_{0,\bV^*} + \|x_t - \overline{x}_t\|_{0,\bV^*}).
\end{align*}
\item [(iii)] for each $v\in U$, $f(\cdot,\cdot,v)\in\cS^{\infty}\left( \Lambda(\bH);\bR \right)$ and $\beta(\cdot,\cdot,v)\in\cS^{\infty}\left( \Lambda(\bH);\bH \right)$.
\end{enumerate}
\begin{rmk}
In addition to the uniform Lipschitz assumption in $(\cA 1)$, introducing Assumption $(\cA 2)$ (ii) eliminates certain cases, for instance in Example \ref{example} when $\tilde{\beta}(\tilde{X}_t) = C\cdot D \tilde{X}(t)$, with $C$ being a constant vector and $D$ the usual gradient operator. This issue can be avoided if we further introduce a new Gelfand triple $\bV_1 \subseteq \bV\subseteq \bV_1^*$, where $\bV_1$ is another reflexive Hilbert space and is continuously, densely, and compactly embedded into space $\bV$. In that case, in addition to Example \ref{example}, we could further set $\bV_1:= W_0^{2,2}(\cO)$ and $\bV_1^*:=L_0^2(\cO)$ as an example.  
\end{rmk}

\begin{lem}\label{lem-approx}
 Let $(\cA1)$ and $(\cA 2)$ hold. For each $\eps>0$, there exist $d\in\bN^+$ and a partition $0=t_0<t_1<\cdots<t_{N-1}<t_N=T$ for some $N>3$ and functions 
\begin{align*}
(G^N,f^N, \beta^N) \in C^{1}(\bR^{ mN} \times \bH^{N+1};\bR)\times C( U;C^1([0,T]\times\bR^{ mN} \times \bH^{N+1};\bR)) 
\\
\times C(U;C^1([0,T]\times\bR^{ mN} \times \bH^{N+1};\bH))
\end{align*}

such that given $x(0)\in\bH$, for each $k \in \bN^+$, 
\begin{align*}
	&f^{\eps}_k(t) 
	\\
	:=&\esssup_{(x,v)\in \Lambda_{0,t}^{0,k;x(0)}(\bV)\times U}
		\left| f^N( W({t_1\wedge t}),\ldots, W({t_N\wedge t}),t,x({t_0\wedge t}),\ldots, x({t_N\wedge t}),v)-f(t,x,v)\right|,
\\
 	&\beta^{\eps}_k(t) 
	\\
	:= &\esssup_{(x,v)\in \Lambda_{0,t}^{0,k;x(0)}(\bV) \times U}
		\left\| ^d{P}\beta^N ( W({t_1\wedge t}),\ldots, W({t_N\wedge t}),t,x({t_0\wedge t}),\ldots, x({t_N\wedge t}),v) - \beta(t,x,v)\right\|_{\bV^*},
\\
&G^{\eps}_k
\\:=&\esssup_{x\in \Lambda_{0,T}^{0,k;x(0)}(\bV)} \left|G^N( W({t_1}),\ldots, W({t_N}),x(t_0),\ldots,x(t_N))-G(x) \right| \text{ are }\sF_t\text{-adapted with}				
\end{align*}
\begin{align}
	\left\| G^{\eps}_k  \right\|_{L^2(\Omega,\sF_T;\bR)} + \left\| f^{\eps}_k  \right\|_{L^2(\Omega\times[0,T];\bR)}   + \left\| \beta^{\eps}_k  \right\|_{L^2(\Omega\times[0,T];\bR)}    <\eps(1+k)\left( 1 + \|x(0)\|_{\bH} \right).
	\label{approx-error}
\end{align}
Moreover, $G^N$, $f^N$, and $(^{d}P)\beta^N$ are uniformly Lipschitz-continuous in the space variable $x$ in $\bH$ with an identical Lipschitz-constant $L_c$ independent of $N$, $k$, $d$, and $\eps$.
\end{lem}
Note that since $\beta(t, x, v)$ are of infinite dimensions, the proof of Lemma \ref{lem-approx} is not as trivial as the finite dimensional case (\cite[Lemma 4.2]{qiu2022controlled}). Not only do we need to perform approximations of $\beta(t,x,v)$ regarding the path dependence applying the partition in $t$, we are also using finite dimensional projection of the infinite dimensional functions and thus a uniform convergence of such projections is required to make sense of the finite dimensional approximation method. As for the coefficients $f$ and $G$, the uniform convergence is also required because of the infinite dimensional path dependence, regardless of the fact that these coefficients themselves take value in $\bR$.

  
\begin{proof}[Sketched proof of Lemma \ref{lem-approx}]
We consider the approximations for the function $f$ first. For any $l\in\bN^+$, let us introduce the time partition of $[0,T]$ as $0=t_0\le t_1\le \cdots \le t_l = T$. Then for any $\omega\in\Omega$, $t\in [0,T]$, $v\in U$, $x_t\in\Lambda_{0,t}^{0,k;x(0)}(\bV)$, by assumption $(\cA 2)$ (iii) and dominated convergence theorem, the following approximation works as $l\to\infty$
\begin{equation*}
f(\omega,t,x_t,v) \sim f(\omega,0,x_0,v) 1_{[0,t_1]}(t) + \sum_{j=1}^{l-1} f(\omega,t_j,x_{t_j},v) 1_{(t_j,t_{j+1}]}(t), 
\end{equation*}
where the time-continuity introduced in assumption $(\cA 2)$ (iii) can be further relaxed to left continuity with right limit.

Note here the right hand side is not time continuous. We will work on this issue later. Next, by \cite[Lemma 1.3]{da2014stochastic}, for any $j\in\bN^+$, $\omega\in\Omega$, $t_j\in [0,T]$, $x_{t_j}\in\Lambda_{0,t_j}^{0,k;x(0)}(\bV)$, and $v\in U$, we have
\begin{equation*}
f(\omega,t_j,x_{t_j},v) \sim \sum_{i=1}^{l_j} 1_{A_i^j}(\omega) \tilde{f}_i^j(x_{t_j},v), 
\end{equation*}
with $A_i^j\in\sF_{t_j}$, $i=1,2,\cdots,l_{t_j}$, $\tilde{f}_i^j \in C\left(U; C^1\left(\Lambda_{0,t_j}^{0,k;x(0)}(\bV)\right)\right)$.
For each $j\in\bN^+$, $A_i^j$ are not necessarily disjoint. Of course, since by $(\cA 1)$ $f$ is uniformly bounded and Lipschitz continuous in the path variable, $\tilde{f}_i^j$ functions can be chosen to satisfy such assumptions as well, i.e. for any $v\in U$ and any $t\in [0,T]$, $x_t,y_t \in \Lambda_{0,t}^{0,k;x(0)}(\bV)$ with $y(0)=x(0)$,
\begin{align*}
\left| \tilde{f}_i^j(v,x_t) \right| &\le L+1,
\\
\left|\tilde{f}_i^j(v,x_t) - \tilde{f}_i^j(v,y_t) \right| &\le (L+1)\| x_t - y_t \|_{0,\bV^*}.
\end{align*} 
Then, by \cite[Lemma 4.3.1]{oksendal2003stochastic}, we have the following approximation in $L^2(\Omega,\sF_{t_j})$
\begin{equation*}
1_{A_i^j}(\omega) \sim g_i^j\left( W(\tilde{t}_1),\cdots, W(\tilde{t}_{l_i^j}) \right),
\end{equation*}
where $g_i^j\in C_c^{\infty} \left( \bR^{m\times l_i^j} \right)$, $\tilde{t}_r \in [0,t_j]$, $r = 1,2,\cdots,l_i^j$.

For the approximation of $\tilde{f}_i^j$ regarding the path dependence, let us introduce the following path approximation. For any $t_j\in [0,T]$ and $s\in [0,t_j]$, the following point-wise convergence holds
\begin{equation*}
x_{t_j}(s) \sim (P^M)(x_{t_j})(s),
\end{equation*} 
where
\begin{equation}\label{state appro}
(P^M)(x_{t_j})(s) := \sum_{n=1}^{2^M} x_{t_j} \left( \frac{(n-1)t_j}{2^M} \right) 1_{\left[ \frac{(n-1)t_j}{2^M} , \frac{nt_j}{2^M} \right)}(s) + x_{t_j}(t_j) 1_{\{t_j\}}(s).
\end{equation}
Then by $(\cA 2)$ and Lemma 3.2 $(iii)$, we have for any $t_j\in [0,T]$ and any $v\in U$,
\begin{align}
&\left| \tilde{f}_i^j(v, x_{t_j}) - \tilde{f}_i^j(v,(P^M)(x_{t_j})) \right| \nonumber
\\
\le &\left\{(L+1)\left\| x_{t_j} - (P^M)(x_{t_j}) \right\|_{0,\bV^*} \right\} \nonumber
\\
= &\left\{(L+1) \sup_{s\in[0,t_j]} \left\| x_{t_j}(s) - (P^M)(x_{t_j})(s) \right\|_{\bV^*} \right\} \nonumber
\\
= &\left\{ (L+1) \sup_{n} \sup_{s\in \left[ \frac{(n-1)t_j}{2^M}, \frac{nt_j}{2^M} \right)} \left\| x_{t_j}(s) - x\left(\frac{(n-1)t_j}{2^M}\right) \right\|_{\bV^*} \right\} \nonumber
\\
\le &\left\{ (L+1)K\cdot \left(\frac{t_j}{2^M}\right)^{\frac{1}{2}} (1+\|x(0)\|_{\bH}) \right\} \label{f_uniform}
\\
\to &0, \text{ as } M\to\infty\nonumber.
\end{align}

In addition, for any $t\in [0,T]$, $1_{(t_j,t_{j+1}]}(t)$ can be increasingly approximated by compactly-supported smooth functions $\phi^j(t)$.

Let $l=N$. To sum up, the function $f$ may be approximated by functions of the following form:
\begin{align*}
&f^N( W(t_1 \land t),\cdots, W(t_N \land t), t, x(0), x(t_1 \land t), \cdots, x(t_N \land t), v )
\\
= &\sum_{j=1}^{N} \sum_{i=1}^{N_j} g_i^j \left( W(t_1), \cdots, W(t_j) \right) \tilde{f}_i^j(v, x(0), x(t_1\land t), \cdots, x(t_j\land t)) \phi^j \left(t\right).
\end{align*}
The required approximation for $G$ can be obtained in a similar but simplified way since it does not have time dependence. Note the approximation error above is given by $\eps (1+k)(1+\|x(0)\|_{\bH})$ and the $k$-dependence comes from the path space.

Therefore it is easy to prove that for any $x_t \in \Lambda_{0,t}^{0,k;x(0)}(\bV)$ and any $N>1$, there exists $\eps>0$ such that
\begin{equation*}
\| G_k^{\eps} \|_{L^2(\Omega,\sF_T;\bR)} + \| f_k^{\eps} \|_{L^2(\Omega\times [0,T];\bR)} < \eps(1+k)(1+\|x(0)\|_{\bH}).
\end{equation*}
For the function $\beta$, notice for any $x_t \in \Lambda_{0,t}^{0,k;x(0)}(\bV)$, $\beta$ is uniformly bounded in $\bH$, which is dense and compactly embedded in $\bV^*$. Denote by $I$ the $identity\text{ }operator$ in $\bV^*$, we have
\begin{align*}
\| (^d{P})\beta^N - \beta \|_{\bV^*} &= \| (^d{P})\beta^N - (^d{P})\beta + (^d{P} - I)\beta \|_{\bV^*}
\\
& \le \| (^d{P})\beta^N - (^d{P})\beta \|_{\bV^*} + \| (^d{P} - I)\beta \|_{\bV^*}
\\
& \le \| \beta^N - \beta \|_{\bV^*} + \| (^d{P} - I)\beta \|_{\bV^*}
\\
& = B_1 + B_2
\end{align*}
where
\begin{align*}
B_1 = \| \beta^N - \beta \|_{\bV^*}, \text{ }B_2 = \| (^d{P} - I)\beta \|_{\bV^*}.
\end{align*}
Again by applying the same approach in the above discussion for $f$ function, we can prove $B_1 \le \eps(1+k)(1+\|x(0)\|_{\bH})$, so we only need to work on $B_2$. 
It is sufficient to prove the following uniform convergence:
\begin{equation}\label{uni_goal}
\esssup_{v\in U,x_t\in\Lambda_{0,t}^{0,k;x(0)}(\bV)} \left\| (^{d_j}{P} - I)\beta(t,x_t,v) \right\|_{\bV^*} \to 0 \text{ }\text{ a.e. } (\omega,t)\in \Omega\times [0,T], \text{ as }d_j\to\infty,
\end{equation}
where $\{d_j\}_{j\in\bN} \subseteq \bN^+$ is an infinite subsequence with $d_1<d_2<\cdots$.

Since, by $(\cA 1)$, $\beta$ is uniformly bounded in $\bH$, then for any $t\in [0,T]$, $x_t\in\Lambda^0_t(\bH)$, and $v\in U$, $\beta(t,x_t,v)$ takes value in $B_L^{\bH}(0)$, where $B_L^{\bH}(0)\subset \bH\subset \bV^*$, $B_L^{\bH}(0)$ is bounded in $\bH$ and compactly embedded in $\bV^*$. It is sufficient to prove the following uniform convergence holds
\begin{equation}\label{uni_h}
\esssup_{h\in B_L^{\bH}(0)} \left\| (^{d_j}{P} - I)h \right\|_{\bV^*} \to 0,
\end{equation}
where we consider the operator
\begin{equation*}
{^d{P}} - I: \bV^*\to \bV^*.
\end{equation*}
Define
\begin{equation*}
g_d(h) = \left\| ( {^d{P}} - I )h \right\|_{\bV^*},
\end{equation*}
and let space ${^d{P}\bB}$ be the set of $d$-dimensional projections $(^d{P})y$ of all $y\in\bB$, where $\bB$ is any Banach space. And for any space $\bB$, we define $\bB^*$ as its dual space. Thus for any linear map $\Psi: {^d{P}}\bV\to ({^d{P}}\bV)^*$, by Hahn-Banach theorem, there exists $\overline{\Psi}:\bV\to \bV^*$ such that for any $h\in \bV$
\begin{align*}
\Psi((^d{P})h) = & {_{\bV^*}\langle} \overline{\Psi}, (^d{P}h) \rangle_{\bV},
\\
=& {_{\bV^*}\langle} \overline{\Psi}, (^d{P})(^d{P}h) \rangle_{\bV},
\\
=&  {_{\bV^*}\langle} (^d{P})\overline{\Psi}, (^d{P}h) \rangle_{\bV}.
\end{align*}
Thus ${^d{P}}\bV^* = (^d{P}\bV)^*$. Similarly since $\bV$ is reflexive, we could obtain $({^d{P}}\bV^*)^* = {^d{P}}\bV^{**} = {^d{P}}\bV$. So $((^d{P}\bV)^*)^* = {^d{P}}\bV$ and $ ( (^d{P}\bV^*)^* )^* = {^d{P}}\bV^* $, i.e. for any $d\in\bN^+$, $\left( ^d{P}\bV, \|\cdot\|_{\bV} \right)$ and $\left( ^d{P}\bV^*, \|\cdot\|_{\bV^*} \right)$ are both reflexive spaces.

For any $d\in\bN^+$, $h\in \bV$, we have
\begin{align*}
\left\| (^d{P})h \right\|_{\bV} = &\sup_{b\in {^d{P}}\bV^*} \frac{ {_{\bV^*}\langle} b, {^d{P}}h \rangle_{\bV} }{ \|b\|_{\bV^*} }
\\
= &\sup_{b\in {^d{P}}\bV^*} \frac{ {_{\bV^*}\langle} b, h \rangle_{\bV} }{ \|b\|_{\bV^*} }
\\
\le &\sup_{b\in \bV^*} \frac{ {_{\bV^*}\langle} b, h \rangle_{\bV} }{ \|b\|_{\bV^*} }
\\
= &\|h\|_{\bV}.
\end{align*}
Then for any $d\in\bN^+$, $h\in \bV$, $h^*\in \bV^*$, 
\begin{align*}
{_{\bV^*}\langle} (^d{P})h^*, h \rangle_{\bV} = &{_{\bV^*}\langle} h^*, (^d{P})h \rangle_{\bV},
\\
\le &\|h^*\|_{\bV^*} \cdot \|(^d{P})h\|_{\bV},
\\
\le &\|h^*\|_{\bV^*} \cdot \|h\|_{\bV}.
\end{align*}
Then we must have
\begin{equation*}
\left\| (^d{P})h^* \right\|_{\bV^*} = \sup_{h\in \bV} \frac{ {_{\bV^*}\langle} (^d{P})h^*, h \rangle_{\bV} }{\|h\|_{\bV}} \le \|h^*\|_{\bV^*}.
\end{equation*}
This grants us equicontinuity of the function $g(\cdot)$ by
\begin{equation*}
g_d(h) = \left\| ( {^d{P}} - I )h \right\|_{\bV^*} \le 2\|h\|_{\bV^*},
\end{equation*}
where $h\in \bV^*$. Recall $B_L^{\bH}(0)$ is compactly embedded in $\bV^*$. Combined with the obvious pointwise convergence, $\{ g_d(\cdot) \}_{d\in\bN^+}$, contains a subsequence $\{ g_{d_j}(\cdot) \}$, for $j=1,2,3,\cdots$, that converges to 0 uniformly, i.e. if we pick $d = d_1, d_2, d_3, \cdots$, \eqref{uni_h} and thus \eqref{uni_goal} holds.

  \end{proof}
\begin{rmk}
The approximation lemma above differs from that of the finite dimensional case \cite[Lemma 4.2]{qiu2022controlled} in two ways. On the one hand, the initial state $x(0)$ needs to be fixed in order to make sense of the uniform convergence introduced by \eqref{state appro}. This is caused by the unboundedness of $A: \bV\to\bV$ or $\bV\to\bH$. This directly leads to the initial state dependence in not only Lemma \ref{lem-approx} but Proposition \ref{prop-cor-cmp}. On the other hand, by applying the finite dimensional approximation method, an additional layer of uniform convergence of \eqref{uni_goal} needs to be satisfied due to the infinite dimensional setup.

It is worthwhile to mention that the $x(0)$-dependence in Proposition \ref{prop-cor-cmp} and Lemma \ref{lem-approx} can be avoided if we alternatively introduce a stronger uniform boundedness assumption to the coefficients. For instance, assume further there exists $L>0$ such that for all $x_T\in \Lambda_T(\bV^*)$, $t\in [0,T]$ and $\tilde{x}_t\in\Lambda_t(\bV^*)$, there holds
\begin{equation*}
\esssup_{\omega\in\Omega} | G(x_T) |+\esssup_{\omega\in\Omega} \max_{v\in U} | f(t, \tilde{x}_t, v) |+\esssup_{\omega\in\Omega} \max_{v\in U} \| \beta(t, \tilde{x}_t, v)  \|_{\bH} \le \frac{L}{1+ \|x(0)\|_{\bH}\lor \|\tilde{x}(0)\|_\bH}.
\end{equation*}
In this case, the approximation of the coefficients introduced by \eqref{state appro} also works in a uniform sense.
\end{rmk}


\begin{thm}\label{thm-main}   
 Let Assumptions $(\cA 1)$ and $(\cA 2)$ hold. The viscosity solution to SPHJ equation \eqref{SPHJ} is unique.
 \end{thm}
\begin{proof}

For each $k \in \mathbb{N}^+$ and $\xi\in \bV$, we define
\begin{align*}
\overline{\mathcal{V}}_k(\xi) = \{   \phi\in\mathcal{C}_{\sF}^1 : \phi(T, x) 
&\ge 
      G(x) \text{  }\forall x\in\Lambda^0_T(\bH), \text{a.s.}, \text{ and for each } t\in[0, T) \text{ with } y\in\Lambda_{0,t}^{0, k ;\xi}(\bV), 
\\
&\text{ess}
      \liminf_{(s, x)\to(t^+, y)}  [  - \mathfrak{d}_s\phi(s, x) - \mathcal{H}(s, x, \nabla\phi(s, x))  ] \ge 0, \text{ a.s.}  \},
\end{align*}


\begin{align*}
\underline{\mathcal{V}}_k(\xi) = \{   \phi\in\mathcal{C}_{\sF}^1 : \phi(T, x) &\le G(x) \text{  }\forall x\in\Lambda^0_T(\bH), \text{a.s.}, \text{ and for each } t\in[0, T) \text{ with } y\in\Lambda_{0,t}^{0, k ;\xi}(\bV),\\ 
&\text{ess}  \limsup_{(s, x)\to(t^+, y)}  [  - \mathfrak{d}_s\phi(s, x) - \mathcal{H}(s, x, \nabla\phi(s, x))  ] \le 0, \text{ a.s.} \}.
\end{align*}

Set
\begin{equation*}
\overline{u}_k(t,y) = \essinf_{\phi_k \in \overline{\mathcal{V}}_k(\xi), \xi\in\bV} \phi_k(t,y),\text{ } \underline{u}_k(t,y) = \esssup_{\phi_k \in \underline{\mathcal{V}}_k(\xi), \xi\in\bV} \phi_k(t,y),
\end{equation*}
and
\begin{equation*}
\overline{u} = \lim_{k\to +\infty} \overline{u}_k, \text{ }\underline{u} = \lim_{k\to +\infty} \underline{u}_k.
\end{equation*}
By Proposition \ref{prop-cor-cmp} and Theorem \ref{exist}, the viscosity solution $V$ satisfies $\underline{u}\le V\le\overline{u}$ on $\cup_{k=1}^{\infty} \cup_{\xi\in \bV}\Lambda_{0,T}^{0,k;\xi}(\bV)$, which is dense in $\Lambda_T^0(\bH)$. For the purpose of proving the uniqueness of the viscosity solution, we only need to verify $\underline{u} = V = \overline{u}$ on $\cup_{k=1}^{\infty} \cup_{\xi\in \bV}\Lambda_{0,T}^{0,k;\xi}(\bV)$.


\textbf{Step 1.} Like what we do with viscosity subsolutions and supersolutions, we construct functions from $\overline{\mathcal{V}}_k(\xi)$ and $\underline{\mathcal{V}}_k(\xi)$ to dominate the value function $V$ from above and below, respectively. Let $(\Omega', \sF', \{  \sF'_t  \}_{t\ge 0}, \mathbb{P}')$ be another completely filtered probability space on which an $d$ dimensional standard Brownian motion $B = \{B(t) : t \ge 0\}$ is well defined. The filtration $\{  \sF'_t  \}_{t\ge 0}$ is generated by $B$ and augmented by all $\mathbb{P}'$-null sets in $\sF'$.

For simplicity we put
\begin{equation*}
(\tilde{\Omega}, \tilde{\sF}, \{  \tilde{\sF}_t  \}_{t\ge 0}, \tilde{\mathbb{P}}) = (\Omega\times\Omega', \sF\times\sF', \{  \sF_t\times\sF'_t  \}_{t\ge 0}, \mathbb{P}\otimes\mathbb{P}'),
\end{equation*}
and denote by $\overline{\mathcal{\cU}}$ the set of all $U$-valued and $\tilde{\sF}_t$-adapted processes. Then we have two independent Brownian motions $B(t)$ and $W(t)$ and all the previous results still hold on the enlarged probability space.

For each $\xi\in \bV$ and any $\epsilon\in(0, 1)$ and $k\in\mathbb{N}^+$, choose $(G^{\epsilon}_k, f^{\epsilon}_k, \beta^{\epsilon}_k)$ and $(G^N, f^N, ({^d{P}}){\beta}^N)$ as defined in Lemma \ref{lem-approx}. By the theory of BSDEs (\cite{briand2003lp}), we have the pairs $(Y^{\epsilon}_k, Z^{\epsilon}_k)$ and $(y, z)$ be the unique adapted solutions to the BSDEs
\begin{equation*}
Y^{\epsilon}_k(s) = G^{\epsilon}_k + \int_s^T \! (  f^{\epsilon}_k(t) + C_1 \beta^{\epsilon}_k(t)  )dt - \int_s^T \! Z^{\epsilon}_k(t) dW(t),
\end{equation*}
and 
\begin{equation*}
y(s) = \|  B_T  \|_{0,\bH} + \int_s^T \! \|  B_r  \|_{0,\bH} dr - \int_s^T \! z(r) dB(r),
\end{equation*}
respectively, with $C_1\ge 0$, which will be determined later. 

For each $s\in[0, T)$ and $x_s\in \Lambda_{0,s}^{0,k}(\bV)$, define $x_s^{d} := (^d{P})(x_s)\in (^d{P})\left(\Lambda_{0,s}^{0,k}(\bV)\right)$ as the $d$-dimensional projection of the corresponding path $x_s$. Let
\begin{align*}
V^{\epsilon}(s,x_s^d) = \mbox{ess}\inf_{\theta\in\overline{\mathcal{U}}} E_{\sF_s} \Bigg[  \int_s^T \! &f^N\bigg(  W((t_1 \land t)),\cdots, W(t_N \land t), t, X^{s,x_s^d;\theta,N}(0), X^{s,x_s^d;\theta,N}(t_1 \land t), \cdots,
\\ 
&\text{ }\text{ }\text{ }\text{ }\text{ }X^{s,x_s^d;\theta,N}(t_N \land t), \theta(t)  \bigg)dt
\\
+G^N&\left(  W(t_1), \cdots, W(t_N), X^{s,x_s^d;\theta,N}(0), X^{s,x_s^d;\theta,N}(t_1),\cdots, X^{s,x_s^d;\theta,N}(t_N)  \right)  \Bigg],
\end{align*}
where $X^{s,x_s^d;\theta,N}(t) \in C([0,T];\hat{\bH}^d)$ satisfies the following finite dimensional stochastic differential equation 
\begin{equation*}
\left\{
\begin{split}
dX(t) = &{^d{P}}(AX(t)) dt
\\
&{+^d{P}}\beta^N\left(  W(t_1 \land t),\cdots, W(t_N \land t), t, X(0), X(t_1 \land t), \cdots, X(t_N \land t), \theta(t)  \right)dt
\\
&+ \delta \sum_{i=1}^d e_idB^i(t), \text{ } t\in[s, T]; \label{fin_x}
\\
X(t) = &x_s^d(t), \text{ }t\in[0, s]
\end{split}
\right.
\end{equation*}
with $\delta\in(0, 1)$ as an arbitrary constant. 
For each $s\in [t_{N-1}, T]$, $x_s\in \Lambda^{0,k}_{0,s}(\bV)$, let
\begin{align*}
V^{\epsilon}(s, x_s^d) &= \tilde{V}^{\epsilon}(  W(t_1), \cdots, W(t_{N-1}), W(s), s, x^d(0), \cdots, x^d(t_{N-1}), x^d(s)  )
\end{align*}
with
\begin{align*}
      &\tilde{V}^{\epsilon}(  W(t_1), \cdots, W(t_{N-1}), \tilde{y}, s, x^d(0), \cdots, x^d(t_{N-1}), \tilde{x}  )
\\
      =& \essinf_{\theta\in\mathcal{U}} E_{\sF_s, W(s) = \tilde{y}, x_s^d(s) = \tilde{x}} \Bigg[  \int_s^T \! f^N(  W((t_1)),\cdots, W(t_{N-1}), W(t), t, \cdots, x^d(t_{N-1}),
\\ 
&\text{ }\text{ }\text{ }\text{ }  X^{s,x_s^d;\theta,N}(t), \theta(t)  )dt + G^N(  W(t_1), \cdots, W(t_N), x^d(0),\cdots, X^{s,x_s^d;\theta,N}(T)  ) \Bigg].
\end{align*}
For simplicity, we will write $x^d(t_j) = x_s^d(t_j)$ for $j = 0, \cdots, N-1$ because for $s\in(t_{N-1}, T]$, they are fixed with the terminal condition. By the viscosity solution theory of fully nonlinear parabolic PDEs(\cite{lions1983optimal}), the function $\tilde{V}^{\epsilon}(  W(t_1), \cdots, W(t_{N-1}), \tilde{y}, s, x^d(0), \cdots, x^d(t_{N-1}), \tilde{x}  )$ satisfies the following HJB equation:
\begin{align*}
- D_t &u(\tilde{y}, t, \tilde{x}) = \frac{1}{2}tr(  D_{\tilde{y}\tilde{y}} u(\tilde{y}, t, \tilde{x}) ) + \frac{\delta^2}{2} tr(  D_{\tilde{x}\tilde{x}} u(\tilde{y}, t, \tilde{x})  )  
\\
&\text{ }\text{ }\text{ }\text{ }\text{ }\text{ }\text{ }\text{ }\text{ }\text{ }\text{ }\text{ }\text{ }\text{ }\text{ } + \mbox{ess}\inf_{v\in U} \{  [  (^d{P}Ax^d(t))'
\\
&\text{ }\text{ }\text{ }\text{ }\text{ }\text{ }\text{ }\text{ }\text{ }\text{ }\text{ }\text{ }\text{ }\text{ }\text{ } + ((^d{P})\beta^N)'(  W(t_1), \cdots, W(t_{N-1}), \tilde{y}, t, x^d(0), \cdots, x^d(t_{N-1}), \tilde{x}, v  ) ] D_{\tilde{x}} u(\tilde{y}, t, \tilde{x})  
\\
&\text{ }\text{ }\text{ }\text{ }\text{ }\text{ }\text{ }\text{ }\text{ }\text{ }\text{ }\text{ }\text{ }\text{ }\text{ } + f^N (  W(t_1), \cdots, W(t_{N-1}), \tilde{y}, t, x^d(0), \cdots, x^d(t_{N-1}), \tilde{x}, v  ) \};
\\
&u(\tilde{y}, T, \tilde{x}) = G^N (  W(t_1), \cdots, W(t_{N-1}), \tilde{y}, t, x^d(0), \cdots, x^d(t_{N-1}), \tilde{x}  ).
\end{align*}
Thus the regularity of viscosity solutions (\cite[Theorem 6.4.3]{krylov1987nonlinear}) gives for each $x(t_i) \in \Lambda^{0,k}_{0,T}(\bV)$, $i = 0, \cdots, N-1$,
\begin{align*}
\tilde{V}^{\epsilon}&(  W(t_1), \cdots, W(t_{N-1}), \cdot , \cdot, x^d(0), \cdots, x^d(t_{N-1}), \cdot  )
\\
      &\in\cap_{\overline{t}\in[t_{N-1}, T)} L^{\infty}(  \Omega, \mathcal{F}_{t_{N-1}}; C^{1+\frac{\overline{\alpha}}{2}, 2+\overline{\alpha}}([  t_{N-1}, \overline{t}  ]\times\mathbb{R}^{m}\times \hat{\bH}^d)  ),
\end{align*}
for some $\overline{\alpha}\in(0,1)$, where the $time-space$ H\"older space $C^{1+\frac{\overline{\alpha}}{2}, 2+\overline{\alpha}}([  t_{N-1}, \overline{t}  ]\times\mathbb{R}^{m}\times \hat{\bH}^d)$ is defined as usual. The above arguments can be extended similarly on time interval $[t_{N-2}, t_{N-1}]$ with the previously computed $V^{\epsilon}(t_{N-1}, x^d)$ as the terminal value, and it goes recursively on $[ t_{N-3}, t_{N-2} ], \cdots, [0, t_1]$.

On $t\in [t_{N-1}, T]$, by applying the It\^o-Kunita formula (\cite[Page 118-119]{kunita1981some}) to $V^{\epsilon}$, for any $x_t\in \Lambda^{0,k}_{0,t}(\bV)$ and $x_T\in \Lambda^{0,k}_{0,T}(\bV)$, we can have the following differential equation
\begin{align*}
&- dV^{\epsilon}\left(t, (x^d - \delta\sum_i^d e_iB^i)_t\right)
\\
=& \mbox{ess}\inf_{v\in U} \Bigg\{  {^d{P}} ( Ax^d(t) )' \nabla V^{\epsilon}\left( t, (x^d - \delta\sum_{i=1}^d e_iB^i)_t\right)
\\
&\text{ }\text{ }\text{ }\text{ }\text{ }\text{ }\text{ }\text{ }\text{ }\text{ }\text{ }\text{ } + \left( (^d{P}) \beta^N\right)'\Bigg(  W(t_1), \cdots, W(t), t, x^d(0), \cdots, x^d(t_{N-1}) - \delta\sum_{i=1}^d e_iB^i(t_{N-1}), 
\\
&\text{ }\text{ }\text{ }\text{ }\text{ }\text{ }\text{ }\text{ }\text{ }\text{ }\text{ }\text{ }\text{ }\text{ }\text{ }\text{ }\text{ }\text{ }\text{ }\text{ }\text{ }\text{ }\text{ }\text{ }\text{ }\text{ }\text{ }\text{ }\text{ }\text{ }\text{ }\text{ }\text{ }x^d(t) -\delta\sum_{i=1}^d e_iB^i(t), v  \Bigg) \cdot \nabla V^{\epsilon}\left( t, (x^d - \delta\sum_{i=1}^d e_iB^i)_t \right)
\\ 
&\text{ }\text{ }\text{ }\text{ }\text{ }\text{ }\text{ }\text{ }\text{ }\text{ }\text{ }\text{ } + f^N \Bigg(  W(t_1), \cdots, W(t), t, x^d(0), \cdots,  
\\
&\text{ }\text{ }\text{ }\text{ }\text{ }\text{ }\text{ }\text{ }\text{ }\text{ }\text{ }\text{ }\text{ }\text{ }\text{ }\text{ }\text{ }\text{ }\text{ }\text{ }\text{ }\text{ }x^d(t_{N-1}) - \delta\sum_{i=1}^d e_iB^i(t_{N-1}),x^d(t) - \delta\sum_{i=1}^d e_iB^i(t), v  \Bigg) \Bigg\}dt
\\
& - D_{\tilde{y}} \tilde{V}^{\epsilon} \Bigg(  W(t_1),\cdots, W(t), t, x^d(0), \cdots,  
\\
&\text{ }\text{ }\text{ }\text{ }\text{ }\text{ }\text{ }\text{ }\text{ }\text{ }\text{ }\text{ }\text{ }\text{ }x^d(t_{N-1}) - \delta\sum_{i=1}^{d} e_iB^i(t_{N-1}),x^d(t) - \delta\sum_{i=1}^d e_iB^i(t)  \Bigg) dW(t)
\\
& + \delta  \nabla V^{\epsilon}\left(t, (x^d -\delta\sum_{i=1}^d e_iB^i)_t\right)d B(t),
\\
&V^{\epsilon}(T, x^d_T) = G^N\Bigg(  W(t_1), \cdots, W(T), x^d(0), \cdots, 
\\
&\text{ }\text{ }\text{ }\text{ }\text{ }\text{ }\text{ }\text{ }\text{ }\text{ }\text{ }\text{ }\text{ }\text{ }\text{ }\text{ }\text{ }\text{ }\text{ }\text{ }\text{ }\text{ }\text{ }\text{ }x^d(t_{N-1}) - \delta\sum_{i=1}^d e_iB^i(t_{N-1}), x^d(T) - \delta\sum_{i=1}^d  e_iB^i(T)  \Bigg).
\end{align*}
Again it follows similarly on other intervals $[0, t_1), \cdots, [t_{N-2}, t_{N-1})$. 

Then by Lemma \ref{lem-approx}, Proposition \ref{prop3.2} (iv), (v) and the definition of the gradient, there exists $\tilde{L}>0$ such that for any $t\in[0, T]$ with $x_t\in\Lambda^{0,k}_{0,t}(\bV)$,
\begin{equation*}
\|  \nabla V^{\epsilon}(t, x_t^d)  \|_{\bV} \le \tilde{L}, \text{ a.s.}
\end{equation*}
 
Here $\tilde{L}$ is not dependent of $\epsilon$, $d$, and $N$. Hence by Definition \ref{cf1} we have $V^{\epsilon}(\cdot, \cdot - \delta B(\cdot))\in \mathcal{C}^1_{\sF}$.

Let $C_1 = \tilde{L}$ and $C_2 = 4L_c ( \tilde{L} + 1 )$, where $L_c$ is the Lipschitz constant defined in Lemma 4.2. And let
\begin{equation*}
\overline{V}^{\epsilon}_k (s, x) = \overline{V}^{\epsilon}_k (s, x^d) = V^{\epsilon} \left(s, (x^d - \delta\sum_{i=1}^d e_iB^i)_s\right) + Y^{\epsilon}_k (s) + \delta C_2 y(s),
\end{equation*}
\begin{equation*}
\underline{V}^{\epsilon}_k (s, x) = \underline{V}^{\epsilon}_k (s, x^d) = V^{\epsilon} \left(s, (x^d - \delta\sum_{i=1}^d e_iB^i)_s\right) - Y^{\epsilon}_k (s) - \delta C_2 y(s).
\end{equation*}
Clearly for any $s\in [0,T]$, $x\in\Lambda^{0,k}_{0,s}(\bV)$, and any $k\in\bN^+$, both $\overline{V}^{\epsilon}_k (s, x^d)$ and $\underline{V}^{\epsilon}_k (s, x^d)$ take value in $\bR$. The uniform Lipschitz continuity assumption in $(\cA 1)$ will grant us
\begin{align*}
&\left\|  ^d{P}\beta^N(t, x_t^d, v) - {^d{P}}\beta^N\left(t, (x^d - \delta\sum_{i=1}^d e_iB^i)_t, v\right)  \right\|_{\bV^*} \le \delta L_c\| B_t \|_{0,\bH},
\\
&\left|  f^N(t, x_t^d, v) - f^N\left(t, (x^d - \delta\sum_{i=1}^d e_iB^i)_t, v\right)  \right| \le \delta L_c\| B_t \|_{0,\bH},
\end{align*}
for some constant $L_c>0$.


It holds that for all $(t, x_t)$ with $t\in(t_{N-1}, T)$ and $x_t\in \Lambda_{0,t}^{0,k;\xi}(\bV)$ with the corresponding $x_t^d$ satisfying the finite dimensional SDE,
\begin{align*}
&- \mathfrak{d}_t \overline{V}^{\epsilon}_k - \mathcal{H}(\nabla\overline{V}^{\epsilon}_k)
\\
= &-\mathfrak{d}_t \overline{V}^{\epsilon}_k - \essinf_{v\in U} \left\{ (Ax)'\nabla\overline{V}_k^{\epsilon} + \beta'\nabla\overline{V}_k^{\epsilon} + f \right\}
\\
= &- \mathfrak{d}_t \overline{V}^{\epsilon}_k - \essinf_{v\in U} \{ ( ^d{P}(Ax^d)  )' \nabla\overline{V}_k^{\epsilon} + ((^d{P})\beta^N)' \nabla\overline{V}^{\epsilon}_k + f^N + f_k^{\epsilon} +\tilde{L}\beta^{\epsilon}_k + \delta C_2 \| B_t \|_{0,\bH}
\\
& +  (Ax)'\nabla\overline{V}_k^{\epsilon} - (^d{P}(Ax^d) )' \nabla\overline{V}_k^{\epsilon} + \beta'\nabla\overline{V}_k^{\epsilon} - ((^d{P})\beta^N)' \nabla\overline{V}^{\epsilon}_k - \beta^{\epsilon}_k \tilde{L} + f - f^N 
\\
&- f^{\epsilon}_k - \delta C_2 \|  B_t \|_{0,\bH} \}
\\
\ge &- \mathfrak{d}_t \overline{V}^{\epsilon}_k - \essinf_{v\in U} \{ ( ^d{P}(Ax^d) ))' \nabla\overline{V}_k^{\epsilon} + ((^d{P})\beta^N)' \nabla\overline{V}^{\epsilon}_k + f^N + f_k^{\epsilon} +\tilde{L}\beta^{\epsilon}_k + \delta C_2 \| B_t \|_{0,\bH}  \}
\\
= &- \mathfrak{d}_t \left[ V^{\epsilon}\left(t, (x^d -\delta\sum_{i=1}^d e_iB^i)_t\right) + Y_k^{\epsilon}(t) + \delta C_2y(t) \right]
\\
& - \essinf_{v\in U} \{ ( ^d{P}(Ax^d)  )' \nabla\overline{V}_k^{\epsilon} + ((^d{P})\beta^N)' \nabla\overline{V}^{\epsilon}_k + f^N + f_k^{\epsilon} +\tilde{L}\beta^{\epsilon}_k + \delta C_2 \| B_t \|_{0,\bH}  \} 
\\
= &- \mathfrak{d}_t V^{\epsilon}\left(t, (x^d - \delta\sum_{i=1}^d e_iB^i)_t\right) - \mathfrak{d}_t \left[ G_k^{\epsilon} + \int_t^T \! ( f_k^{\epsilon}(s) + C_1 \beta_k^{\epsilon}(s) ) ds - \int_t^T \! Z_k^{\epsilon}(s)dW(s) \right] 
\\
&- \delta C_2 y'(t) - \essinf_{v\in U} \left\{ ( ^d{P}(Ax^d)  )' \nabla\overline{V}_k^{\epsilon} + ((^d{P})\beta^N)' \nabla\overline{V}^{\epsilon}_k + f^N + f_k^{\epsilon} +\tilde{L}\beta^{\epsilon}_k + \delta C_2 \| B_t \|_{0,\bH}  \right\} 
\\
\ge &-\mathfrak{d}_t V^{\epsilon}\left(t, (x^d-\delta\sum_{i=1}^d e_iB^i)_t\right) - \essinf_{v\in U} \left\{ \left( ^d{P}(Ax^d) \right)' \nabla \overline{V}_k^{\epsilon} + \left((^d{P})\beta^N\right)' \nabla \overline{V}_k^{\epsilon} + f^N  \right\}
\\
\ge &-\mathfrak{d}_t V^{\epsilon}\left(t, (x^d-\delta\sum_{i=1}^d e_iB^i)_t\right) - \essinf_{v\in U} \left\{ \left( ^d{P}(Ax^d) \right)' \nabla V_k^{\epsilon} + \left((^d{P})\beta^N\right)' \nabla V_k^{\epsilon} + f^N  \right\}
\\
 = &-\mathfrak{d}_t V^{\epsilon}\left(t, (x^d-\delta\sum_{i=1}^d e_iB^i)_t\right) + \mathfrak{d}_t V^{\epsilon}\left(t, (x^d-\delta\sum_{i=1}^d e_iB^i)_t\right)
\\
= &0. 
\end{align*}
Here the first equality comes straight from the definition of the Hamilton function, and the first inequality is the direct result of Lemma \ref{lem-approx} and the above result
\begin{equation*}
\left\|  \nabla V^{\epsilon}(t, x_t^d)  \right\|_{\bV} \le \tilde{L}, \text{ a.s.,}
\end{equation*}
since
\begin{align*}
& \left( Ax - {^d{P}}(Ax^d) \right)'\nabla\overline{V}_k^{\epsilon} + \left(\beta - (^d{P})\beta^N\right)' \nabla\overline{V}^{\epsilon}_k - \beta^{\epsilon}_k \tilde{L} + f - f^N - f^{\epsilon}_k - \delta C_2 \|  B_t \|_{0,\bH}
\\
=& \left( Ax^d \right)'\nabla\overline{V}_k^{\epsilon} - \left( ^d{P}(Ax^d) \right)'\nabla\overline{V}_k^{\epsilon} + \beta' \nabla\overline{V}^{\epsilon}_k - \left(^d{P}\beta^N\right)' \nabla\overline{V}^{\epsilon}_k - \beta^{\epsilon}_k \tilde{L} + f - f^N - f^{\epsilon}_k 
\\
&- \delta C_2 \|  B_t \|_{0,\bH}
\\
=& \left( Ax^d \right)'\nabla\overline{V}_k^{\epsilon} - \left( Ax^d \right)'(^d{P})(\nabla\overline{V}_k^{\epsilon}) + \beta' \nabla\overline{V}^{\epsilon}_k - \left((^d{P})\beta^N \right)'  (\nabla\overline{V}_k^{\epsilon})  - \beta^{\epsilon}_k \tilde{L} + f - f^N - f^{\epsilon}_k
\\
& - \delta C_2 \|  B_t \|_{0,\bH}
\\
=& \left( Ax^d \right)'\nabla\overline{V}_k^{\epsilon} - \left( Ax^d \right)' \nabla\overline{V}_k^{\epsilon} + \beta' \nabla\overline{V}^{\epsilon}_k - \left((^d{P})\beta^N \right)'  (\nabla\overline{V}_k^{\epsilon})  - \beta^{\epsilon}_k \tilde{L} + f - f^N - f^{\epsilon}_k
\\
& - \delta C_2 \|  B_t \|_{0,\bH}
\\
\le&  \left|f - f^N\right| - f^{\epsilon}_k - \delta C_2 \|  B_t \|_{0,\bH}
\\
\le&  - \delta C_2 \|  B_t \|_{0,\bH}
\\
\le& 0
\end{align*}



Likewise, it follows similarly on intervals $(t_{N-2}, t_{N-1}), \cdots, [0, t_1)$ that
\begin{equation*}
- \mathfrak{d}_t \overline{V}^{\epsilon}_k - \mathcal{H}(\nabla\overline{V}^{\epsilon}_k) \ge 0,
\end{equation*}
which together with the terminal condition relation $\overline{V}^{\epsilon}_k(T) = G^{\epsilon}_k + G^N + \delta C_2 \|  B_T  \|_{0,\bH} \ge G$, we have $\overline{V}^{\epsilon}_k \in \overline{\mathcal{V}}_k(\xi)$. Similarly we obtain $\underline{V}^{\epsilon}_k \in \underline{\mathcal{V}}_k(\xi)$.

$\text{ }$

\textbf{Step 2.} Let $B_t^{s, 0}(r) = B(r) - B(r\land s)$ for $0\le s\le t\le T$ for any $r\in [0,t]$. And for any $t\in [0,T]$, define $\delta^N(t) = X^{s,x_s^d;\theta,N}(t) - X^{s,x_s;\theta}(t)$, where $x_s\in\Lambda_{0,s}^{0,k;\xi}(\bV)$. Then by Lemma \ref{lem-approx} and It\^o's formula, we have for $\forall \text{ }  \theta \in \mathcal{U}$,
\begin{align*}
&\left\|  X^{s, x_s^d; \theta, N}(t) - X^{s, x_s; \theta}(t)  \right\|_{\bH}^2
\\
=&\left\| x^d(s) - x(s) \right\|_{\bH}^2 + 2\int_s^t \!_{\bV^*}{\langle} AX^{s,x_s^d;\theta,N}(r) - AX^{s,x_s;\theta}(r), \delta^N(r)  \rangle_{\bV}dr 
\\
&+ 2\int_s^t \! _{\bV^*}{\langle} (^d{P})\beta^N(r,X_r^{s,x_s^d;\theta,N},\theta(r)) -\beta(r,X_r^{s,x_s;\theta},\theta(r)), \delta^N(r)   \rangle_{\bV}  dr
\\
&+ \int_s^t \! \delta^2 dr + 2\int_s^t \! \langle \delta\sum_{i=1}^d e_i dB^i(r),\delta^N(r)  \rangle_{\bH}.
\end{align*}
Let
\begin{align*} 
\tilde{B}_1(t) &= \int_s^t \! \delta^2 dr + 2\int_s^t \! \langle \delta\sum_{i=1}^d e_i dB^i(r),\delta^N(r)  \rangle_{\bH}.
\end{align*}
Then by Lemma 4.2, Assumption $(\cA 2)$, we have
\begin{align*}
&\left\|  X^{s, x_s^d; \theta, N}(t) - X^{s, x_s; \theta}(t)  \right\|_{\bH}^2
\\
\le &\left\| x^d(s) - x(s) \right\|_{\bH}^2 + 2\int_s^t \!_{\bV^*}{\langle} AX^{s,x_s^d;\theta,N}(r) - AX^{s,x_s;\theta}(r), \delta^N(r)  \rangle_{\bV}dr 
\\
& +2\int_s^t \! \|\delta^N(r)\|_{\bV} \cdot \left\| (^d{P})\beta^N( r,X_r^{s,x_s;\theta,N},\theta(r) ) - (^d{P})\beta^N( r,X_r^{s,x_s;\theta},\theta(r) ) \right\|_{\bV^*} dr 
\\
& +2\int_s^t \! \|\delta^N(r)\|_{\bV} \cdot \left\| (^d{P})\beta^N( r,X_r^{s,x_s;\theta},\theta(r) ) - \beta( r,X_r^{s,x_s;\theta},\theta(r) ) \right\|_{\bV^*} dr + \tilde{B}_1(t)
\\
\le &\left\| x^d(s) - x(s) \right\|_{\bH}^2 + c_1\int_s^t \! \|\delta_N(r)\|_{\bH}^2  dr - c_2\int_s^t \! \|\delta_N(r)\|_{\bV}^2  dr 
\\
&+ 2 \int_s^t \! \|\delta^N(r)\|_{\bV} \left( L_c \max_{\tau\in [0,r]}\| \delta^N(\tau) \|_{\bH} + \beta_k^{\eps}(r) \right)dr + \tilde{B}_1(t)
\\
\le &\left\| x^d(s) - x(s) \right\|_{\bH}^2 + c_1\int_s^t \! \|\delta_N(r)\|_{\bH}^2  dr - c_2\int_s^t \! \|\delta_N(r)\|_{\bV}^2  dr
\\
& + (\frac{c_2}{4}+\frac{c_2}{4})\int_s^t \! \| \delta^N(r) \|_{\bV}^2 dr + \frac{4L_c^2}{c_2}\int_s^t \! \max_{\tau\in [0,r]}\| \delta^N(\tau) \|_{\bH}^2 dr + \frac{4}{c_2} \int_s^t \! \left(\beta_k^{\eps}(r)\right)^2dr + \tilde{B}_1(t)
\\
\le &\left\| x^d(s) - x(s) \right\|_{\bH}^2 + (c_1^+ + \frac{4L_c^2}{c_2})\int_s^t \! \max_{\tau\in [0,r]} \|\delta_N(\tau)\|_{\bH}^2  dr - \frac{c_2}{2} \int_s^t \! \| \delta^N(r) \|_{\bV}^2 dr 
\\
&+ \tilde{B}_1(t) + \frac{4}{c_2}\|\beta_k^{\eps}\|_{L^2(\Omega\times [0,T];\bR)}^2
\end{align*}
Thus by Gr\"onwall's inequality, we have

\begin{equation*}
E_{\sF_s} \left[  \sup_{s\le t\le T}\left\|  X^{s, x_s^d; \theta, N}(t) - X^{s, x_s; \theta}(t)  \right\|_{\bH}^2  \right] \le C_5\left(  \delta^2 + \|\beta_k^{\eps}\|_{L^2(\Omega,[0,T];\bR)}^2 + \left\| x^d(s) - x(s) \right\|_{\bH}^2\right),
\end{equation*}
and likewise, here $C_5$ is independent of the choice of $s, x_s, \xi, \delta, N, k, \epsilon, d$, and $\theta$. Combined with $(\cA 2)$, \text{H\"older's inequality}, and Lemma \ref{lem-approx}, we have
\begin{align*}
&\text{ }\text{ }\text{ }\text{ } E\left|  V^{\epsilon}(s, x_s^d) - V(s, x_s)  \right|
\\
&\le E \text{ }\mbox{ess}\sup_{\theta\in\overline{\mathcal{U}}} E_{\sF_s} \Bigg[  \int_s^T \! (  |  f^N (t, X_t^{s, x_s^d; \theta, N}, \theta(t)  ) - f(t, X_t^{s, x_s; \theta}, \theta(t)) | )dt
\\
&\text{ }\text{ }\text{ }\text{ } +|  G^N (X_T^{s, x_s^d; \theta, N}) - G(X_T^{s, x_s; \theta})  |  \Bigg]
\\
&\le C\left((\delta^2 +  \|  \beta^{\epsilon}_k  \|^2_{L^2 (\Omega\times[0, T]; \bR)} +\left\| x^d(s) - x(s) \right\|_{\bH}^2)^{\frac{1}{2}} + \|  f^{\epsilon}_k  \|_{L^2 (\Omega\times[0, T]; \mathbb{R})} + \|  G^{\epsilon}_k  \|_{L^2 (\Omega; \mathbb{R})}  \right)
\\
&\le C_6 \left(\epsilon(1+k)(1+\|\xi\|_{\bH}) + \left\| (^d{P} - I) x(s) \right\|_{\bH} + \delta \right).
\end{align*}
Likewise, the constant $C_6$ does not depend on the choice of $N, \theta, \epsilon, d, k, \delta, \xi$, and $(s, x_s)$.


And since $\bV$ is compactly embedded into the Hilbert space $\bH$, the proof of the uniform convergence
\begin{equation*}
\lim_{d\to\infty} \left\|(^d{P} - I) x(s)\right\|_{\bH} = 0
\end{equation*}
is standard for any $s\in [0,T]$ and $x(s)\in\bV$. Then by definition of $\overline{V}^{\epsilon}_k$ and $\underline{V}^{\epsilon}_k$, for $\forall s\in[0, T]$ and $x_s\in\Lambda^{0, k;\xi}_{0, s}(\bV)$, we have the following convergence
\begin{equation*}
\lim_{\delta\to 0,d\to\infty}E|  \overline{V}^{\epsilon}(s, x_s^d) - V(s, x_s)  | + \lim_{\delta\to 0,d\to\infty}E|  \underline{V}^{\epsilon}(s, x_s^d) - V(s, x_s)  |  = 0.
\end{equation*}
Since the value function $V$ is a viscosity solution, there exist two infinite series of increasing positive integers $\{\overline{k}_n\}_{n\in\bN^+}$ and $\{\underline{k}_n\}_{n\in\bN^+}$ satisfying
\begin{equation*}
\lim_{n\to\infty} \overline{k}_n = \lim_{n\to\infty} \underline{k}_n = +\infty,
\end{equation*}
such that for each $n\in\bN^+$, $t\in [0,T]$ and $x_t\in\Lambda_{0,t}^{0,\overline{k}_n;\xi}(\bV)$ it holds
\begin{equation*}
V(t,x_t) \le \overline{V}^{\epsilon}_{\overline{k}_n}(t, x_t^d), \text{ a.s.}, 
\end{equation*}
and respectively for all $n\in\bN^+$, $t\in [0,T]$ and $x_t\in\Lambda_{0,t}^{0,\underline{k}_n;\xi}(\bV)$ it holds
\begin{equation*}
\underline{V}^{\epsilon}_{\underline{k}_n}(t, x_t^d) \le V(t,x_t), \text{ a.s.} 
\end{equation*}
Note the above equalities hold by the arbitrariness of $\epsilon, \delta, \xi$, and $\overline{k}_n$ (respectively $\underline{k}_n$). Further combined with the denseness of $\cup_{n=1}^{\infty} \cup_{\xi\in \bV} \Lambda_{0, T}^{0, \overline{k}_n;\xi}(\bV)$ in $\Lambda_T^0(\bH)$ and respectively $\cup_{n=1}^{\infty} \cup_{\xi\in \bV} \Lambda_{0, T}^{0, \underline{k}_n;\xi}(\bV)$ in $\Lambda_T^0(\bH)$, we have
\begin{equation*}
\lim_{\delta\to 0,d\to\infty}\underline{u}(t,x_t^d) = V(t, x_t) = \lim_{\delta\to 0,d\to\infty}\overline{u}(t, x_t^d) \text{ a.s.,}
\end{equation*}
for all choice of $t\in[0, T]$ and $x_t\in\Lambda_{t}^0(\bH)$.

\end{proof}

\begin{rmk}\label{rmk-superparab}
The above proof is inspired by but different from the conventional Perron's method, for instance, in \cite{buckdahn2015pathwise, ekren2016viscosity1, qiu2018viscosity}, and extending the space from $\mathbb{R}^d$ into infinite dimensional Hilbert space does bring up tremendous challenges. The main difference is two-folded. For one thing, in an infinite dimensional set up, we need to carefully ensure the compactness of the state space because of the possibly unbounded operator $A: \bV\to\bV$ or $\bV\to\bH$. The operator $A$ is only bounded $\bV\to\bV^*$; for the other, finite dimensional approximation would be applied to prove the uniqueness of the viscosity solution. On the other hand,  by enlarging the original filtered probability space with an independent Brownian motion $B$,  we have actually constructed the regular approximations of the value function $V$ with a regular perturbation induced by $\delta  B$, which corresponds to approximation to the optimization \eqref{Control-probm}-\eqref{state-proces-contrl} with piecewise Markovian stochastic controls. Such approximation seems interesting even for the case where all the coefficients $\beta,\,f,$ and $G$ are just deterministic and path-dependent.
\end{rmk}

\section{Appendix}
\subsection{Proof of Lemma \ref{lemma3.1}}
\begin{proof}
(i) comes naturally with the uniqueness of the solution of (\ref{state-proces-contrl}). As of (ii), for each $\tau,l\in [r,T]$, It\^o's formula of square norms yields that
\begin{align*}
&\| X^{r,\xi;\theta}(\tau) \|_{\bH}^2 \\
&= \| X^{r,\xi;\theta}(r) \|_{\bH}^2 + \int_r^{\tau} \! 2{_{\bV^*}\langle} AX^{r,\xi;\theta}(s), X^{r,\xi;\theta}(s) \rangle_{\bV} ds + \int_r^{\tau} \! 2{_{\bV^*}\langle} \beta(s, X_s^{r,\xi;\theta},\theta(s)), X^{r,\xi;\theta}(s) \rangle_{\bV} ds
\\
&\le \|\xi\|_{0,\bH}^2 + \int_r^{\tau} \! c_1 \|X^{r,\xi;\theta}(s)\|_{\bH}^2 - c_2 \|X^{r,\xi;\theta}(s)\|_{\bV}^2 ds + 2 \int_r^{\tau} \! \|\beta(s,X_s^{r,\xi;\theta},\theta(s))\|_{\bH} \|X^{r,\xi;\theta}(s)\|_{\bH} ds.
\end{align*}
It follows that
\begin{align*}
&\| X_T^{r,\xi;\theta}(\tau) \|_{\bH}^2 + \int_r^{\tau} \! c_2 \cdot \|X^{r,\xi;\theta}(s)\|_{\bV}^2 ds
\\
\le &\|\xi\|_{0,\bH}^2 + \int_r^{\tau} \! c_1^+ \cdot \|X^{r,\xi;\theta}(s)\|_{\bH}^2 ds + L \int_r^{\tau} \! 2\|X^{r,\xi;\theta}(s)\|_{\bH} ds
\\
\le &\|\xi\|_{0,\bH}^2 + \int_r^{\tau} \! c_1^+ \cdot \|X^{r,\xi;\theta}(s)\|_{\bH}^2 ds + L \int_r^{\tau} \! \left(1 + \|X^{r,\xi;\theta}(s)\|_{\bH}^2\right) ds
\\
\le &\|\xi\|_{0,\bH}^2 + LT + (L + c_1^+)\int_r^{\tau} \! \max_{q\in [r,s]} \|X^{r,\xi;\theta}(q)\|_{\bH}^2 + \left(\int_r^s \! c_2 \left\| X^{r,\xi;\theta}(q) \right\|_{\bV}^2 dq\right) ds.
\end{align*}
Then by Gronwall's inequality, for some $K_1 = \max\{2, 2LT\}$, $K_1, K>0$, $K^2 = K_1 e^{2(L+c_1^+)T}$ we have
\begin{align*}
&\max_{l\in [r,T]} \| X^{r,\xi;\theta}(l) \|_{\bH}^2 + \int_r^{T} \! c_2 \cdot \|X^{r,\xi;\theta}(s)\|_{\bV}^2 ds
\\
\le &2\left[ \|\xi\|_{0,\bH}^2 + LT \right]e^{\int_r^{T} \! 2(L+c_1^+) ds}
\\
\le &K_1 e^{2(L+c_1^+)(T-r)} (1+\|\xi\|_{0,\bH}^2)
\\
\le &K_1 e^{2(L+c_1^+)T} (1+\|\xi\|_{0,\bH}^2)
\\
= &K^2(1 + \|\xi\|_{0,\bH}^2).
\end{align*}

For (iii), we use the similar approach as in (ii). By definition
\begin{equation*}
d_{0,\bV^*}(X_s^{r,\xi;\theta}, X_t^{r,\xi;\theta}) = |s-t|^{\frac{1}{2}} + \sup_{\tau\in [t,s]} \left\| X_s^{r,\xi;\theta}(\tau) - X_{t,s-t}^{r,\xi;\theta}(t) \right\|_{\bV^*}.
\end{equation*}
And we know that for $0\le r\le t\le \tau \le s\le T$,
\begin{align*}
X_t^{r,\xi;\theta}(t) = &X_t^{r,\xi;\theta}(r) + \int_r^t \! AX^{r,\xi;\theta}(p) + \beta(p,X_p^{r,\xi;\theta},\theta(p)) dp
\\
= &\xi(r) + \int_r^t \! AX^{r,\xi;\theta}(p) + \beta(p,X_p^{r,\xi;\theta},\theta(p)) dp,
\end{align*}
\begin{align*}
X_{\tau}^{r,\xi;\theta}(\tau) = &X_{\tau}^{r,\xi;\theta}(r) + \int_r^{\tau} \! AX^{r,\xi;\theta}(p) + \beta(p,X_p^{r,\xi;\theta},\theta(p)) dp
\\
= &\xi(r) + \int_r^{\tau} \! AX^{r,\xi;\theta}(p) + \beta(p,X_p^{r,\xi;\theta},\theta(p)) dp.
\end{align*}
It follows that
\begin{equation*}
X_s^{r,\xi;\theta}(\tau) - X_{t,s-t}^{r,\xi;\theta}(t) = X^{r,\xi;\theta}(\tau) - X^{r,\xi;\theta}(t) = \int_t^{\tau} \! AX^{r,\xi;\theta}(p) + \beta(p,X_p^{r,\xi;\theta},\theta(p)) dp.
\end{equation*}

Let $\tilde{K} = \left(2Tc^2L^2 + \frac{2c_3^2}{c_2}K^2\right)^{\frac{1}{2}}$. Then by H\"older's inequality and (ii), it holds that 
\begin{align*}
&\left\| X_s^{r,\xi;\theta}(\tau) - X_{t,s-t}^{r,\xi;\theta}(t) \right\|_{\bV^*}
\\
\le &\int_t^{\tau} \! \left\| AX^{r,\xi;\theta}(p) + \beta(p,X_p^{r,\xi;\theta},\theta(p)) \right\|_{\bV^*} dp
\\
\le &\left(\int_t^{\tau} \! \left\| AX^{r,\xi;\theta}(p) + \beta(p,X_p^{r,\xi;\theta},\theta(p)) \right\|_{\bV^*}^2 dp\right)^{\frac{1}{2}}\cdot |\tau - t|^{\frac{1}{2}}
\\
\le &\text{ }|s - t|^{\frac{1}{2}} \left(\int_t^{\tau} \! 2\left\| AX^{r,\xi;\theta}(p) \right\|_{\bV^*}^2 + 2\left\| \beta(p,X_p^{r,\xi;\theta},\theta(p)) \right\|_{\bV^*}^2 dp\right)^{\frac{1}{2}}
\\
\le &\text{ }|s - t|^{\frac{1}{2}} \left(\int_t^{\tau} \! 2c_3^2 \left\| X^{r,\xi;\theta}(p) \right\|_{\bV}^2 + 2c^2 \cdot L^2 \text{ }dp\right)^{\frac{1}{2}}
\\
\le &\text{ }|s - t|^{\frac{1}{2}} \left( 2c^2 L^2 |s-t| + \frac{2c_3^2}{c_2}K^2\left(1 + \|\xi\|_{0,\bH}^2\right) \right)^{\frac{1}{2}}
\\
\le &\text{ }|s - t|^{\frac{1}{2}} \left( 2Tc^2L^2 + \frac{2c_3^2}{c_2}K^2\right)^{\frac{1}{2}} \left(1 + \|\xi\|_{0,\bH}\right) 
\\
= &\text{ }\tilde{K}\left(1 + \|\xi\|_{0,\bH}\right) |s-t|^{\frac{1}{2}}.
\end{align*}
Letting $\overline{K} = 1+ \tilde{K}$, we have
\begin{align*}
d_{0,\bV^*} (X_s^{r,\xi;\theta}, X_t^{r,\xi;\theta}) \le &\tilde{K}\left(1 + \|\xi\|_{0,\bH}\right) |s-t|^{\frac{1}{2}} + |s-t|^{\frac{1}{2}}
\\
\le &\text{ }\overline{K}\left(1 + \|\xi\|_{0,\bH}\right) |s-t|^{\frac{1}{2}}.
\end{align*}
Here $\overline{K}$ depends on $c, c_2, c_3, T, L$, and is independent of the control process $\theta$.

For (iv), given another $\hat{\xi}\in L^0(\Omega,\sF_r; \Lambda_r(\bH))$, by similar approach, the It\^o's formula of square norms implies that
\begin{align*}
&\| X^{r,\xi;\theta}(l) - X^{r,\hat{\xi};\theta}(l) \|_{\bH}^2
\\
= &\|\xi(r) - \hat{\xi}(r)\|_{\bH}^2 + \int_r^l \! 2\langle A( X^{r,\xi;\theta}(s) - X^{r,\hat{\xi};\theta}(s) ), X^{r,\xi;\theta}(s) - X^{r,\hat{\xi};\theta}(s) \rangle ds
\\
&+ \int_r^l \! 2\langle \beta(s, X_s^{r,\xi;\theta},\theta(s)) - \beta(s, X_s^{r,\hat{\xi};\theta},\theta(s)), X^{r,\xi;\theta}(s) - X^{r,\hat{\xi};\theta}(s) \rangle ds
\\
\le &\|\xi - \hat{\xi}\|_{0,\bH}^2 + \int_r^l \! c_1\| X^{r,\xi;\theta}(s) - X^{r,\hat{\xi};\theta}(s) \|_{\bH}^2 - c_2\| X^{r,\xi;\theta}(s) - X^{r,\hat{\xi};\theta}(s) \|_{\bV}^2 ds
\\
&+ 2\int_r^l \! \|\beta(s, X_s^{r,\xi;\theta},\theta(s)) - \beta(s, X_s^{r,\hat{\xi};\theta},\theta(s))\|_{\bV^*} \cdot \|X^{r,\xi;\theta}(s) - X^{r,\hat{\xi};\theta}(s)\|_{\bV} ds
\\
\le &\|\xi - \hat{\xi}\|_{0,\bH}^2 + \int_r^l \! c_1\| X^{r,\xi;\theta}(s) - X^{r,\hat{\xi};\theta}(s) \|_{\bH}^2 - c_2\| X^{r,\xi;\theta}(s) - X^{r,\hat{\xi};\theta}(s) \|_{\bV}^2 ds
\\
&+ \frac{2}{c_2}\int_r^l \! \|\beta(s, X_s^{r,\xi;\theta},\theta(s)) - \beta(s, X_s^{r,\hat{\xi};\theta},\theta(s))\|_{\bV^*}^2 ds + \frac{c_2}{2} \int_r^l \!\|X^{r,\xi;\theta}(s) - X^{r,\hat{\xi};\theta}(s)\|_{\bV}^2 ds
\\
\le &\|\xi - \hat{\xi}\|_{0,\bH}^2 + \int_r^l \! c_1\| X^{r,\xi;\theta}(s) - X^{r,\hat{\xi};\theta}(s) \|_{\bH}^2 - \frac{c_2}{2}\| X^{r,\xi;\theta}(s) - X^{r,\hat{\xi};\theta}(s) \|_{\bV}^2 ds
\\
&+ \frac{2L^2}{c_2}\int_r^l \! \|X_s^{r,\xi;\theta} - X_s^{r,\hat{\xi};\theta}\|_{0,\bH}^2  ds.
\end{align*}
Then it holds that
\begin{align*}
&\| X^{r,\xi;\theta}(l) - X^{r,\hat{\xi};\theta}(l) \|_{\bH}^2 + \frac{c_2}{2} \int_r^l \! \| X^{r,\xi;\theta}(s) - X^{r,\hat{\xi};\theta}(s) \|_{\bV}^2 ds
\\
\le &\|\xi - \hat{\xi}\|_{0,\bH}^2 + \int_r^l \! \left(c_1+\frac{2L^2}{c_2}\right)\max_{\tau\in [0,s]} \| X^{r,\xi;\theta}(\tau) - X^{r,\hat{\xi};\theta}(\tau) \|_{\bH}^2 ds
\\
\le &\|\xi - \hat{\xi}\|_{0,\bH}^2 + \left(\frac{2L^2}{c_2} + c_1^+\right) \int_r^l \! \max_{\tau\in [0,s]} \| X^{r,\xi;\theta}(s) - X^{r,\hat{\xi};\theta}(s) \|_{\bH}^2 ds
\end{align*}
Setting $K = \sqrt{3}e^{\frac{3}{2}T(\frac{2L^2}{c_2} + c_1^+)}$, by Gr\"onwall's inequality we have
\begin{align*}
\max_{l\in [r,T]} \| X^{r,\xi;\theta}(l) - X^{r,\hat{\xi};\theta}(l) \|_{\bH}^2 + c_2 \int_r^T \! \| X^{r,\xi;\theta}(s) - X^{r,\hat{\xi};\theta}(s) \|_{\bV}^2 ds\le &K^2 \|\xi - \hat{\xi}\|_{0,\bH}^2.
\end{align*}
We obtain (iv). 

For (v), letting $0\le r\le l\le s\le T$, we have
\begin{equation*}
X^{r,\xi_r;\theta}(l) - \xi(r) = \int_r^l \! \left[ AX^{r,\xi_r;\theta}(\tau) + \beta(\tau,X_{\tau}^{r,\xi_r;\theta},\theta(\tau)) \right] d\tau.
\end{equation*}
Applying It\^o's formula gives
\begin{align*}
&\left\| X^{r,\xi_r;\theta}(l) - \xi(r) \right\|_{\bH}^2 
\\
= &\int_r^l \! 2{_{\bV^*}\langle} AX^{r,\xi_r;\theta}(\tau) + \beta(\tau,X_{\tau}^{r,\xi_e;\theta},\theta(\tau)), X^{r,\xi_r;\theta}(\tau) - \xi(r) \rangle_{\bV} d\tau
\\
\le &\int_r^l \! 2{_{\bV^*}\langle} A\left(X^{r,\xi_r;\theta}(\tau)-\xi(r)\right), X^{r,\xi_r;\theta}(\tau) - \xi(r) \rangle_{\bV} d\tau + \int_r^l \! 2{_{\bV^*}\langle} A\xi(r), X^{r,\xi_r;\theta}(\tau) - \xi(r) \rangle_{\bV} d\tau
\\
&+ 2L\int_r^l \! \| X^{r,\xi_r;\theta}(\tau) - \xi(r) \|_{\bH} d\tau
\\
\le &\int_r^l \! c_1\| X^{r,\xi_r;\theta}(\tau) - \xi(r) \|_{\bH}^2 d\tau - \int_r^l \! c_2\| X^{r,\xi_r;\theta}(\tau) - \xi(r) \|_{\bV}^2 d\tau 
\\
&+ 2\int_r^l \! \|A\xi(r)\|_{\bH} \cdot \|X^{r,\xi_r;\theta}(\tau) - \xi(r)\|_{\bH} d\tau + 2L\max_{\tau\in [r,l]} \left\| X^{r,\xi_r;\theta}(\tau) - \xi(r) \right\|_{\bH} \cdot |l-r|
\\
\le &\int_r^l \! c_1^+\| X^{r,\xi_r;\theta}(\tau) - \xi(r) \|_{\bH}^2 d\tau - \int_r^l \! c_2\| X^{r,\xi_r;\theta}(\tau) - \xi(r) \|_{\bV}^2 d\tau 
\\
&+ 2 \|A\xi(r)\|_{\bH}\cdot|l-r| \max_{\tau\in [r,l]} \|X^{r,\xi_r;\theta}(\tau) - \xi(r)\|_{\bH} + \frac{1}{4}\max_{\tau\in [r,l]} \|X^{r,\xi_r;\theta}(\tau) - \xi(r)\|_{\bH}^2 + 4L^2|l-r|^2
\\
\le &\int_r^l \! c_1^+\| X^{r,\xi_r;\theta}(\tau) - \xi(r) \|_{\bH}^2 d\tau - \int_r^l \! c_2\| X^{r,\xi_r;\theta}(\tau) - \xi(r) \|_{\bV}^2 d\tau + \frac{1}{2} \max_{\tau\in [r,l]} \left\| X^{r,\xi_r;\theta}(\tau) - \xi(r) \right\|_{\bH}^2 
\\
&+ |l-r|^2 \left( 4L^2 + 4\|A\xi(r)\|_{\bH}^2 \right) .
\end{align*}
Putting $\tilde{K}:= \max\left\{ 8L^2, 8 \right\}$, we have
\begin{align*}
&\frac{1}{2}\max_{\tau\in [r,l]}\left\| X^{r,\xi_r;\theta}(\tau) - \xi(r) \right\|_{\bH}^2 + c_2 \int_r^l \! \| X^{r,\xi_r;\theta}(\tau) - \xi(r) \|_{\bV}^2 d\tau
\\
\le &\int_r^l \! 2c_1^+\| X^{r,\xi_r;\theta}(\tau) - \xi(r) \|_{\bH}^2 d\tau + \tilde{K}(1+\|A\xi(r)\|_{\bH}^2) \cdot |l-r|^2.
\end{align*}
By Gr\"onwall's inequality, it holds that
\begin{align*}
\frac{1}{2}\max_{l\in [r,s]} \left\| X^{r,\xi_r;\theta}(l) - \xi(r) \right\|_{\bH}^2 + c_2 \int_r^s \! \| X^{r,\xi_r;\theta}(l) - \xi(r) \|_{\bV}^2 dl \le \tilde{K}(1+\|A\xi(r)\|_{\bH}^2) \cdot |s-r|^2 \cdot e^{2c_1^+T},
\end{align*}
from which we obtain (v).

Clearly the constants $K$ above only depends on $c_1^+$, $c_2$, $c_3$, $L$ and $T$. It is not dependent on the control process $\theta$.
\end{proof}

\subsection{Proof of Proposition \ref{prop3.2}}
\begin{proof}
We follow a similar approach as in \cite{qiu2017weak}. (v) follows directly from Assumption $(\cA 1)$ and (iv) of Lemma \ref{lemma3.1} since the coefficients $f(t, X_t, \theta)$ and $G(X_T)$ are uniformly Lipschitz continuous in $\Lambda_t(\bH)$ and in $\Lambda_T(\bH)$, respectively.

For each $(t,\xi,\overline{\theta})\in [0,T]\times L^0(\Omega,\sF_t;\Lambda_t(H)) \times \cU$, we denote
\begin{equation*}
\mathbb{J}(t,\xi,\overline{\theta}) := \left\{ J(t,\xi,\theta): \theta\in\cU, J(t,\xi,\theta)\le J(t,\xi,\overline{\theta}) \right\}.
\end{equation*}
Clearly this set is non-empty since it contains at least one element $J(t,\xi,\overline{\theta})$. And $J(t,\xi,\theta)$ is uniformly bounded since the coefficients $f$, $G$ are uniformly bounded under Assumption $(\cA 1)$. Then by \cite[Theorem A.3]{karatzas1998methods}, $V(t,\xi) = \essinf_{\theta\in\cU} J(t,\xi,\theta)$ exists.

For any $J(t,\xi,\tilde{\theta}), J(t,\xi,\hat{\theta})\in \mathbb{J}(t,\xi,\overline{\theta})$, we construct a control process
\begin{equation*}
\underline{\theta}(s) = \left[ \tilde{\theta}\cdot1_{ J(t,\xi,\tilde{\theta}) \le J(t,\xi,\hat{\theta}) } + \hat{\theta}\cdot 1_{ J(t,\xi,\tilde{\theta}) > J(t,\xi,\hat{\theta})} \right] \cdot 1_{[t,T]}(s) + \overline{\theta}(s) \cdot 1_{[0,t)}(s).
\end{equation*}
Clearly $\underline{\theta}$ takes values in $U$ and is $(\sF_t)_{t\in [0,T]}$-adapted. This shows that $\mathbb{J}(t,\xi,\overline{\theta})$ is closed under pairwise minimization for each $t\in [0,T]$, $\xi\in L^0(\Omega,\sF_t;\Lambda_t(\bH))$, $\overline{\theta}\in\cU$. Thus there exists a sequence of controls $\{\theta_n\}_{n\in\mathbb{Z}^+}$ such that $J(t,\xi,\theta_n)$ converges decreasingly to $V(t,\xi)$ with probability 1. Then (i) follows naturally.

For each $(\theta,x_0)\in \cU\times \bH$, and any $0\le t\le \tilde{t}\le T$, by dominated convergence theorem, we have
\begin{align*}
&E_{\sF_t} V\left(\tilde{t},X_{\tilde{t}}^{0,x_0;\theta}\right) + E_{\sF_t} \int_t^{\tilde{t}} \! f\left(s,X_s^{0,x_0;\theta},\theta(s)\right) ds
\\
= &\text{ }E_{\sF_t} \lim_{n\to\infty} J\left(\tilde{t},X_{\tilde{t}}^{0,x_0;\theta},\theta_n\right) + E_{\sF_t} \int_t^{\tilde{t}} \! f\left(s,X_s^{0,x_0;\theta},\theta(s)\right) ds
\\
= &\lim_{n\to\infty} E_{\sF_t} J\left(\tilde{t},X_{\tilde{t}}^{0,x_0;\theta},\theta_n\right) + E_{\sF_t} \int_t^{\tilde{t}} \! f\left(s,X_s^{0,x_0;\theta},\theta(s)\right) ds
\\
= &\lim_{n\to\infty} E_{\sF_t} \left[ \int_{\tilde{t}}^T \! f\left(s, X_s^{\tilde{t}, X_{\tilde{t}}^{0,x_0;\theta};\theta_n}, \theta_n(s)\right) ds + G\left(X_T^{\tilde{t},X_{\tilde{t}}^{0,x_0;\theta};\theta_n}\right) + \int_t^{\tilde{t}} \! f\left(s,X_s^{0,x_0;\theta},\theta(s)\right) ds \right]
\\
\ge &\text{ ess}\inf_{\theta^*\in\cU} E_{\sF_t} \left[ \int_t^T \! f\left(s,X_s^{t,X_t^{0,x_0;\theta};\theta^*},\theta^*(s)\right) ds + G\left(X_T^{t,X_{t}^{0,x_0;\theta};\theta^*}\right) \right]
\\
= &V\left(t,X_t^{0,x_0;\theta}\right),
\end{align*}
where $\theta^*\in\cU$ is the optimal control process. Thus (ii) is proved.

For (iii), with any $0\le t\le\tilde{t}\le T$, $\theta,\tilde{\theta}\in\cU$, $X_t^{0,x_0;\theta}\in\Lambda_t(\bH), X_{\tilde{t}}^{0,x_0;\theta}\in\Lambda_{\tilde{t}}(\bH)$, we have
\begin{align*}
&E_{\sF_t} V\left(\tilde{t}, X_{\tilde{t}}^{0,x_0;\tilde{\theta}}\right)
\\
= &E_{\sF_t} \left[ \essinf_{\theta\in\cU} E_{\sF_{\tilde{t}}} \left[ \int_{\tilde{t}}^T \! f\left(s,X_s^{\tilde{t},X_{\tilde{t}}^{0,x_0;\tilde{\theta}};\theta}, \theta(s)\right)ds + G\left(X_T^{\tilde{t},X_{\tilde{t}}^{0,x_0;\tilde{\theta}};\theta}\right) \right] \right]
\\
\le &\essinf_{\theta\in\cU} E_{\sF_t}  \left[ \int_{\tilde{t}}^T \! f\left(s,X_s^{\tilde{t},X_{\tilde{t}}^{0,x_0;\tilde{\theta}};\theta}, \theta(s)\right)ds + G\left(X_T^{\tilde{t},X_{\tilde{t}}^{0,x_0;\tilde{\theta}};\theta}\right) \right].
\end{align*}
Then by (ii), assumption $(\cA 1)$, and Lemma \ref{lemma3.1} (i), (iv), we have
\begin{align*}
&L|\tilde{t} - t|
\\
= &E_{\sF_t} \int_t^{\tilde{t}} \! L ds
\\
\ge &E_{\sF_t} \int_t^{\tilde{t}} \! f\left(s,X_s^{0,x_0;\tilde{\theta}},\theta(s)\right) ds
\\
\ge &V(t, X_t^{0,x_0;\tilde{\theta}}) - E_{\sF_t} V\left(\tilde{t}, X_{\tilde{t}}^{0,x_0;\tilde{\theta}}\right)
\\
\ge &\essinf_{\theta\in\cU}E_{\sF_t} \left[ \int_t^T \! f\left(s,X_s^{t,X_{t}^{0,x_0;\tilde{\theta}},\theta}, \theta(s)\right)ds + G\left(X_T^{t,X_{t}^{0,x_0;\tilde{\theta}},\theta}\right) \right]
\\
&- \essinf_{\theta\in\cU}E_{\sF_t} \left[ \int_{\tilde{t}}^T \! f\left(s,X_s^{\tilde{t},X_{\tilde{t}}^{0,x_0;\tilde{\theta}};\theta}, \theta(s)\right)ds + G\left(X_T^{\tilde{t},X_{\tilde{t}}^{0,x_0;\tilde{\theta}};\theta}\right) \right]
\\
\ge &\essinf_{\theta\in\cU}E_{\sF_t} \Bigg[  \int_{\tilde{t}}^T \! f\left(s,X_s^{t,X_{t}^{0,x_0;\tilde{\theta}};\theta}, \theta(s)\right) - f\left(s,X_s^{\tilde{t},X_{\tilde{t}}^{0,x_0;\tilde{\theta}};\theta}, \theta(s)\right)ds
\\
&\text{ }\text{ }\text{ }\text{ }\text{ }\text{ }\text{ }\text{ }\text{ }\text{ }\text{ }\text{ }\text{ }\text{ }\text{ }+ \int_t^{\tilde{t}} \! f\left(s,X_s^{t,X_{t}^{0,x_0;\tilde{\theta}};\theta}, \theta(s)\right)ds + G\left(X_T^{t,X_{t}^{0,x_0;\tilde{\theta}};\theta}\right) - G\left(X_T^{\tilde{t},X_{\tilde{t}}^{0,x_0;\tilde{\theta}};\theta}\right) \Bigg]
\\
\ge &\essinf_{\theta\in\cU}E_{\sF_t} \Bigg[ \int_t^{\tilde{t}} \! -L ds + \int_{\tilde{t}}^T \! (-L) \left\| X_s^{t, X_t^{0,x_0;\tilde{\theta}};\theta} - X_s^{\tilde{t}, X_{\tilde{t}}^{0,x_0;\tilde{\theta}};\theta} \right\|_{0,\bH} ds
\\
&\text{ }\text{ }\text{ }\text{ }\text{ }\text{ }\text{ }\text{ }\text{ }\text{ }\text{ }\text{ }\text{ }\text{ }\text{ } - L \left\| X_T^{t, X_t^{0,x_0;\tilde{\theta}};\theta} - X_T^{\tilde{t}, X_{\tilde{t}}^{0,x_0;\tilde{\theta}};\theta} \right\|_{0,\bH} \Bigg]
\\
\ge &- \esssup_{\theta\in\cU} \left[ L|\tilde{t} - t| + (L + L |T - \tilde{t}|)\left\| X_T^{t, X_t^{0,x_0;\tilde{\theta}};\theta} - X_T^{\tilde{t}, X_{\tilde{t}}^{0,x_0;\tilde{\theta}};\theta} \right\|_{0,\bH}  \right]
\\
= &- \esssup_{\theta\in\cU} \left[ L|\tilde{t} - t| + (L + L |T - \tilde{t}|)\left\| X_T^{\tilde{t}, X_{\tilde{t}}^{t,X_t^{0,x_0;\tilde{\theta}};\theta}; \theta} - X_T^{\tilde{t}, X_{\tilde{t}}^{t,X_t^{0,x_0;\tilde{\theta}};\tilde{\theta}}; \theta} \right\|_{0,\bH}  \right]
\\
\ge &- \esssup_{\theta\in\cU} \left[ L|\tilde{t} - t| + K(L + L |T - \tilde{t}|)\left\| X_{\tilde{t}}^{t,X_t^{0,x_0;\tilde{\theta}};\theta} - X_{\tilde{t}}^{t,X_t^{0,x_0;\tilde{\theta}};\tilde{\theta}} \right\|_{0,\bH}  \right].
\\
\end{align*}
Let $\xi = X_t^{0,x_0;\tilde{\theta}}$. Then for any $0\le t\le s\le \tilde{t}$, by Ito's formula and H\"older's inequality,
\begin{align*}
&\left\| X_{\tilde{t}}^{t,\xi;\theta}(s) - X_{\tilde{t}}^{t,\xi;\tilde{\theta}}(s) \right\|_{\bH}^2
\\
= &2\int_t^s \! \left\langle A\left( X_{\tilde{t}}^{t,\xi;\theta}(r) - X_{\tilde{t}}^{t,\xi;\tilde{\theta}}(r) \right), X_{\tilde{t}}^{t,\xi;\theta}(r) - X_{\tilde{t}}^{t,\xi;\tilde{\theta}}(r) \right\rangle dr
\\
& + 2\int_t^s \! \left\langle \beta\left( r,X_r^{t,\xi;\theta}, \theta(r) \right) - \beta\left( r,X_r^{t,\xi;\tilde{\theta}}, \tilde{\theta}(r) \right), X_{\tilde{t}}^{t,\xi;\theta}(r) - X_{\tilde{t}}^{t,\xi;\tilde{\theta}}(r) \right\rangle dr
\\
\le &c_1 \int_t^s \! \left\| X_{\tilde{t}}^{t,\xi;\theta}(r) - X_{\tilde{t}}^{t,\xi;\tilde{\theta}}(r) \right\|_{\bH}^2 dr - c_2 \int_t^s \! \left\| X_{\tilde{t}}^{t,\xi;\theta}(r) - X_{\tilde{t}}^{t,\xi;\tilde{\theta}}(r) \right\|_{\bV}^2 dr
\\
& + 4L \int_t^s \! \left\| X_{\tilde{t}}^{t,\xi;\theta}(r) - X_{\tilde{t}}^{t,\xi;\tilde{\theta}}(r) \right\|_{\bH} dr
\\
\le &2L|s-t| + (c_1^+ + 2L) \int_t^s \! \left\| X_{\tilde{t}}^{t,\xi;\theta}(r) - X_{\tilde{t}}^{t,\xi;\tilde{\theta}}(r) \right\|_{\bH}^2 dr - c_2 \int_t^s \! \left\| X_{\tilde{t}}^{t,\xi;\theta}(r) - X_{\tilde{t}}^{t,\xi;\tilde{\theta}}(r) \right\|_{\bV}^2 dr
\end{align*}
Thus denoting $\overline{K} = 4L \cdot e^{2T(c_1^+ + 2L)}$, by Gr\"onwall's inequality, we have
\begin{align*}
&\max_{s\in [t,\tilde{t}]}\left\| X_{\tilde{t}}^{t,\xi;\theta}(s) - X_{\tilde{t}}^{t,\xi;\tilde{\theta}}(s) \right\|_{\bH}^2 + c_2 \int_t^{\tilde{t}} \! \left\| X_{\tilde{t}}^{t,\xi;\theta}(r) - X_{\tilde{t}}^{t,\xi;\tilde{\theta}}(r) \right\|_{\bV}^2 dr
\\
\le& 4L|\tilde{t} - t|\cdot e^{2\left(c_1^+ + 2L \right)T}
\\
\le&\overline{K} |\tilde{t}-t|.
\end{align*}
Letting $\tilde{K} = K\cdot \overline{K}^{\frac{1}{2}}$, we have
\begin{align*}
&L| \tilde{t} - t |
\\
\ge &V(t, X_t^{0,x_0;\tilde{\theta}}) - E_{\sF_t} V\left(\tilde{t}, X_{\tilde{t}}^{0,x_0;\tilde{\theta}}\right)
\\
\ge &- \esssup_{\theta\in\cU} \left[ L|\tilde{t} - t| + K(L + L |T - \tilde{t}|)\left\| X_{\tilde{t}}^{t,X_t^{0,x_0;\tilde{\theta}};\theta} - X_{\tilde{t}}^{t,X_t^{0,x_0;\tilde{\theta}};\tilde{\theta}} \right\|_{0,\bH}  \right]
\\
\ge &- \esssup_{\theta\in\cU} \left[ L|\tilde{t} - t| + \tilde{K}(L + L |T - \tilde{t}|)|\tilde{t} - t|^{\frac{1}{2}}  \right]
\\
\to &\text{ }0, \text{ }\text{ as } | \tilde{t} - t |\to 0.
\end{align*}

Thus by taking expectation on each side of the above inequality, we obtain the time-continuity of $E \left[V\left(t,X_t^{0,x_0;\theta}\right)\right]$. This result, along with the regularity of super-martingale, implies the right continuity of $V(t,X_t^{0,x_0;\theta})$. And by the backward stochastic differential equation theory, $E_{\sF_t} V\left(\tilde{t}, X_{\tilde{t}}^{0,x_0;\theta}\right)$ is continuous in $t\in [0,\tilde{t}]$. Combined with the obvious fact that $|\tilde{t} - t|$ is continuous in $t$, the left time-continuity of $V(t,X_t^{0,x_0;\theta})$ follows.

For (iv), the joint continuity comes naturally by (iii) and (v). For any $(t,X_t)\in [0,T]\times\Lambda_t(\bH)$,
\begin{align*}
|V(t,X_t)| = &\left| \essinf_{\theta\in\cU} E_{\sF_t} \left[ \int_t^T \! f(s,X_s^{t,X_t;\theta},\theta(s)) ds + G(X_T^{t,X_t;\theta}) \right] \right|
\\
\le &\left| \esssup_{\theta\in\cU} E_{\sF_t} \left[ \int_t^T \! f(s,X_s^{t,X_t;\theta},\theta(s)) ds + G(X_T^{t,X_t;\theta}) \right] \right|
\\
\le &\esssup_{\theta\in\cU} E_{\sF_t} \left[ \int_t^T \! \left|f(s,X_s^{t,X_t;\theta},\theta(s))\right| ds + \left|G(X_T^{t,X_t;\theta})\right| \right]
\\
\le &L(T-t) + L
\\
\le &L(T+1).
\end{align*}
Analogously, we can get $|J(t,X_t,\theta)| \le L(T+1)$ for any $\theta\in\cU$ as well.
\end{proof}

\subsection{Proof of Lemma \ref{lemma3.4}}
\begin{proof}
W.l.o.g., we only need to prove that \eqref{lemma3.4eq} holds for any $\tau\in (0, \underline{t}_1)$, $\rho=0$ and $x_0=x\in \bV$. Then for each $N\in\mathbb{N}^{+}$, $N>2$, let $t_i=\frac{i\tau}{N}$ with $i=0, ..., N$. In this way we obtain a partition of $[0, \tau]$ with $0=t_0<t_1<...<t_N=\tau$. And notice that here we have two partitions: partition of $[0, T]$ with $\underline{t}_j$, $j=0,1,\cdots,n$, and partition of $[0, \tau]$ with $t_i$, $i=0,1,\cdots,N$.
For each $\theta\in\mathcal{U}$, let
\begin{equation*}
^{N}X(t) = \sum_{i=0}^{N-1} X^{0, x; \theta}(t_i)1_{[t_i, t_{i+1})}(t)+X^{0, x; \theta}(\tau)1_{\{ \tau \}}(t), t\in [0, \tau]
\end{equation*}
\begin{equation*}
^{N}X_{t^{-}}(s)={^{N}X}(s)1_{[0, t)}(s)+\lim_{r\to t^{-}} {^{N}X}(r)1_{\{ t \}}(s), 0\le s\le t\le r .
\end{equation*}
Easily for any $t\in [0,\tau]$, we have the approximation hold as following:
\begin{equation*}
\lim_{N\to\infty} \left(\left\| X_t^{0, x; \theta} - {^{N}X}_t \right\|_{0,\bV^*} + \left\| X_t^{0, x; \theta} - {^{N}X}_{t^{-}} \right\|_{0,\bV^*}\right)=0.
\end{equation*}
 For each $i\in {0, ..., N-1}$, we have
 \begin{align*}
&\text{ }\text{ }\text{ }\text{ } u(t_{i+1}, {^{N}X}_{t_{i+1}})-u(t_i, {^{N}X}_{t_i})
\\
&=u\left(t_{i+1}, {^{N}X}_{t_{{i+1}^{-}}}\right)-u\left(t_i, {^{N}X}_{t_i}\right)+ u\left(t_{i+1}, {^{N}X}_{t_{i+1}}\right)-u\left(t_{i+1}, {^{N}X}_{t_{{i+1}^{-}}}\right) 
\\
:&= I_1^i + I_2^i,
 \end{align*}
 where by $u\in\mathcal{C}^1_{\sF}$, and Definition \ref{cf1} $(i)$,
 \begin{equation*}
 I_1^i = u\left(t_{i+1}, {^{N}X}_{t_{{i+1}^{-}}}\right)-u\left(t_i, {^{N}X}_{t_i}\right) =\int_{t_i}^{t_{i+1}} \! \mathfrak{d}_s u\left(s, {^{N}X}_{s^{-}}\right)ds + \int_{t_i}^{t_{i+1}} \! \mathfrak{d}_{\omega} u\left(r, {^{N}X}_{r^{-}}\right)dW(r),
 \end{equation*}
 \begin{equation*}
 I_2^i = u\left(t_{i+1}, {^{N}X}_{t_{i+1}}\right)-u\left(t_{i+1}, {^{N}X}_{t_{{i+1}^{-}}}\right).
 \end{equation*}
 For the first term of $I_2^i$, by the definition of vertical perturbation, we have
 \begin{equation*}
 u\left(t_{i+1}, {^{N}X}_{t_{i+1}}\right) = u\left(t_{i+1}, {^{N}X}_{t_{{i+1}^{-}}}^{ ^{N}X(t_{i+1}) - {^{N}X}(t_{{i+1}^{-}}) }\right)  = u\left(t_{i+1}, {^{N}X}_{t_{{i+1}^{-}}}^{ {X}(t_{i+1}) - {X}(t_i) }\right).
  \end{equation*}
Thus
\begin{equation*}
I_2^i = u\left(t_{i+1}, {X}_{t_{{i+1}^{-}}}^{ {X}(t_{i+1}) - {X}(t_i) }\right) - u\left(t_{i+1}, {^{N}X}_{t_{{i+1}^{-}}}\right).
\end{equation*}
Then by the definition of vertical derivative and integration by parts formula,
\begin{equation*}
I_2^i = \int_{t_i}^{t_{i+1}} \!  \nabla u\left(t_{i+1}, {^{N}X}_{t_{{i+1}^{-}}}^{ {X}(s) - {X}(t_i) }\right)\left[ AX^{0,x;\theta}(s) + \beta\left(s, X_s^{0, x; \theta}, \theta(s)\right) \right]ds.
\end{equation*}
Recall that
\begin{equation*}
u\left(\tau, {^{N}X}_{\tau}\right) - u\left(0, x\right) = \sum_{i=0}^{N-1} (I_1^i + I_2^i).
\end{equation*}
Letting $N\to\infty$, by the dominated convergence theorem we complete the proof.

\end{proof}

\bibliographystyle{siam}

\bibliography{ref_qjn}

\begin{thebibliography}{10}

\bibitem{bayraktar2018path}
{\sc E.~Bayraktar and C.~Keller}, {\em Path-dependent {Hamilton--Jacobi}
  equations in infinite dimensions}, J. Funct. Anal., 275 (2018),
  pp.~2096--2161.

\bibitem{Bayraktar-Qiu_2017}
{\sc E.~Bayraktar and J.~Qiu}, {\em Controlled reflected {SDEs} and {Neumann}
  problem for backward {SPDEs}}, Ann. Appl. Probab., 29 (2019), pp.~2819--2848.

\bibitem{bender2016first}
{\sc C.~Bender and N.~Dokuchaev}, {\em A first-order {BSPDE} for swing option
  pricing}, Mathematical Finance, 26 (2016), pp.~461--491.

\bibitem{briand2003lp}
{\sc P.~Briand, B.~Delyon, Y.~Hu, E.~Pardoux, and L.~Stoica}, {\em ${L}^p$
  solutions of backward stochastic differential equations}, Stochastic
  Processes and their Applications, 108 (2003), pp.~604--618.

\bibitem{buckdahn2015pathwise}
{\sc R.~Buckdahn, C.~Keller, J.~Ma, and J.~Zhang}, {\em Pathwise viscosity
  solutions of stochastic {PDE}s and forward path-dependent {PDE}s—a rough
  path view}, arXiv preprint arXiv:1501.06978,  (2015).

\bibitem{cardaliaguet2019master}
{\sc P.~Cardaliaguet, F.~Delarue, J.-M. Lasry, and P.-L. Lions}, {\em The
  master equation and the convergence problem in mean field games:({AMS}-201)},
  Princeton University Press, 2019.

\bibitem{cont2013functional}
{\sc R.~Cont, D.-A. Fourni{\'e}, et~al.}, {\em Functional {I}t{\^o} calculus
  and stochastic integral representation of martingales}, The Annals of
  Probability, 41 (2013), pp.~109--133.

\bibitem{cosso2018path}
{\sc A.~Cosso, S.~Federico, F.~Gozzi, M.~Rosestolato, and N.~Touzi}, {\em
  Path-dependent equations and viscosity solutions in infinite dimension}, Ann.
  Probab., 46 (2018), pp.~126--174.

\bibitem{crandall1984some}
{\sc M.~G. Crandall, L.~C. Evans, and P.-L. Lions}, {\em Some properties of
  viscosity solutions of {H}amilton-{J}acobi equations}, Transactions of the
  American Mathematical Society, 282 (1984), pp.~487--502.

\bibitem{crandall1983viscosity}
{\sc M.~G. Crandall and P.-L. Lions}, {\em Viscosity solutions of
  {H}amilton-{J}acobi equations}, Transactions of the American mathematical
  society, 277 (1983), pp.~1--42.

\bibitem{crandall1985hamilton}
\leavevmode\vrule height 2pt depth -1.6pt width 23pt, {\em Hamilton-{}acobi
  equations in infinite dimensions i. {U}niqueness of viscosity solutions},
  Journal of functional analysis, 62 (1985), pp.~379--396.

\bibitem{da2014stochastic}
{\sc G.~Da~Prato and J.~Zabczyk}, {\em Stochastic equations in infinite
  dimensions}, Cambridge university press, 2014.

\bibitem{ekren2016viscosity1}
{\sc I.~Ekren, N.~Touzi, J.~Zhang, et~al.}, {\em Viscosity solutions of fully
  nonlinear parabolic path dependent {PDE}s: {P}art {I}}, Annals of
  Probability, 44 (2016), pp.~1212--1253.

\bibitem{ekren2016viscosity}
\leavevmode\vrule height 2pt depth -1.6pt width 23pt, {\em Viscosity solutions
  of fully nonlinear parabolic path dependent {PDE}s: {P}art {II}}, Annals of
  Probability, 44 (2016), pp.~2507--2553.

\bibitem{ekren2016pseudo}
{\sc I.~Ekren and J.~Zhang}, {\em Pseudo-{M}arkovian viscosity solutions of
  fully nonlinear degenerate {PPDE}s}, Probability, Uncertainty and
  Quantitative Risk, 1 (2016), pp.~1--34.

\bibitem{fabbri2017stochastic}
{\sc G.~Fabbri, F.~Gozzi, and A.~Swiech}, {\em Stochastic optimal control in
  infinite dimension}, Probability and Stochastic Modelling. Springer,  (2017).

\bibitem{hu2002semi}
{\sc Y.~Hu, J.~Ma, and J.~Yong}, {\em On semi-linear degenerate backward
  stochastic partial differential equations}, Probability Theory and Related
  Fields, 123 (2002), pp.~381--411.

\bibitem{ishii1990viscosity}
{\sc H.~Ishii and P.-L. Lions}, {\em Viscosity solutions of fully nonlinear
  second-order elliptic partial differential equations}, Journal of
  Differential equations, 83 (1990), pp.~26--78.

\bibitem{jensen1988uniqueness}
{\sc R.~Jensen, P.-L. Lions, and P.~E. Souganidis}, {\em A uniqueness result
  for viscosity solutions of second order fully nonlinear partial differential
  equations}, Proceedings of the American mathematical society, 102 (1988),
  pp.~975--978.

\bibitem{karatzas1998methods}
{\sc I.~Karatzas, S.~E. Shreve, I.~Karatzas, and S.~E. Shreve}, {\em Methods of
  mathematical finance}, vol.~39, Springer, 1998.

\bibitem{krylov1987nonlinear}
{\sc N.~V. Krylov}, {\em Nonlinear elliptic and parabolic equations of the
  second order}, vol.~7, Springer, 1987.

\bibitem{kunita1981some}
{\sc H.~Kunita}, {\em Some extensions of {I}to's formula}, in S{\'e}minaire de
  Probabilit{\'e}s XV 1979/80, Springer, 1981, pp.~118--141.

\bibitem{leao2018weak}
{\sc D.~Le{\~a}o, A.~Ohashi, A.~B. Simas, et~al.}, {\em A weak version of
  path-dependent functional {I}t{\^o} calculus}, Annals of Probability, 46
  (2018), pp.~3399--3441.

\bibitem{lions1983optimal}
{\sc P.-L. Lions}, {\em Optimal control of diffusion processes and
  {H}amilton--{J}acobi--{B}ellman equations part 2: viscosity solutions and
  uniqueness}, Communications in partial differential equations, 8 (1983),
  pp.~1229--1276.

\bibitem{lions1988viscosity}
\leavevmode\vrule height 2pt depth -1.6pt width 23pt, {\em Viscosity solutions
  of fully nonlinear second-order equations and optimal stochastic control in
  infinite dimensions. {P}art {I}: The case of bounded stochastic evolutions},
  (1988).

\bibitem{lukoyanov2007viscosity}
{\sc N.~Y. Lukoyanov}, {\em On viscosity solution of functional
  {H}amilton--{J}acobi type equations for hereditary systems}, Proceedings of
  the Steklov Institute of Mathematics, 259 (2007), pp.~S190--S200.

\bibitem{mete1988hamilton}
{\sc H.~Mete~Soner}, {\em On the {H}amilton--{J}acobi--{B}ellman equations in
  {B}anach spaces}, Journal of optimization theory and applications, 57 (1988),
  pp.~429--437.

\bibitem{oksendal2003stochastic}
{\sc B.~{\O}ksendal}, {\em Stochastic differential equations}, Springer, 2003.

\bibitem{peng1992stochastic}
{\sc S.~Peng}, {\em Stochastic {H}amilton--{J}acobi--{B}ellman equations}, SIAM
  Journal on Control and Optimization, 30 (1992), pp.~284--304.

\bibitem{peng2016bsde}
{\sc S.~Peng and F.~Wang}, {\em {BSDE}, {P}ath-dependent {PDE} and nonlinear
  {F}eynman-{K}ac formula}, Science China Mathematics, 59 (2016), pp.~19--36.

\bibitem{prevot2007concise}
{\sc C.~Pr{\'e}v{\^o}t and M.~R{\"o}ckner}, {\em A concise course on stochastic
  partial differential equations}, vol.~1905, Springer, 2007.

\bibitem{qiu2017weak}
{\sc J.~Qiu}, {\em Weak solution for a class of fully nonlinear stochastic
  {H}amilton--{J}acobi--{B}ellman equations}, Stochastic Processes and their
  Applications, 127 (2017), pp.~1926--1959.

\bibitem{qiu2018viscosity}
{\sc J.~Qiu}, {\em Viscosity solutions of stochastic
  {H}amilton--{J}acobi--{B}ellman equations}, SIAM Journal on Control and
  Optimization, 56 (2018), pp.~3708--3730.

\bibitem{qiu2022controlled}
{\sc J.~Qiu}, {\em Controlled ordinary differential equations with random
  path-dependent coefficients and stochastic path-dependent
  {H}amilton--{J}acobi equations}, Stochastic Processes and their Applications,
  154 (2022), pp.~1--25.

\bibitem{qiu2019uniqueness}
{\sc J.~Qiu and W.~Wei}, {\em {U}niqueness of viscosity solutions of stochastic
  {H}amilton-{J}acobi equations}, Acta Mathematica Scientia, 39 (2019),
  pp.~857--873.

\bibitem{ren2017comparison}
{\sc Z.~Ren, N.~Touzi, and J.~Zhang}, {\em Comparison of viscosity solutions of
  fully nonlinear degenerate parabolic path-dependent {PDEs}}, SIAM J. Math.
  Anal., 49 (2017), pp.~4093--4116.

\end{thebibliography}

\end{document}